\documentclass{amsart}

\usepackage{amssymb}
\usepackage{amsmath}
\usepackage{graphicx}
\usepackage{epic,eepic} % for Figure of H_X

\newcommand{\erase}[1]{}

\theoremstyle{plain}
\newtheorem{theorem}{Theorem}[section]
\newtheorem{lemma}[theorem]{Lemma}
\newtheorem{proposition}[theorem]{Proposition}
\newtheorem{corollary}[theorem]{Corollary}

\theoremstyle{definition}
\newtheorem{definition}[theorem]{Definition}
\newtheorem{example}[theorem]{Example}
\newtheorem{algorithm}[theorem]{Algorithm}

\newtheorem{criterion}[theorem]{Criterion}
\newtheorem{claim}[theorem]{Claim}

\theoremstyle{remark}
\newtheorem{remark}[theorem]{Remark}

\numberwithin{equation}{section}
\numberwithin{table}{section}
\numberwithin{figure}{section}

\newcommand{\step}[1]{{\it Step #1}}

\newcommand{\C}{\mathord{\mathbb C}}
\newcommand{\D}{\mathord{\mathbb D}}
\newcommand{\F}{\mathord{\mathbb F}}
\renewcommand{\H}{\mathord{\mathbb H}}
\renewcommand{\P}{\mathord{\mathbb  P}}
\newcommand{\Q}{\mathord{\mathbb  Q}}
\newcommand{\R}{\mathord{\mathbb R}}
\newcommand{\Z}{\mathord{\mathbb Z}}

\newcommand{\BBB}{\mathord{\mathcal B}}

\newcommand{\DDD}{\mathord{\mathcal D}}
\newcommand{\FFF}{\mathord{\mathcal F}}
\newcommand{\HHH}{\mathord{\mathcal H}}

\newcommand{\NNN}{\mathord{\mathcal N}}
\newcommand{\PPP}{\mathord{\mathcal P}}
\newcommand{\RRR}{\mathord{\mathcal R}}
\newcommand{\VVV}{\mathord{\mathcal V}}

\newcommand{\inj}{\hookrightarrow}
\newcommand{\isom}{\mathbin{\,\raise -.6pt\rlap{$\to$}\raise 3.5pt \hbox{\hskip .3pt$\mathord{\sim}$}\;}}
\newcommand{\set}[2]{\{\; {#1} \; \mid \; {#2} \;  \}}
\newcommand{\shortset}[2]{\{ {#1} \,|\, {#2}   \}}
\newcommand{\gen}[1]{\langle {#1}  \rangle}

\newcommand{\tensor}{\otimes}

\newcommand{\inv}{\sp{-1}}

\newcommand{\sprime}{\sp\prime}

\newcommand{\spar}[1]{\sp{(#1)}}
\newcommand{\spprime}{\sp{\prime\prime}}
\newcommand{\spcirc}{\sp{\mathord{\circ}}}
\newcommand{\sptimes}{\sp{\times}}
\newcommand{\sperp}{\sp{\perp}}
\newcommand{\dual}{\sp{\vee}}

\newcommand{\semidirectproduct}{\rtimes}

\newcommand{\bdr}{\partial\,}

\newcommand{\Hom}{\mathord{\mathrm {Hom}}}
\newcommand{\GL}{\mathord{\mathrm {GL}}}
\newcommand{\PGL}{\mathord{\mathrm {PGL}}}
\newcommand{\OG}{\mathord{\mathrm {O}}}
\newcommand{\SL}{\mathord{\mathrm{SL}}}

\newcommand{\id}{\mathord{\mathrm {id}}}
\newcommand{\Ker}{\operatorname{\mathrm {Ker}}\nolimits}
\newcommand{\Image}{\operatorname{\mathrm {Im}}\nolimits}
\newcommand{\Aut}{\operatorname{\mathrm {Aut}}\nolimits}

\newcommand{\pr}{\mathord{\mathrm {pr}}}

\newcommand{\rank}{\operatorname{\mathrm {rank}}\nolimits}
\newcommand{\disc}{\operatorname{\mathrm {disc}}\nolimits}

\newcommand{\closure}[1]{\overline{#1}}

\newcommand{\rmand}{\textrm{and}}
\newcommand{\quand}{\quad\rmand\quad}

\renewcommand{\L}{\mathord{\mathbf{L}}}
\newcommand{\MW}{\mathord{\mathrm{MW}}}
\newcommand{\Nef}{\mathord{\mathrm{Nef}}}
\newcommand{\intM}[2]{\langle{#1}\rangle_{#2}}
\newcommand{\aut}{\operatorname{\mathit {Aut}}\nolimits} 
\newcommand{\autG}{\aut_G}

\newcommand{\Chams}{\D}
\newcommand{\Roots}{\RRR}
\newcommand{\sphyp}{\sp*}
\newcommand{\procadj}{\mathord{\mathtt{adj}}}
\newcommand{\discgr}[1]{A_{#1}}
\newcommand{\clPPP}{\overline{\PPP}}
\newcommand{\clPPPQ}{\overline{\PPP}^{\,\Q}}
\newcommand{\clH}{\overline{\H}}
\newcommand{\clHQ}{\overline{\H}^{\,\Q}}
\newcommand{\clD}{\overline{D}}
\newcommand{\clDDD}{\overline{\DDD}}
\newcommand{\bdrclH}{\bdr\clH}
\newcommand{\bdrclHQ}{\bdr\clH^{\,\Q}}
\newcommand{\lineseg}[1]{\overline{#1}}
\newcommand{\RtsLS}{\Roots_{\L|S}}
\newcommand{\opp}{\mathord{{\rm opp}}}
\newcommand{\HS}{\mathord{H\hskip -1pt S}}
\newcommand{\HB}{\mathord{H\hskip -1pt B}}
\newcommand{\bdrHB}{\mathord{\bdr \hskip -1.2pt H\hskip -1.2pt B}}

\newcommand{\DR}{\mathord{{\rm DR}}}

\newcommand{\Nc}{N}

\newcommand{\propVG}{\mathord{\rm [G]}}
\newcommand{\propVone}{\mathord{\rm [V1]}}
\newcommand{\propVtwo}{\mathord{\rm [V2]}}
\newcommand{\propVthree}{\mathord{\rm [V3]}}
\newcommand{\propVfour}{\mathord{\rm [V4]}}
\newcommand{\assumpone}{\mathord{\rm [SG1]}}
\newcommand{\assumptwo}{\mathord{\rm [SG2]}}
\newcommand{\assumpthree}{\mathord{\rm [SG3]}}

\begin{document}

\title[Automorphism groups of $K3$ surfaces]
{An algorithm to compute  automorphism groups of $K3$ surfaces
and an application to singular $K3$ surfaces}

\author{Ichiro Shimada}
\address{
Department of Mathematics, 
Graduate School of Science, 
Hiroshima University,
1-3-1 Kagamiyama, 
Higashi-Hiroshima, 
739-8526 JAPAN
}
\email{shimada@math.sci.hiroshima-u.ac.jp}

\thanks{Partially supported by
JSPS Grant-in-Aid for Challenging Exploratory Research No.23654012
and 
JSPS Grants-in-Aid for Scientific Research (C) No.25400042 
}

\begin{abstract}
Let $X$ be a  complex algebraic  $K3$ surface or a supersingular $K3$ surface in odd characteristic.
We present an algorithm by which,
under certain assumptions on $X$,
we can calculate a finite set of generators 
of the image of the natural homomorphism from the automorphism group of  $X$ to 
the orthogonal group of the  N\'eron-Severi lattice of $X$.
We then apply this algorithm to certain complex $K3$ surfaces,
among which are  singular $K3$ surfaces whose transcendental lattices are of small discriminants.
\end{abstract}

%\subjclass[2000]{14J28, 14J50, 14Q10}
%14J   (1973-now) Surfaces and higher-dimensional varieties [For analytic theory, see 32Jxx]
%14J28   (1980-now) K3 surfaces and Enriques surfaces
%14J   (1973-now) Surfaces and higher-dimensional varieties [For analytic theory, see 32Jxx]
%14J50   (1980-now) Automorphisms of surfaces and higher-dimensional varieties
%14Q   (1991-now) Computational aspects in algebraic geometry [See also 12Y05, 13Pxx, 68W30]
%14Q10   (1991-now) Surfaces, hypersurfaces

\maketitle

% main text starts here

\section{Introduction}\label{sec:Intro}
The automorphism group $\Aut(X)$ of an algebraic  $K3$ surface $X$ is 
an important and interesting object.
Suppose that $X$ is defined over the complex number field $\C$,
or is supersingular in odd characteristic.
Then, thanks to  the Torelli-type theorem
due to  Piatetski-Shapiro and Shafarevich~\cite{MR0284440} 
and Ogus~\cite{MR563467},~\cite{MR717616},
we can study $\Aut(X)$ by the N\'eron-Severi lattice $S_X$ of $X$.
We denote by  $\OG(S_X)$ the orthogonal group of $S_X$.
Then we have a natural homomorphism
\begin{equation*}\label{eq:naturalhom}
\varphi_X\colon  \Aut(X) \to \OG(S_X).
%\varphi_X:  \Aut(X) \to \OG(S_X).
\end{equation*}
It is known that this homomorphism has only  a finite kernel.
Using the reduction theory for arithmetic subgroups of $\OG(S_X)$, 
Sterk~\cite{MR786280} and Lieblich and Maulik~\cite{arXiv11023377} 
proved that $\Aut(X)$ is finitely generated.
The N\'eron-Severi lattices for which  $\Aut(X)$ are finite were classified by
Nikulin~\cite{MR633160},~\cite{MR1802343} and Vinberg~\cite{MR2429266}.
On the other hand,
when $\Aut(X)$ is infinite, 
it is in general a  difficult problem to give a set of generators.
\par
%\medskip
We also have the following related problem.
Let $\Nef(X)$ denote the nef cone of $X$;  that is, the cone of $S_X\tensor \R$ consisting of vectors
$x\in S_X\tensor \R$  such that $\intM{x, C}{}\ge 0$ holds for any curve $C$ on $X$,
where  $\intM{\phantom{a}, \phantom{a}}{}$ is the intersection form on $S_X$.
In order to classify various geometric objects on $X$
(for example, smooth rational curves, Jacobian fibrations, or polarizations of a fixed degree)
\emph{modulo} $\Aut(X)$, 
it is useful to describe explicitly a fundamental domain  %$F_X$  
of the action of $\Aut(X)$ on $\Nef(X)$. % via $\varphi_X$.
\par
%\medskip
Several authors have studied these problems 
 by using the idea of Borcherds~\cite{MR913200},~\cite{MR1654763}
to embed $S_X$ into an even unimodular hyperbolic lattice of rank $26$.
In these works, however,
%they had to impose on $S_X$ a strong condition,
they required that $S_X$ should satisfy a certain strong condition
(see~Section~\ref{subsec:history} below for the details),
and hence the range of applications is limited.
\par
%\medskip
The purpose of this paper is to present an algorithm (Algorithm~\ref{algo:main})
that calculates, under   assumptions on $S_X$  milder than  the preceding works,
 a  finite set of generators of the image of $\varphi_X$
and a closed domain $F$ of  $\Nef(X)$ with the following properties:
\begin{itemize}
\item[(i)] For any $v\in \Nef(X)$, there exists an element $g\in \Aut(X)$ such that $v^g\in F$.
%\item[(ii)] For a general $v\in \Nef(X)$, there exist
%only a finite number of  $g\in \Aut(X)$ such that $v^g\in F$.
%
\item[(ii)] The domain $F$ is tiled by a finite number of convex cones, which we call chambers.
Each chamber is bounded by  a finite number of hyperplanes  and its stabilizer subgroup in $\Aut(X)$ is finite.
%
%\item[(ii)] The domain $F$ is tiled by a finite number of chambers.
%Each chamber has a finite number of walls and its stabilizer subgroup in $\Aut(X)$ is finite.
\end{itemize}
%
%The domain $F$ is larger than  a fundamental domain,  but is  useful in the classification of 
%geometric objects on $X$ modulo $\Aut(X)$.
See Remark~\ref{rem:weakfund}
for the relation of $F$ with a fundamental domain of 
the action of $\Aut(X)$ on $\Nef(X)$.
The detailed description of the assumptions we impose on $S_X$ 
will be given  in Section~\ref{sec:geometricapplication}.
\par
The algorithm can be applied to a wide class of $K3$ surfaces.
We give two examples. 
%one is a $K3$ surface with a small Picard number,
%and the other is  a $K3$ surface with a large Picard number.
%
%
%
%
\begin{example}\label{example:24}
As an example of a $K3$ surface with small Picard number and an infinite automorphism group,
we consider a complex $K3$ surface $X$  whose N\'eron-Severi lattice 
$S_X$ has
a Gram matrix 
$$
M:=\left[\begin{array}{ccc}
0 & 1 & 0 \\
1 & -2 & 0\\
0 & 0 & -24
\end{array}\right]
$$
with respect to a  certain basis $f_{\phi}, z_{\phi},  v$  such that 
$f_{\phi}$ and $z_{\phi}$ are the classes of a fiber and the zero section of a Jacobian fibration  $\phi: X\to \P^1$.
Then the  Mordell-Weil group $\MW_{\phi}$ of $\phi$ is isomorphic to $\Z$.
Hence $\Aut(X)$ contains an infinite subgroup $\MW_{\phi}\semidirectproduct \Z/2\Z$
(see~Section~\ref{sec:rho3}).
We assume  that the period $\omega_X$ of $X$ is generic in $T_X\tensor \C$,
where $T_X$ is the transcendental lattice of $X$.
%
%Let $X$ be a complex $K3$ surface with Picard number $3$ such that
%$S_X$ has
%a Gram matrix 
%$$
%M:=\left[\begin{array}{ccc}
%0 & 1 & 0 \\
%1 & -2 & 0\\
%0 & 0 & -24
%\end{array}\right]
%$$
%under  certain basis,
%%that the vector $a\in S_X$ written as $[2,1, 0]$ with respect  to  this basis is ample,
%and that the period $\omega_X$ of $X$ is generic in $T_X\tensor \C$,
%where $T_X$ is the transcendental lattice of $X$.
We let $\OG(S_X)$ act on $S_X$ from the right so that
$$
\OG(S_X)=\set{g\in \GL_3(\Z)}{g \,M\,{}^t g=M}.
$$
Then %the natural homomorphism 
$\varphi_X$
is injective, and its image is generated by the following  matrices:
$$
\left[
\begin{array}{ccc}
1 & 0 & 0 \\
0 & 1 & 0\\
0 & 0 & -1
\end{array}\right],
\quad
\left[
\begin{array}{ccc}
1 & 0 & 0 \\
12 & 1 & -1\\
24 & 0 & -1
\end{array}\right],
\quad
\left[\begin{array}{ccc}
37 & 12 & -5 \\
36 & 13 & -5\\
360 & 120 & -49
\end{array}\right],
\quad
\left[\begin{array}{ccc}
97 & 48 & -14 \\
0 & 1 & 0\\
672 & 336& -97
\end{array}\right].
$$
The first two elements of these four matrices are the images of  the involutions that generate
$\MW_{\phi}\semidirectproduct \Z/2\Z \cong \Z/2\Z* \Z/2\Z$.
We need two more  automorphisms to generate the full automorphism group $\Aut(X)$.
Moreover we can show that $\Aut(X)$ acts on the set of smooth rational curves on $X$ transitively.
In Figure~\ref{figure:Poincare},
we give a tessellation of $\Nef(X)$ by a fundamental domain of $\Aut(X)$.
See~Section~\ref{sec:rho3} or the author's web-page~\cite{compdataweb} for more examples of this type.
\end{example}
\begin{example}\label{example:singK3disc11}
A complex algebraic $K3$ surface  is said to be \emph{singular} if its Picard number 
attains the possible maximum  $20$.
By the result of Shioda and Inose~\cite{MR0441982},
the isomorphism class of 
a singular $K3$ surface $X$ is determined by its oriented transcendental lattice $T_X$,
and  $\Aut(X)$ is always infinite.
See~\cite{MR1820211} for the standard Gram matrices 
$$
\left[
\begin{array}{cc} 
a & b \\ b & c 
\end{array}
\right]
$$
of oriented transcendental lattices of singular $K3$ surfaces.
Let $\disc T_X$ denote the discriminant of $T_X$.
Since $T_X$ is an even lattice,
we see that  $\disc T_X=ac-b^2$ is congruent to $0\;\textrm{or}\; 3 \bmod 4$. %, because $a$ and $c$ are even.
We consider singular $K3$ surfaces $X$
with small  $\disc T_X$.
The automorphism groups of the singular $K3$ surfaces $X$
with $\disc T_X=3$ and $4$ %$\disc T_X=4$  
were determined by Vinberg~\cite{MR719348},
and $\Aut(X)$
of the singular $K3$ surface $X$
with $\disc T_X=7$ was determined by Ujikawa~\cite{MR3113614}.
The case where $\disc T_X=8$ can be handled by  \emph{usual} Borcherds' method
(see  Section~\ref{sec:singK3}).
Therefore the next example to be considered is 
the singular  $K3$ surface $X$
whose transcendental lattice  is given by a Gram matrix
$$
\left[
\begin{array}{cc} 
2 & 1 \\ 1 & 6
\end{array}
\right].
$$
In this case, $\varphi_X$ is injective.
We find that 
$\Image \varphi_X$ is generated by $767$ elements of $\OG(S_X)$,
and,  as in Example~\ref{example:24},
we can present them  explicitly  as $20\times 20$ matrices 
with respect to a certain basis of $S_X$.
These matrices and related computational data are given in the author's web-page~\cite{compdataweb}. 
From these data, we see that
the number of smooth rational curves on $X$ is at most $347$ modulo $\Aut(X)$.
(Note that the number $767$ of generators of $\Aut(X)$ and the upper-bound 
$347$ of the number 
of  smooth rational curves modulo $\Aut(X)$ are not minimal.)
See~Section~\ref{sec:singK3}  for more examples of this type.
\end{example}
\par
Our main algorithm (Algorithm~\ref{algo:main}) 
can be applied to a wide class of $K3$ surfaces theoretically.
The most crucial assumption that $S_X$ be embedded 
primitively into an even unimodular hyperbolic lattice of rank $10$, $18$ or $26$ 
is always satisfied when $X$ is defined over $\C$
(see Proposition~\ref{prop:embeddable}).
However,
experiments show that the computational complexity
of Algorithm~\ref{algo:main}
grows rapidly as the rank and the discriminant of $S_X$ become large.
\par
To the best knowledge of the author,
the automorphism group of the complex Fermat quartic surface $X$ is still unknown.
This surface is a singular $K3$ surface with $\disc T_X=64$,
and we can apply our algorithm to this surface.
It turns out, however,  that the computation is very heavy,
and we could not finish the calculation
(see~Remark~\ref{rem:complexFQ}).
\par
Even when the computation is too heavy to be completed in a reasonable time,
our algorithm yields many interesting automorphisms 
and projective models of the given $K3$ surface on the way of computation.
We have applied this method to the supersingular $K3$ surface in characteristic $5$ with Artin invariant $1$
in~\cite{MR3286672}.
\par
It is a totally different problem to give geometrically  an automorphism $g$ of $X$
such that $\varphi_X(g)$ is equal to a given matrix in $\OG(S_X)$.
%We will discuss an algorithmic approach to this problem in another paper.
In~\cite{ShimadaAutSingK3}, we discuss an algorithmic approach to this problem,
and apply it to  certain singular $K3$ surfaces.
\subsection{The difference of our algorithm from the preceding works}\label{subsec:history}
We briefly review 
 Borcherds' method~\cite{MR913200},~\cite{MR1654763},
and its applications to $K3$ surfaces.
A lattice of rank $n>1$ is said to be \emph{hyperbolic}  if 
its signature  is $(1, n-1)$.
Let $S$ be an even hyperbolic lattice of rank $<26$.
A \emph{positive cone} $\PPP_S$ of $S$ is one of  the two connected components
of $\shortset{x\in S\tensor \R}{x^2>0}$.
We denote by $\OG^+(S)$ the stabilizer subgroup of $\PPP_S$ in $\OG(S)$.
Let $W(S)$ denote the subgroup of $\OG^+(S)$ 
generated by the reflections in the hyperplanes $(r)\sperp$
%generated by the reflections into the hyperplanes $(r)\sperp$
perpendicular to $r$,
where $r$ runs through the set  $\Roots_S:=\shortset{r\in S}{r^2=-2}$.
The mirrors $(r)\sperp$
decompose $\PPP_{S}$ into the union of 
standard fundamental domains %
of the action of $W(S)$ on $\PPP_S$.
We call each of these standard fundamental domains
an \emph{$\Roots_{S}\sphyp$-chamber}.
Let $G$ be a subgroup of $\OG^+(S)$ with finite index.
Let $\Nc$ be an $\Roots_{S}\sphyp$-chamber,
and let $\aut_{G}(\Nc)$ denote the stabilizer subgroup in $G$
of the $\Roots_{S}\sphyp$-chamber $\Nc$.
In the application to a complex $K3$ surface $X$,
$S$ is the N\'eron-Severi lattice $S_X$,
$\PPP_S$ is the connected component containing an ample class,
the group $G$ consists of elements $g\in \OG^+(S)$
 liftable to  isometries of $H^2(X, \Z)$ that preserve the period of $X$,
and $\Nc$ is the intersection of  $\Nef(X)$ with $\PPP_S$.
\par
%\medskip
%
Borcherds gave a method to calculate 
%a fundamental domain of the action of $\aut_{G}(\Nc)$ on $\Nc$, and 
a finite set of generators of $\aut_{G}(\Nc)$.
Suppose  that  $S$  is primitively embedded 
into an even unimodular hyperbolic lattice $\L:=\textrm{II}_{1, 25}$ of rank $26$
in such a way that every $g\in G$ lifts to an isometry of $\L$.
We also assume that the orthogonal complement $R$ of $S$ in $\L$
cannot be embedded into the (negative-definite) Leech lattice.
%(See~$\assumpone$,~$\assumptwo$, ~$\assumpthree$ in Section~\ref{sec:Borcherds}.)
Let $\PPP_{\L}\subset \L\tensor\R$
be the  positive cone of $\L$ containing $\PPP_S$.
The structure of $\Roots_{\L}\sphyp$-chambers is well-understood by Conway~\cite[Chapter 27]{MR1662447}.
Then the tessellation of $\PPP_{\L}$  by the $\Roots_{\L}\sphyp$-chambers
induces a tessellation of $\PPP_S$,
which is invariant under the action of $G$.
We call the tiles constituting this induced tessellation
 \emph{$\RtsLS\sphyp$-chambers}.
Note that the $\Roots_{S}\sphyp$-chamber $\Nc$ is a union of $\RtsLS\sphyp$-chambers.
Two $\RtsLS\sphyp$-chambers $D$ and $D\sprime$ are said to be \emph{$G$-congruent}
if there exists an element $g\in G$ that maps $D$ to $D\sprime$.
By  the  reduction theory for arithmetic subgroups of $\OG (S)$,
we see that
the number of $G$-congruence classes of $\RtsLS\sphyp$-chambers is finite.
Let 
$$
\D:=\{ D_0, \dots, D_{m-1}\}
$$
be a complete set of representatives of $G$-congruence classes 
of $\RtsLS\sphyp$-chambers such that each $D_i$ is contained in $\Nc$.
From the list $\D$, we can obtain a finite set of generators of $\aut_{G}(\Nc)$.
%Then their union
%$\bigcup D_i$ is a fundamental domain of the action of $\aut_{G}(\Nc)$ on $\Nc$.
%
\par
%\medskip
%
Kondo~\cite{MR1618132} applied this method to the N\'eron-Severi lattice of
a generic Jacobian Kummer surface,
and described its automorphism group.
Since then,
automorphism groups of the following $K3$ surfaces have been determined by this method;
\begin{itemize}
\item[(a)] the supersingular $K3$ surface in characteristic $2$ 
with Artin invariant $1$ by Dolgachev and Kondo~\cite{MR1935564K},
\item[(b)] complex Kummer surfaces of product type
by Keum and Kondo~\cite{MR1806732}, 
\item[(c)] the Hessian quartic surface by Dolgachev and Keum~\cite{MR1897389},
\item[(d)] the singular $K3$ surface $X$ with $\disc T_X=7$ by Ujikawa~\cite{MR3113614}, 
\item[(e)] the supersingular $K3$ surface in characteristic $3$ 
with Artin invariant $1$ by
Kondo and Shimada~\cite{KondoShimada}.
\end{itemize}
The classical two examples of Vinberg~\cite{MR719348} can also be  treated by this method.
%and the automorphism group of the singular $K3$ surface $X$
%with $\disc T_X=8$ can also be treated by this method.
\par
%\medskip
In all these examples, the number $|\D|$ of $G$-congruence classes 
of $\RtsLS\sphyp$-chambers is $1$. 
Borcherds~\cite{MR913200} studied the case 
where the orthogonal complement  $R$ of $S$ in $\L$
contains a root sublattice of finite index,
and gave in~\cite[Lemma 5.1]{MR913200}
 a sufficient condition for 
any two  $\RtsLS\sphyp$-chambers 
to be $\OG^+(S)$-congruent.
%Revised on July 21 2014.
%In fact, 
%Borcherds~\cite[Lemma 5.1]{MR913200} proved that, 
%if the orthogonal complement  $R$ of $S$ in $\L$ is a root lattice,
%then any two  $\RtsLS\sphyp$-chambers are  $\OG^+(S)$-congruent, and 
%except for~\cite{MR719348} and~\cite{MR1618132}, they used this fact in their calculations.
In particular, the method employed in the above examples  is limited to  $K3$ surfaces $X$
such that $S_X$ can be embedded primitively in $\L$ with $R$ containing  a root lattice
as a sublattice of finite index.
Note that, for example,
only a finite number of isomorphism classes of singular $K3$ surfaces
satisfy this condition.
\par
%\medskip
%
We extend Borcherds' method 
to the situation in which   $|\D|$ is not necessarily $1$.
In fact, we have $|\D|=46$ in Example~\ref{example:24} 
and $|\D|=1098$ in Example~\ref{example:singK3disc11} above.
 Starting from an initial $\RtsLS\sphyp$-chamber,
 we compute $\RtsLS\sphyp$-chambers adjacent to the $\RtsLS\sphyp$-chambers obtained so far successively 
 until no new $G$-congruence classes appear.
 Each $\RtsLS\sphyp$-chamber is expressed by its Weyl vector~(see Definition~\ref{def:Sweylvector}).
 Thus our main algorithm contains sub-algorithms
 that calculate the set of walls of a given  $\RtsLS\sphyp$-chamber~(Algorithm~\ref{algorithm:aSnS}),
 compute the Weyl vector of the adjacent $\RtsLS\sphyp$-chamber across a given wall~(Algorithm~\ref{algorithm:adjacent}),
 and determine whether an $\RtsLS\sphyp$-chamber is $G$-congruent to another 
 $\RtsLS\sphyp$-chamber~(Algorithm~\ref{algorithm:Gequiv}).
 In these algorithms,
 we use  methods given in our previous paper~\cite{MR3166075}. 
 In the  calculation of  the set of walls,
 we also employ the standard algorithm of linear programming~(Algorithm~\ref{algorithm:mindefset}).
%We make use of the automorphim group of  an $\RtsLS\sphyp$-chamber
%to reduce the amount of the calculation.
%
%
\subsection{The plan of the paper}
In Section~\ref{sec:lattice},
we fix notions and notation about lattices
and hyperbolic spaces.
In Section~\ref{sec:chamber},
we introduce the notion of chamber decomposition of a positive cone
of a hyperbolic lattice,
and present some algorithms about chambers.
In Section~\ref{sec:Conway},
we review Vinberg-Conway theory~(\cite{MR0422505} and~\cite[Chapter 27]{MR1662447})
on the structure of $\Roots_{\L}\sphyp$-chambers in 
the even unimodular hyperbolic lattice $\L=\textrm{II}_{1, n-1}$  of rank $n=10, 18$ and  $26$. 
Then, in Section~\ref{sec:Borcherds},
we introduce the notion of $\RtsLS\sphyp$-chambers
associated with a primitive embedding $S\inj \L$ of an even hyperbolic lattice $S$,
and  present 
some algorithms.
In Section~\ref{sec:main},
we present our main algorithm~(Algorithm~\ref{algo:main}),
and prove its correctness.
%Up to this point,
%we are concerned only with lattices. 
In Section~\ref{sec:AutX}, 
we review the Torelli-type theorem for $K3$ surfaces,
and show how to calculate  the automorphism group  from the N\'eron-Severi lattice.
In Section~\ref{sec:geometricapplication},
we explain how to apply Algorithm~\ref{algo:main} to the study of 
$K3$ surfaces.
In particular, 
we describe in detail what geometric data of a $K3$ surface $X$
must be calculated before we apply Algorithm~\ref{algo:main} to $X$.
In Sections~\ref{sec:rho3} and~\ref{sec:singK3}, 
we demonstrate Algorithm~\ref{algo:main} on  some $K3$ surfaces
with Picard number $3$, and some singular $K3$ surfaces.
\par 
\medskip
In this paper, $\aut$ denotes the automorphism group of a lattice theoretic object,
whereas $\Aut$ denotes the geometric automorphism group of a $K3$ surface.
\par 
\medskip
For the actual computation,
we used the {\tt C} library  {\tt gmp}~\cite{gmp}.
\section{Preliminaries on lattices and hyperbolic spaces}\label{sec:lattice}
Let $L$ be a free $\Z$-module of finite rank.
We say that a submodule $M$ of $L$ is \emph{primitive} if $L/M$ is torsion-free,
and 
that $v\in L$ is \emph{primitive} if so is the submodule $\gen{v}:=\Z v\subset L$.
\par
Let $L$ be a lattice;
that is, a free $\Z$-module of finite rank with
a non-degenerate symmetric bilinear form $\intM{\phantom{\cdot}, \phantom{\cdot}}{L}\colon L\times L\to \Z$.
The subscript $L$ in $\intM{\phantom{\cdot}, \phantom{\cdot}}{L}$ is omitted if no confusion will occur.
The symmetric bilinear form 
on $L\tensor\R$ 
obtained from $\intM{\phantom{\cdot}, \phantom{\cdot}}{}$ 
%by extending  scalars 
is also denoted by $\intM{\phantom{\cdot}, \phantom{\cdot}}{}$.
We denote by $L\dual:=\Hom(L, \Z)$ the \emph{dual lattice of $L$},
which is naturally embedded in $L\tensor\Q$ by $\intM{\phantom{\cdot}, \phantom{\cdot}}{}$.
We say that $L$ is \emph{unimodular} if $L=L\dual$ holds.
The  \emph{norm} $\intM{v,v}{}$ of $v\in L\tensor \R$ is denoted by $v^2$.
We let the \emph{orthogonal group} $\OG(L)$ of $L$  act on $L$ from the right,
and write the action of $g\in \OG(L)$ on $v\in L\tensor\R$ by $v\mapsto v^g$.
\par
A lattice $L$ is said to be \emph{even} if $x^2\in 2\Z$ holds  for any $x\in L$.
For an even lattice $L$,
we put
$$
\Roots_L:=\set{r\in L}{r^2=-2}.
$$
Elements of $\Roots_L$ are called \emph{$(-2)$-vectors}.
Each $r\in \Roots_L$ defines 
the \emph{reflection}
$$
s_r \colon x\mapsto x+\intM{x, r}{} r,
$$
which is an element of $\OG(L)$.
We denote by $W(L)$ the subgroup of $\OG(L)$ 
generated by all  reflections $s_r$
with respect to  the $(-2)$-vectors $r\in \Roots_L$,
and call it the \emph{Weyl group}.
\par
%\medskip
A lattice $L$ of rank $n>0$ is said to be \emph{negative-definite}  if the signature of the real quadratic space
$L\tensor \R$  is $(0, n)$.
A negative-definite lattice $L$ is said to be a \emph{root lattice}
if $L$ is generated by $\Roots_L$.
\par
%\medskip
A lattice $L$ of rank $n >1$ is said to be \emph{hyperbolic} if the signature of 
$L\tensor \R$  is $(1, n-1)$.
Let $L$ be a hyperbolic lattice.
Then a \emph{positive cone of $L$} is one of the two connected components
of $\shortset{x\in L\tensor\R}{x^2>0}$.
The closure of a positive cone $\PPP_L$ in $L\tensor\R$ is denoted by $\clPPP_L$.
We denote by 
$\clPPPQ_L$
the convex hull of $\clPPP_L\cap(L\tensor\Q)$.
The stabilizer subgroup in $\OG(L)$ of a positive cone is denoted by $\OG^+(L)$.
We have $\OG(L)=\OG^+(L)\times\{\pm 1\}$.
Note that $W(L)$ is contained in $\OG^+(L)$.
\par
%\medskip
The following 
 algorithms  will be used frequently in this paper.
\begin{algorithm}\label{algo:QLc}
Let $Q$ be a positive-definite symmetric matrix of size $n$
with rational entries,
$\ell$ a column vector of length $n$ with rational entries,
and $c$ a rational number.
Then we can calculate the list of all row vectors $x\in \Z^n$
satisfying 
$$
x\,Q\,{}^t x + 2\,  x\, \ell  +c=0
$$
by the method described in~\cite[Section 3.1]{MR3166075}. 
\end{algorithm}
\begin{algorithm}\label{algo:QLcab}
Let $L$ be a hyperbolic lattice,
let $v$ be a vector of $L\tensor \Q$ with $v^2>0$,
let $\alpha$ be a rational number,  and let $d$ be an integer.
Then the finite set
$$
\set{x\in L}{\intM{x, v}{}=\alpha, \;\; \intM{x,x}{}=d}
$$
can be calculated by the method 
described in~\cite[Sections 3.2]{MR3166075}.
\end{algorithm}
\begin{algorithm}\label{algo:isnef}
Let $L$ be a hyperbolic lattice,
let $v, h$ be  vectors of $L\tensor \Q$ such that
$$
\intM{v, h}{}>0, \;\; \intM{h, h}{}>0,  \;\; \intM{v, v}{}>0,
$$
and
let $d$ be a negative integer.
Then the finite set
$$
\set{x\in L}{\intM{v, x}{}<0, \;\; \intM{h, x}{}>0, \;\; \intM{x, x}{}=d}
$$
can be calculated
by the method 
described in~\cite[Sections 3.3]{MR3166075}.
\end{algorithm}
Let $L$ be a hyperbolic lattice of rank $m+1$,
and let $\PPP_L$ be a positive cone of $L$.
The multiplicative group $\R_{>0}$ of positive real numbers
acts on $\clPPP_L\setminus\{0\}$ by scalar multiplication.
We put 
\begin{eqnarray*}
&&\clH_L:=(\clPPP_L\setminus\{0\})/\R_{>0}, \quad \clHQ_L:=(\clPPPQ_L\setminus\{0\})/\R_{>0}, \quad
\H_L:=\PPP_L/\R_{>0},
\\ 
&&\bdrclH_L:=\clH_L\setminus \H_L, \quad \bdrclHQ_L:=\clHQ_L\setminus \H_L,  
\end{eqnarray*}
and denote by $\pi_L \colon \clPPP_L\setminus\{0\}\to \clH_L$ the natural projection.
The space $\H_L$ is naturally endowed with a structure of the hyperbolic $m$-space.
A point of $\bdrclHQ_L$ is called a \emph{rational boundary point}.
For simplicity,
when we are given a subset $T$ of $\clPPP_L$,
we denote by $\pi_L(T)\subset\clH_L$ the image of  $T\setminus\{0\}$ by $\pi_L$.  
For example,  we have $\bdrclHQ_L=\pi_L(\clPPP_L\cap L)$.
A subset $K$ of $\H_L$ is said to be a \emph{linear subspace of $\H_L$}
if there exists a linear subspace $\tilde{K}$ of $L\tensor\R$
such that $K=\pi_L(\PPP_L\cap \tilde{K})$ holds.
Let $b$ be a point of $\bdr\clH_L$.
A linear subspace $K$ of $\H_L$ is said to \emph{pass through $b$ at infinity}
if the linear subspace $\tilde{K}$ of $L\tensor\R$
that satisfies $K=\pi_L(\PPP_L\cap \tilde{K})$ contains a non-zero vector $v$ such that $b=\pi_L(v)$.
%$$
%\bdrclHQ_L=\pi_L(\clPPP_L\cap L).
%$$
\par
%\medskip
We denote by 
$\HHH_L$ the set of hyperplanes of $\PPP_L$.
We put
$$
\NNN_L:=\set{v\in L\tensor \R}{v^2<0}, 
$$
and, for $v\in L\tensor\R$, we put
$$
[v]\sperp:=\set{x\in L\tensor \R}{\intM{x, v}{}=0}
\quand
(v)\sperp:=[v]\sperp\cap \PPP_L.
$$
Since $[v]\sperp$ intersects $\PPP_L$ if and only if $v^2<0$, the map $v\mapsto (v)\sperp$ 
induces a bijection 
from $\NNN_L/\R\sptimes$ to $\HHH_L$. 
For $v\in \NNN_L$, we obtain a hyperplane $\pi_L((v)\sperp)$ of the hyperbolic space $\H_L$.
\par
We recall some properties of \emph{horospheres}.
For the details, 
see  Ratcliffe~\cite[Chapter 4]{MR1299730}.
Let $b$ be a point of $\bdrclH_L$.
Consider the upper halfspace model
\begin{equation}\label{eq:upper}
\set{(z_1, \dots, z_m)\in \R^m}{z_1>0}
\end{equation}
of $\H_L$ such that $b$ corresponds to the infinite point given by $z_1=\infty$.
Then every horosphere $\HS_b$ with the base $b$ is defined by
$z_1=\gamma$
with some positive real constant $\gamma$.
Therefore, 
as a Riemannian submanifold of $\H_L$,
every horosphere is isomorphic to a Euclidean affine space.
If $K$ is a linear subspace of $\H_L$ that passes through $b$ at infinity, 
then $K\cap \HS_b$ is an affine subspace of any horosphere $\HS_b$ with the base $b$.
We say that a horosphere defined by $z_1=\gamma$ is \emph{smaller} than
a horosphere defined by $z_1=\gamma\sprime$ if $\gamma>\gamma\sprime$.
\par
The definition of horospheres can be restated as follows.
Let $f$ be a non-zero vector in $\clPPP_L$ with $f^2=0$ such that $\pi_L(f)=b$.
Then the function  
$$
h_f\colon v\mapsto \intM{v, f}{}^2/v^2
$$
on $\NNN_L\cup \PPP_L\cup (-\PPP_L)$ is invariant under the scaling of $v$ by multiplicative constants,
and hence its restriction to $\PPP_L$
induces a function
$$
\bar{h}_f\colon \H_L \to \R_{>0}.
$$
The horospheres with the base $b$ are exactly the level sets of the function $\bar{h}_f$.
Note that the horosphere defined by  $\bar{h}_f=\alpha$
is smaller than
the horosphere defined by  $\bar{h}_f=\alpha\sprime$ if and only if $\alpha<\alpha\sprime$.
\par
The following lemma should be well-known,
but we could not find appropriate references.
Let $f$ and $b$ be as above.
%Let $f$ be a non-zero vector in $\clPPP_L$ with $f^2=0$,
%and we put $b:=\pi_L(f)\in \bdrclH_L$.
The function $-h_f(v)=-\intM{v, f}{}^2/v^2$ restricted to $\NNN_L$
measures how far  the hyperplane $\pi_L((v)\sperp)$ of $\H_L$
is from  the boundary point $b$.
For a horosphere $\HS_b$ with the base $b$, 
we put
$$
c(f, \HS_b):=\sup\; \set{-h_f(v)}{v\in \NNN_L, \;\;\pi_L((v)\sperp)\cap \HS_b \ne \emptyset}.
$$
\begin{lemma}\label{lemma:HS}
Suppose that $\HS_b$  is defined by $\bar{h}_f=\alpha$.
Then we have $c(f, \HS_b)\le \alpha$.
\end{lemma}
\begin{proof}
We choose  linear coordinates $(x_0, x_1, \dots, x_m)$ of $L\tensor\R$ such that
the quadratic form $x\mapsto x^2$ on $L\tensor \R$ is given by
$$
(x_0, x_1, \dots, x_m)\mapsto x_0^2-x_1^2-x_2^2-\cdots-x_m^2,
$$
such that $\PPP_L$ is contained in the halfspace $x_0> 0$, and such that
$f=(1,1,0,\dots, 0)$.
Let $v=(v_0,  v_1, \dots, v_m)$ be a vector in $\NNN_L$.
%We have $\intM{v, f}{}=v_0-v_1$.
Suppose that there  exists a  vector  $x=(x_0, x_1, \dots, x_m)$ in $\PPP_L$
such that $x\in (v)\sperp$ and that $\pi_L(x)\in \HS_b$.
It is enough to show that
\begin{equation}\label{eq:target}
%-\frac{\intM{v, f}{}^2}{v^2}\le \alpha.
-h_f(v)\le \alpha.
\end{equation}
Rescaling $x$ by a positive real constant, we can assume that $x^2=1$.
If $\intM{v, f}{}=0$, then~\eqref{eq:target} holds.
Suppose that $\intM{v, f}{}\ne 0$.
Replacing $v$ by  $-v$ if necessary, we can assume that $\intM{v, f}{}=v_0-v_1$ is positive.
Since  $\pi_L(x)\in \HS_b$, we have $h_f(x)=\alpha$. 
Since $x\in \PPP_L$, we have $\intM{x, f}{}=x_0-x_1>0$.
Combining these,
we get
$$
x_0-x_1=\sqrt{\alpha}.
$$
Using $\intM{v, x}{}=v_0x_0-v_1 x_1- \dots-v_m x_m=0$, we obtain
$$
x_0=\frac{(v_2x_2+\cdots+v_m x_m)-v_1\sqrt{\alpha}}{v_0-v_1},
\quad
x_1=\frac{(v_2x_2+\cdots+v_m x_m)-v_0\sqrt{\alpha}}{v_0-v_1}.
$$
Combining these with $x^2=1$ and using $v_0\ne v_1$, we get
\begin{equation}\label{eq:gsum}
g_2(x_2)+\cdots +g_m(x_m)+\alpha(v_0+v_1)+(v_0-v_1)=0,
\end{equation}
where
$$
g_i(t)=(v_0-v_1)\,t^2-2\sqrt{\alpha} \,v_i t.
$$
Since $\intM{v, f}{}=v_0-v_1>0$, we have  $g_i(t)\ge -\alpha v_i^2/(v_0-v_1)$ for any $t\in \R$.
Thus we obtain $\alpha v^2+\intM{v, f}{}^2\le 0$ from~\eqref{eq:gsum}.
Since $v^2<0$, we get~\eqref{eq:target}.
\end{proof}
\section{Chamber decomposition}\label{sec:chamber}
Let $L$ be an even hyperbolic lattice
with a fixed positive cone $\PPP_L$.
For a subset $\Delta$ of $\NNN_L=\shortset{v\in L\tensor\R}{v^2<0}$, 
we define a cone $\Sigma_L(\Delta)$ in $L\tensor \R$ by 
\begin{equation}\label{eq:Sigma}
\Sigma_L(\Delta):=\set{x\in L\tensor \R}{\intM{x, v}{}\ge 0\;\;\textrm{for all}\;\; v\in \Delta}.
\end{equation}
\begin{definition}\label{def:chamber}
A closed subset $D$ of $\PPP_L$ is called a \emph{chamber}
if its interior $D\spcirc$ is non-empty and
there exists a subset $\Delta$ of $\NNN_L$  such that $D=\Sigma_L(\Delta)\cap \PPP_L$ holds.
Let $D$ be a chamber.
A hyperplane $(v)\sperp$ of $\PPP_L$ %defined by $\intM{x, v}{}=0$  
is called a \emph{wall of $D$} if $(v)\sperp\cap D\spcirc$ is empty
and $(v)\sperp\cap D$ contains a non-empty open subset of $(v)\sperp$.
%We denote by $\clD $ the closure  of $D$ in $\clPPP_L$.
 \end{definition}
 \begin{definition}
 Let $\FFF\subset\HHH_L$ be a locally finite family of hyperplanes in $\PPP_L$.
 Then  the closure in $\PPP_L$ of each connected component of
$$
\PPP_L\setminus \bigcup_{(v)\sperp \in \FFF}\, (v)\sperp
$$ 
is a chamber,
which we call an \emph{$\FFF$-chamber}.
\end{definition}
By definition, 
every wall of an $\FFF$-chamber is an element of $\FFF$.
If $D$ and $D\sprime$ are distinct $\FFF$-chambers, then $D\spcirc \cap D\sprime=\emptyset$ holds.
 \begin{definition}
Let $D$ be an $\FFF$-chamber,  and let $(v)\sperp\in \FFF$ be a wall of $D$.
Then there exists a unique $\FFF$-chamber $D\sprime$
such that $D\sprime\ne D$ and that 
$D\cap D\sprime\cap (v)\sperp$ contains a non-empty open subset of $(v)\sperp$.
We say that  $D\sprime$ is  \emph{adjacent to $D$ across the wall $(v)\sperp$}.
\end{definition}
We fix a subgroup $G$ of $\OG^+(L)$ with finite index.
We assume  that $G$ satisfies the following condition of the existence of a membership algorithm:
\begin{itemize}
\item[$\propVG$]
There exists an algorithm by which we can determine,
for a given $g\in \OG^+(L)$,   whether $g\in G$ or not.
\end{itemize}
Let $\VVV$  be a subset of  $\NNN_L\cap L\dual=\shortset{v\in L\dual}{v^2<0}$, and consider the family of hyperplanes
$$
\VVV\sphyp:=\set{(v)\sperp}{v\in \VVV}
$$
in $\PPP_L$.
We assume that  $\VVV$ has the following properties:
\begin{itemize}
\item[$\propVone$] There exists  a positive real number $c\in \R$ such that,  for any $v\in \VVV$, we have $-v^2<c$. 
\item[$\propVtwo$] The set  $\VVV$ is invariant under  the action of $G$ on $\NNN_L \cap L\dual$. 
\end{itemize}
Then we can consider $\VVV\sphyp$-chambers by  the following:
\begin{lemma}\label{lem:locallyfinite}
The family $\VVV\sphyp$ of hyperplanes 
is  locally finite in $\PPP_L$.
\end{lemma}
\begin{proof}
Since  $\shortset{v^2}{v\in L\dual}$ is discrete in $\R$,
the subset $\shortset{v^2}{v\in \VVV}$ of the interval $(-c, 0)\subset \R$  is finite by~$\propVone$
and $\VVV\subset \NNN_L\cap L\dual$.
For a negative real number $a$ and a compact subset $J$ of $\PPP_L$,
the locus
$$
\set{x\in L\tensor \R}{x^2=a,  \,\, (x)\sperp\cap J\ne \emptyset}
$$
is compact.
Since  $L\dual$ is discrete in $L\tensor\R$, 
we obtain the proof of Lemma~\ref{lem:locallyfinite}.
\end{proof}
\begin{definition}
Let $D$ be  a $\VVV\sphyp$-chamber.
 A subset $\Delta$ of $\NNN_L\cap L\dual$ is called a \emph{defining set of $D$}
 if $D=\Sigma_{L}(\Delta)\cap \PPP_L$ holds.
% (Note that $\Delta$ need not be contained in $\VVV$.)
 A defining set $\Delta$ of $D$ is said to be \emph{minimal} 
 if the following hold:
 \begin{itemize}
 \item For any $v\in \Delta$, the hyperplane $(v)\sperp$ is a wall of $D$, and 
 \item if $v$ and $v\sprime$ are distinct vectors of $\Delta$, then $(v)\sperp\ne (v\sprime)\sperp$.
 \end{itemize}
 We define two types of minimal defining sets,  each of which is unique 
 for a given $\VVV\sphyp$-chamber $D$.
 The one is called the \emph{$\VVV$-minimal defining set},
 denoted by $\Delta_{\VVV}(D)$, and characterized by the following property:
 %
% \begin{itemize}
% \item 
 $\Delta_{\VVV}(D)$ is a subset of $\VVV$, and  if $v\in \Delta_{\VVV}(D)$, then $\alpha v\notin \VVV$
 for any $\alpha\in \R$ with $0<\alpha<1$.
% \end{itemize}
 %
 The other is called the \emph{primitively minimal defining set},
 denoted by $\Delta_{L\dual}(D)$, and characterized by the following property:
 %
 %\begin{itemize}
 %\item 
 Every $v\in \Delta_{L\dual}(D)$ is primitive in $L\dual$.
 %\end{itemize}
 %
 %
 (Note that $\Delta_{L\dual}(D)$ may not be contained in $\VVV$, because elements of $\VVV$ need not be primitive in $L\dual$.)
 \end{definition}
Let $D$ be 
 a $\VVV\sphyp$-chamber.
 Then $D^g$ is also a $\VVV\sphyp$-chamber for any $g\in G$ 
 by the property~$\propVtwo$.
 For a  $\VVV\sphyp$-chamber $D$,
we put
$$
\autG(D):=\set{g\in G}{D^g=D}.
$$
Let $D$ and $D\sprime$ be $\VVV\sphyp$-chambers.
We say that  $D$ and $D\sprime$ are \emph{$G$-congruent}
if there exists an element $g\in G$ such that $D\sprime=D^g$.
The following simple observation is the key point of our method:
\begin{equation}\label{eq:intpt}
\textrm{
If $g\in G$ satisfies $D^g\cap D\sp{\prime\circ}\ne\emptyset$, 
then $D^g=D\sprime$.}
\end{equation}
In particular, if $g\in G$ satisfies $D^g\cap D\spcirc\ne\emptyset$, then $g\in \autG(D)$.
\begin{example}\label{example:W}
The subset $\Roots_L$ of $\NNN_L\cap L\dual$ has the properties~$\propVone$ and~$\propVtwo$
for $G=\OG^+(L)$.
Let $D$ be an $\Roots_L\sphyp$-chamber.
Then $D$ is a fundamental domain of the action of the Weyl group $W(L)$ on $\PPP_L$, %ok
and  $W(L)$ 
is generated by the reflections $s_r$, 
where $r$ runs through the $\Roots_L$-minimal defining set $\Delta_{\Roots_L}(D)$
of $D$.
Moreover,  any two $\Roots_L\sphyp$-chambers are $\OG^+(L)$-congruent,
and $\OG^+(L)$ is isomorphic to the semi-direct product
$W(L)\semidirectproduct \aut_{\OG^+(L)}(D)$.
\end{example}
We consider the following properties of $\VVV$:
%We further assume that  $\VVV$ has the following properties:
%
\begin{itemize}
\item[$\propVthree$] 
Any $\VVV\sphyp$-chamber  has a \emph{finite} defining set.
%For any $\VVV\sphyp$-chamber $D$, there exists a \emph{finite} subset $\Delta$ of 
%$\NNN_L\cap L\dual$ such that
%$D$ is equal to $\Sigma_L(\Delta)\cap \PPP_L$.
\item[$\propVfour$] For any $\VVV\sphyp$-chamber $D$, 
the set  $\pi_L(\clD )\cap \bdrclH_L$ is contained in $\bdrclHQ_L$,
where  $\clD $ is the closure  of $D$ in $\clPPP_L$.
\end{itemize}
The main results of this section are the following:
\begin{theorem}\label{thm:finiteGequiv}
Suppose that $\VVV$ satisfies $\propVone$-$\propVfour$.
Then there exist only a finite number of $G$-congruence classes of $\VVV\sphyp$-chambers.
\end{theorem}
\begin{theorem}\label{thm:finiteautG}
Suppose that $\VVV$ satisfies $\propVone$-$\propVfour$.
Then the  group $\aut_G(D)$ is finite for any  $\VVV\sphyp$-chamber $D$.
\end{theorem}
\subsection{Proof of Theorem~\ref{thm:finiteGequiv}}
We assume that  $\VVV$ satisfies $\propVone$-$\propVfour$.
\begin{definition}
A subset $\Pi$ of $\clPPPQ_L$ is called a \emph{rational polyhedral cone}
if there  exist a finite number of non-zero vectors 
$v_1, \dots, v_n\in \clPPP_L\cap L$ such that 
$$
\Pi=\R_{\ge 0} v_1+ \cdots+ \R_{\ge 0} v_n.
$$
\end{definition}
Recall that $G$ is assumed to be of finite index in $\OG^+(L)$.
We have the following result from the reduction theory of arithmetic subgroups
of $\OG(L)$
(see Ash et al.~\cite[Chapter II, Section 4]{MR2590897} and Sterk~\cite{MR786280}).
\begin{theorem}\label{thm:ratpoly}
There exist a finite number of rational polyhedral cones
$\Pi_1, \dots, \Pi_N$ in $\clPPPQ_L$ such that $\PPP_L$ is equal to 
$$
\bigcup_{g \in G} \bigcup_{i=1}^N \,(\Pi_i^g\cap \PPP_L).
$$
\end{theorem} 
Therefore 
Theorem~\ref{thm:finiteGequiv} follows from  the following:
\begin{proposition}\label{prop:Pi}
Let $\Pi$ be a rational polyhedral cone in $\clPPPQ_L$.
Then the number of $\VVV\sphyp$-chambers that intersect $\Pi\cap\PPP_L$ is finite.
\end{proposition}
For the proof of Proposition~\ref{prop:Pi},
we need two corollaries of Lemma~\ref{lemma:HS}.
% 
%\par
%\medskip
Let $b$ be a point of $\bdr\clH_L$.
A \emph{closed horoball}  $\HB_b$ with the base $b$ is a subset of $\H_L$ defined  by 
$z_1\ge \gamma$ with some positive real constant $\gamma$
in the upper halfspace model~\eqref{eq:upper}
with $b$ at $z_1=\infty$.
Let $\bdrHB_b$ be the horosphere defined  by $z_1=  \gamma$.
The map $\rho_b\colon \HB_b\to \bdrHB_b$ defined by  
$$
\rho_b(z_1, z_2, \dots, z_m):=(\gamma, z_2, \dots, z_m)
$$
is called the \emph{natural projection}.
%The map from $\HB_b$ to  $\bdrHB_b$  given by 
%$$
%(z_1, z_2, \dots, z_m)\mapsto (\gamma, z_2, \dots, z_m)
%$$
%is called the \emph{natural projection} and is denoted by 
%$\rho_b\colon \HB_b\to \bdrHB_b$. 
Let $b=\pi_L(f)\in \bdrclHQ_L$  be a rational  boundary point,
where $f$ is a non-zero vector in $\clPPP_L \cap L$ with  $f^2=0$.
We put
$$
\VVV_b:=\set{v\in \VVV}{\intM{v, f}{}=0}.
$$
If  $v\in \VVV$ satisfies $v\notin\VVV_b$,
then we have $\intM{v, f}{}^2\ge 1$ because $f\in L$ and $\VVV\subset L\dual$. 
By the property~$\propVone$, there exists a positive real number $\delta_{b}$ such that
$$
\delta_{b}< -\intM{v,f}{}^2/v^2\quad \textrm{for any $v\in \VVV\setminus \VVV_b$}.
$$
Therefore we obtain the following corollary of Lemma~\ref{lemma:HS}:
\begin{corollary}\label{cor:HS}
Let $b$ be a rational boundary point.
If we choose a sufficiently small closed horoball  $\HB_b$
with the base $b$, then the following hold.
\par
{\rm (1)} 
Let $v$ be an element of $\VVV$.
Then the hyperplane $\pi_L((v)\sperp)$ of $\H_L$ intersects $\HB_b$
if and only if
$\pi_L((v)\sperp)$ passes through $b$ at infinity.
%The following conditions for $v\in \VVV$ are equivalent;
%{\rm (i)} the hyperplane $\pi_L((v)\sperp)$ of $\H_L$ intersects $\HB_b$,
%{\rm (ii)} $\pi_L((v)\sperp)$ intersects the   horosphere $\bdrHB_b$, 
%and
%{\rm (iii)} $\pi_L((v)\sperp)$ passes through $b$ at infinity.
\par
{\rm (2)} 
Let  $D$ be a $\VVV\sphyp$-chamber.
%If $\pi_L(\closure{D})$ does not contain $b$,
If $b\notin \pi_L(\closure{D})$,
then $\pi_L(D)\cap\HB_b$ is empty,
whereas  %if $\pi_L(\closure{D})$  contains $b$,
if $b\in \pi_L(\closure{D})$,
then $\pi_L(D)\cap\HB_b=\rho_b\inv(\pi_L(D)\cap\bdrHB_b )$ holds.
\end{corollary}
We regard $\bdrHB_b$ as a Euclidean affine space.
Then the family of affine hyperplanes
$$
\set{\pi_L((v)\sperp)\cap \bdrHB_b }{v\in \VVV_b}
$$
of $\bdrHB_b $ is locally finite,
because $\VVV\sphyp$ is locally finite in $\PPP_L$ and $\pi_L\inv(\bdrHB_b )\subset\PPP_L$.
Therefore we obtain the following:
\begin{corollary}\label{cor:J}
Let $b$ and $\HB_b$ be as in Corollary~\ref{cor:HS},
 and let $J$ be a compact subset of $\bdrHB_b $.
Then the  number of $\VVV\sphyp$-chambers 
that intersect the subset $\pi_L\inv (\rho_b\inv (J))$ of $\PPP_L$
is finite.
\end{corollary}
\begin{proof}[Proof of Proposition~\ref{prop:Pi}]
Let $v_1, \dots, v_n\in \clPPP_L\cap L$ be non-zero vectors
such that rational polyhedral cone $\Pi$ is equal to $\R_{\ge 0} v_1+ \cdots+ \R_{\ge 0} v_n$.
We number $v_1, \dots, v_n$ in such a way that 
$$
v_1^2=\cdots=v_k^2=0
\quand
v_{k+1}^2>0, \;\; \dots, \;\;  v_{n}^2>0.
$$
Let $b_i$ be the rational boundary point $\pi_L(v_i)$ for $i=1, \dots, k$,
and let $\HB_i$ be a sufficiently small closed horoball with the base $b_i$.
The natural projection is denoted by $\rho_i :\HB_i\to \bdrHB_i$.
We put  $\HB_i\spcirc :=\HB_i\setminus \bdrHB_i$.
Since $\pi_L(\Pi)\cap\bdr\clH_L=\{b_1, \dots, b_k\}$,
we see that 
$$
\pi_L(\Pi)\sprime:=\pi_L(\Pi)\setminus\bigcup_{i=1}^k (\HB_i\spcirc \cap \pi_L(\Pi))
$$
is a compact subset of $\H_L$.
Therefore the number of $\VVV\sphyp$-chambers $D$
such that $\pi_L(D)$ intersects $\pi_L(\Pi)\sprime$ is finite.
%Therefore there exist only a finite number of $\VVV\sphyp$-chambers $D$
%such that $\pi_L(D)$ intersects $\pi_L(\Pi)\sprime$.
For $j\ne i$,
let $p_{i, j}\in \bdrHB_i$ denote  the intersection point 
of $\bdrHB_i$ and the geodesic line in $\H_L$ passing through $b_i$ at infinity and 
passing through $\pi_L(v_j)$ possibly at infinity.
Then the convex hull $J_i$ of these points $p_{i, j}$ with $j\ne i$
in the Euclidean affine space $\bdrHB_i$ is compact and satisfies
$\HB_i \cap \pi_L(\Pi)=\rho_i\inv (J_i)$.
Consequently,  the number of $\VVV\sphyp$-chambers $D$
such that $\pi_L(D)$ intersects $\HB_i \cap \pi_L(\Pi)$
is finite by Corollary~\ref{cor:J}.
\end{proof}
\subsection{Proof of Theorem~\ref{thm:finiteautG}}
We continue to assume that  $\VVV$ satisfies $\propVone$-$\propVfour$.
 \begin{lemma}\label{lem:clSigma}
 Let
 $\Delta$ be a subset of $\NNN_L\cap L\dual$ such that $D=\Sigma_L(\Delta)\cap\PPP_L$ is
 a  $\VVV\sphyp$-chamber.
 Then $\Sigma_L(\Delta)$ is contained  in $\clPPP_L$.
 \end{lemma}
 \begin{proof}
 Note that $\Sigma_L(\Delta)$ is a closed convex subset of $L\tensor\R$.
 Suppose that there exists an element $x_0\in \Sigma_L(\Delta)$
 such that $x_0\notin\clPPP_L$.
 Let $U$ be a non-empty open subset of $D\spcirc$.
 Then,
 for any $y\in U$,
 the line segment $\lineseg{x_0 y}$  of $L\tensor\R$ connecting $x_0$ and $y$
 is contained in $\Sigma_L(\Delta)$,
 and $\lineseg{x_0 y}\cap \PPP_L$  is contained in  $D$.
 Hence the intersection point $z(y)$ of $\lineseg{x_0 y}$ and 
 $\bdr\clPPP_L=\shortset{x\in \clPPP_L}{x^2=0}$ 
 belongs to  the closure $\overline{D}$ of $D$ in $\clPPP_L$.
 Since $U$ is open,
 the subset
 $\pi_L(\shortset{z(y)}{y\in U})$
 of $\pi_L(\overline{D})\cap\bdrclH_L$ has uncountably many points,
 which contradicts the property~$\propVfour$ of $\VVV$.
 \end{proof}
 \begin{lemma}\label{lem:span}
 Let $D$ be a $\VVV\sphyp$-chamber.
 Then any defining set $\Delta$ of $D$ spans $L\tensor\R$.
 \end{lemma}
 \begin{proof}
 Let $V$ be the linear subspace  of $L\tensor \R$ spanned by $\Delta$.
 The orthogonal complement $V\sperp$ of $V$  in $L\tensor\R$
 is contained in $\Sigma_L(\Delta)$ by definition and hence $V\sperp\subset\clPPP_L$ by 
 Lemma~\ref{lem:clSigma}.
 This holds only when $V\sperp=0$.
\end{proof}
 Theorem~\ref{thm:finiteautG} is now easy to prove by $\propVthree$. 
 We give, however,  a proof based on 
 an algorithm (Algorithm~\ref{algorithm:autG}) to compute $\aut_G(D)$.
\begin{lemma}\label{lem:iswall}
Suppose that a defining set $\Delta_1$ of a $\VVV\sphyp$-chamber $D$
satisfies the following;
if $v, v\sprime\in \Delta_1$ are distinct,
then $(v)\sperp\ne (v\sprime)\sperp$ holds.
Let $v$ be an element of $\Delta_1$.
{\rm (1)}
The hyperplane $(v)\sperp$ is a wall of $D$ if and only if 
$\Sigma_L(\Delta_1) \ne \Sigma_L(\Delta_1\setminus \{v\})$.
{\rm (2)}
If $\Delta_1\setminus\{v\}$ does not span $L\tensor \R$,
then $(v)\sperp$ is a wall of $D$.
\end{lemma}
\begin{proof}
 The `only if' part of  (1) is obvious by the assumption on $\Delta_1$.
Conversely, suppose that $\Sigma_L(\Delta_1) \ne \Sigma_L(\Delta_1\setminus \{v\})$.
 Then 
 there exists a vector  $x_0\in L\tensor\R$ such that 
$\intM{v, x_0}{}<0$ and $x_0\in \Sigma_L(\Delta_1\setminus \{v\})$.
 Assume  that $(v)\sperp$ is not a wall of $D$.
Then  $D$ is equal to  $\Sigma_L(\Delta_1\setminus \{v\})\cap \PPP_L$,
and hence $\Sigma_L(\Delta_1\setminus \{v\})$ is contained in $\clPPP_L$
 by Lemma~\ref{lem:clSigma}.
 In particular, we have $x_0\in \clPPP_L$.
Let $y_0$ be a point of $D\spcirc$.
Then the intersection point  $z_0$ of the line segment $\lineseg{x_0y_0}$
and $(v)\sperp$ satisfies $z_0^2> 0$, $\intM{v, z_0}{}=0$ and $\intM{v\sprime, z_0}{}>0$
for any $v\sprime$ in $\Delta_1\setminus \{v\}$.
Since $\VVV\sphyp$ is locally finite in $\PPP_L$, 
these  mean that a sufficiently small open neighborhood of $z_0$ in $(v)\sperp$ is 
contained in $D$.
In anyway,
$(v)\sperp$ is a wall of $D$.
Thus (1) is proved.
If $(v)\sperp$ is not a wall of $D$, then $\Delta_1\setminus\{v\}$ is also a defining set of $D$.
Hence (2) follows from  Lemma~\ref{lem:span}.
 \end{proof}
 \begin{algorithm}\label{algorithm:mindefset}
Let $\Delta$ be a \emph{finite} defining set of  a $\VVV\sphyp$-chamber $D$.
This algorithm calculates  the primitively minimal defining set $\Delta_{L\dual}(D)$ of $D$.

\step{0}. We set $\Delta_1:=\{\}$ and $\Delta_2:=\{\}$.

\step{1}.
For each element $v\in \Delta$,
we calculate the  maximal positive integer $a_v$ such that
$v/a_{v}\in L\dual$,
and append $v/a_{v}$ to the set $\Delta_1$.
Then we have $D=\Sigma_L(\Delta_1)\cap \PPP_L$,
and moreover, if $v, v\sprime\in \Delta_1$ are distinct,
then $(v)\sperp\ne (v\sprime)\sperp$ holds.

\step{2}.
For each $v\in \Delta_1$, we carry out the following computation.
Suppose that $\Delta_1\setminus\{v\}$  does not span $L\tensor \R$.
Then $(v)\sperp$ is a wall of $D$  by Lemma~\ref{lem:iswall},
and we append $v$ to $\Delta_2$.
Suppose that $\Delta_1\setminus\{v\}$  spans $L\tensor \R$.
%
%Then 
%we can solve the following  problem of linear programming
%on the vector space $L\tensor\Q$:
Then 
we can solve the following  problem of linear programming,
in which the variable $x$ ranges through  the vector space $L\tensor\Q$:
\begin{equation}\label{eq:LP}
\begin{cases}
\textrm{minimize} &  \intM{v, x}{}\\
\textrm{subject to} & \intM{v\sprime, x}{}\ge 0 \;\; \textrm{for all}\;\; v\sprime \in \Delta_1\setminus \{v\}.
\end{cases}
\end{equation}
(See, for example,  Chv\'atal~\cite{MR717219} for the algorithms  of linear programming.)
Note that the solution is either $0$ or unbounded to $-\infty$.
If the solution is $0$, then $(v)\sperp$ is not a wall of $D$
by Lemma~\ref{lem:iswall}.
Suppose that the solution is unbounded to $-\infty$.
Then there exists a vector $x_0\in L\tensor\R$ such that 
$\intM{v, x_0}{}<0$ and $x_0\in \Sigma_L(\Delta_1\setminus \{v\})$.
Hence $(v)\sperp$ is a wall of $D$ by Lemma~\ref{lem:iswall}, and 
we append $v$ to the set $\Delta_2$.

\step{3}. We then output $\Delta_2$ as $\Delta_{L\dual}(D)$.
 \end{algorithm}
Remark that,
for any $\VVV\sphyp$-chamber $D$,
the primitively minimal defining set  $\Delta_{L\dual}(D)$ 
 is finite by the property $\propVthree$ of  $\VVV$ and Algorithm~\ref{algorithm:mindefset}.
 We use the obvious brute-force method based on the finiteness of $\Delta_{L\dual}(D)$ in the following two algorithms.
 \begin{algorithm}\label{algorithm:autG}
 Suppose that the primitively minimal defining set 
 $\Delta_{L\dual}(D)$ of a $\VVV\sphyp$-chamber $D$ is given.
 This algorithm calculates all elements of $\autG(D)$.
Let $\Delta_{L\dual}(D)^l$ denote 
the set of ordered $l$-tuples of distinct elements of $\Delta_{L\dual}(D)$,
 where $l:=\rank  L$.
 By Lemma~\ref{lem:span},
 there exists an $l$-tuple $[v_1, \dots, v_l]\in \Delta_{L\dual}(D)^l$
 that forms a basis of $L\tensor \Q$.
 We set $A:=\{\}$.
 For each $[v_1\sprime, \dots, v_l\sprime]\in \Delta_{L\dual}(D)^l$,
 we calculate the linear transformation $g$ of $L\tensor \Q$ such that
 $$
 v_i^g=v_i\sprime \qquad (i=1, \dots, l).
 $$
 Recall that we can determine whether $g\in G$ or not by the assumption $\propVG$.
 If $g$ belongs to $G$ and induces a permutation of $\Delta_{L\dual}(D)$, then we append $g$ to $A$.
 When this calculation is done for all $[v_1\sprime, \dots, v_l\sprime]\in \Delta_{L\dual}(D)^l$,
 the set $A$ is equal to $\autG(D)$.
 \end{algorithm}
 \begin{proof}[Proof of Theorem~\ref{thm:finiteautG}]
Since $\Delta_{L\dual}(D)$ is finite, 
Algorithm~\ref{algorithm:autG} terminates in  finite steps.
\end{proof}
 \begin{algorithm}\label{algorithm:Gequiv}
 Let $D$ and $D\sprime$ be  $\VVV\sphyp$-chambers.
 Suppose that $\Delta_{L\dual}(D)$ and $\Delta_{L\dual}(D\sprime)$ are given.
 This algorithm determines whether $D$ is $G$-congruent to  $D\sprime$ or not.
We fix  an element $[v_1, \dots, v_l]$  of $\Delta_{L\dual}(D)^l$
 that forms a basis of $L\tensor \Q$.
 For each $[v_1\spprime, \dots, v_l\spprime]\in \Delta_{L\dual}(D\sprime)^l$,
 we calculate the linear transformation $g$ of $L\tensor \Q$ that satisfies 
 $ v_i^g=v_i\spprime$  for $i=1, \dots, l$.
 If $g$ belongs to $G$ and 
induces a bijection from  $\Delta_{L\dual}(D)$ to $\Delta_{L\dual}(D\sprime)$,
 then $D$ and $D\sprime$ are $G$-congruent.
 If no such $[v_1\spprime, \dots, v_l\spprime]$ are found, then $D$ and $D\sprime$ are not $G$-congruent.
 \end{algorithm}
\begin{remark}\label{rem:intpt}
Suppose that  $p$ is a point  in the interior of $D$.
By~\eqref{eq:intpt}, we see that
an element  $g\in G$ is contained in $\autG(D)$  if and only if
$p^g\in D$, 
and $g\in G$ induces an isomorphism from $D$ to $D\sprime$ if and only if $p^g\in D\sprime$. 
\end{remark}
\section{Vinberg-Conway theory}\label{sec:Conway}
Let $n$ be $10$, $18$ or $26$.
Throughout this section,
we denote 
by  $\L$  an even unimodular hyperbolic lattice $\mathrm{II}_{1, n-1}$ of rank $n$.
Note that  $\L$ exists and is unique up to isomorphism
(see, for example, ~\cite[Chapter V]{MR0344216}).
We fix a positive cone $\PPP_{\L}$ of $\L$.
Vinberg~\cite{MR0422505} and Conway~\cite[Chapter 27]{MR1662447} described the structure of $\Roots_{\L}\sphyp$-chambers;
that is, the standard fundamental domains of the action of $W(\L)$ on $\PPP_{\L}$ %ok
(see Example~\ref{example:W}).
\begin{definition}\label{def:LLLweylvector}
Let $\DDD$ be an $\Roots_{\L}\sphyp$-chamber.
We say that a vector $w\in \L$  is a \emph{Weyl vector of  $\DDD$}
if the $\Roots_{\L}$-minimal defining set $\Delta_{\Roots_{\L}}(\DDD)$ of $\DDD$ is given by 
\begin{equation*}\label{eq:DeltaL}
\Delta_{\Roots_{\L}}(\DDD)=\set{r\in \Roots_{\L}}{\intM{w, r}{}=1}.
\end{equation*}
\end{definition}
If a Weyl vector of an $\Roots_{\L}\sphyp$-chamber $\DDD$ exists,
then it is unique, because, as will be  shown below,
$\Delta_{\Roots_{\L}}(\DDD)$ spans $\L\tensor\R$.
If $w$ is the Weyl vector of $\DDD$,
then $w^g$ is the Weyl vector of $\DDD^g$ for any $g\in \OG^+(\L)$.
Since  any two  $\Roots_{\L}\sphyp$-chambers   are $\OG^+(\L)$-congruent,
the Weyl vector  of a single $\Roots_{\L}\sphyp$-chamber
gives  Weyl vectors  of all $\Roots_{\L}\sphyp$-chambers via the action of $\OG^+(\L)$.
%
%\begin{theorem}[Conway~\cite{MR690711}]\label{thm:Weyl}
\begin{theorem}[Conway, Chapter 27 of~\cite{MR1662447}]\label{thm:Weyl}
For any   $\Roots_{\L}\sphyp$-chamber $\DDD$, 
there exists a Weyl vector  $w\in \L$ of $\DDD$.
We have 
$$
w^2=\begin{cases}
1240 & \textrm{if $n=10$} \\
620 & \textrm{if $n=18$} \\
0 & \textrm{if $n=26$} \\
\end{cases}
\;\;\quand\;\;
|\Delta_{\Roots_{\L}}(\DDD)|=\begin{cases}
10 & \textrm{if $n=10$} \\
19 & \textrm{if $n=18$} \\
\infty  & \textrm{if $n=26$}. \\
\end{cases}
$$
\end{theorem}
\begin{remark}
The finite sets $\Delta_{\Roots_{\L}}(\DDD)$ for $n=10$ and $18$ had been  calculated by Vinberg~\cite{MR0422505}.
\end{remark}
We give an explicit description of the Weyl vectors and the set $\Delta_{\Roots_{\L}}(\DDD)$, 
and prove the following:
\begin{proposition}\label{prop:bdrQ}
If $\DDD$ is an $\Roots_{\L}\sphyp$-chamber,
then  $\pi_{\L}(\closure{\DDD})\cap \bdrclH_{\L}$ is contained in $\bdrclHQ_{\L}$.
%then  $\pi_{\L}(\closure{\DDD})\cap \bdrclH_{\L}\subset \bdrclHQ_{\L}$.
\end{proposition}
In the following,
we  denote by $U$ the even hyperbolic lattice of rank $2$ with 
a fixed basis $f_U, z_U$, with respect to which the Gram matrix is 
\begin{equation}\label{eq:nonstandardU}
\left[
\begin{array}{cc}
0 & 1 \\
1 & -2 
\end{array}
\right],
\end{equation}
and  by  $f_U\dual, z_U\dual$ the basis of $U$ dual to $f_U, z_U$.
\begin{remark}
We choose this non-standard basis of $U$ for geometric reasons;
$f_U$ will be the class of a fiber of a Jacobian fibration on a $K3$ surface
and $z_U$ will be the class of the zero section.
See Sections~\ref{sec:rho3} and~\ref{sec:singK3}.
\end{remark}
Let $E_8$ denote  the (negative-definite) root lattice of type $E_8$
with 
the standard  basis $e_1, \dots, e_8$,  whose Coxeter graph is 
\begin{equation*}\label{eq:E8}
\def\ha{40}
\def\hav{37}
\def\hd{25}
\def\hdv{22}
\def\he{10}
\def\hev{7}
\setlength{\unitlength}{1.5mm}
\centerline{
{\small
\begin{picture}(100, 9)(-25, 7)
\put(22, 16){\circle{1}}
\put(23.5, 15.5){$e\sb 1$}
\put(22, 10.5){\line(0,1){5}}
\put(9.5, \hev){$e\sb 2$}
\put(15.5, \hev){$e\sb 3$}
\put(21.5, \hev){$e\sb 4$}
\put(27.5, \hev){$e\sb 5$}
\put(33.5, \hev){$e\sb 6$}
\put(39.5, \hev){$e\sb 7$}
\put(45, \hev){$e\sb {8}$}
%\put(50, \hev){$f-\theta$}
\put(10, \he){\circle{1}}
\put(16, \he){\circle{1}}
\put(22, \he){\circle{1}}
\put(28, \he){\circle{1}}
\put(34, \he){\circle{1}}
\put(40, \he){\circle{1}}
\put(46, \he){\circle{1}}
%\put(52, \he){\circle{1}}
\put(10.5, \he){\line(5, 0){5}}
\put(16.5, \he){\line(5, 0){5}}
\put(22.5, \he){\line(5, 0){5}}
\put(28.5, \he){\line(5, 0){5}}
\put(34.5, \he){\line(5, 0){5}}
\put(40.5, \he){\line(5, 0){5}}
\put(50, \hev){.}
%\put(46.5, \he){\line(5, 0){5}}
%\put(66, \hev){$f:=f_U$}
\end{picture}
}
}
\end{equation*}
We denote by  $e_1\dual, \dots, e_8\dual$
the basis of $E_8$ dual to $e_1, \dots, e_8$.
We put
$$
\theta:=3e_1+2e_2+4e_3+6e_4+5e_5+4e_6+3e_7+2e_8.
$$
\subsection{The case where $n=10$.}\label{subsec:10}
We put
$\L:=U\oplus E_8$,
and choose $\PPP_{\L}$ in such a way that
$2f_U+z_U\in \PPP_{\L}$.
%Let $f_U\dual, z_U\dual, e_1\dual, \dots, e_8\dual$
%be the dual basis of the basis $f_U, z_U, e_1, \dots, e_8$ of $\L$.
Then the vector
$$
w_0:=30 f_U\dual +z_U\dual +e_1\dual+ \cdots+ e_8\dual
$$
 is the Weyl vector of an $\Roots_{\L}\sphyp$-chamber $\DDD_0$.
 By Vinberg~\cite{MR0422505}, 
we have 
$$
\Delta_{\Roots_{\L}}(\DDD_0)=\{z_U, e_1, \dots, e_8, f_U-\theta\},
$$
%where $\theta:=3e_1+2e_2+4e_3+6e_4+5e_5+4e_6+3e_7+2e_8$.
and
$$
\pi_{\L}(\clDDD_0)\cap \bdrclH_{\L}=\{\pi_{\L}(f_U)\}\subset \bdrclHQ_{\L}.
$$
\subsection{The case where  $n=18$.}\label{subsec:18}
We put 
$\L:=
U\oplus E_8\oplus E_8$,
and choose $\PPP_{\L}$ in such a way that
$2f_U+z_U\in \PPP_{\L}$.
Let $e_1\sprime, \dots, e_8\sprime$ be the basis of the second $ E_8$
with the same Coxeter graph as  $e_1, \dots, e_8$.
%, and 
%let  $f_U\dual, z_U\dual$, $e_1\dual, \dots, e_8\dual$,  $e_1^{\prime\vee}, \dots, e_8^{\prime\vee}$
%be the dual basis of the basis $f_U, z_U$, $e_1, \dots, e_8$,  $e_1\sprime, \dots, e_8\sprime$
%let  $f_U\dual, z_U\dual$, $e_1\dual, \dots,  e_8^{\prime\vee}$
%be the dual basis of the basis $f_U, z_U$, $e_1, \dots,  e_8\sprime$
%of $\L$.
Then the vector
$$
w_0:=
30 f\dual_U +z\dual_U +e_1\dual+ \cdots+ e_8\dual+e_1^{\prime\vee}+\cdots+ e_8^{\prime\vee}
$$
is the Weyl vector of an $\Roots_{\L}\sphyp$-chamber $\DDD_0$.
%Let $\DDD_0$ be the $\Roots_{\L}\sphyp$-chamber corresponding to $w_0$.
 By Vinberg~\cite{MR0422505}, 
 we have 
$$
\Delta_{\Roots_{\L}}(\DDD_0)=
\{z_U, e_1, \dots, e_8, e_1\sprime, \dots, e_8\sprime, f_U-\theta, f_U-\theta\sprime\},
$$
where  $\theta\sprime$ is 
defined in the same way as $\theta$ with $e_i$ replaced by $e_i\sprime$.
Moreover, we have
$$
\pi_{\L}(\clDDD_0)\cap \bdrclH_{\L}=\{\pi_{\L}(f_U), \pi_{\L}(v_1) \}\subset \bdrclHQ_{\L},
$$
where
$$
v_1:=e_1+e_3+e_1\sprime+e_3\sprime+
2 (z_U+(f_U-\theta)+(f_U-\theta\sprime)+e_4+\cdots+e_8+e_4\sprime+\cdots+e_8\sprime).
$$
\subsection{The case where  $n=26$}
We denote by $\Lambda$ the \emph{negative-definite} Leech lattice;
that is,  $\Lambda$ is an even unimodular negative-definite lattice of rank $24$ such that
$\Roots_{\Lambda}=\emptyset$.
We put $\L:=U\oplus \Lambda$, 
and  choose $\PPP_{\L}$ in such a way that  $2f_U+z_U\in \PPP_{\L}$.
\begin{theorem}[Conway and Sloane, Chapter 26 of~\cite{MR1662447}]\label{thm:ConwaySloane}
A vector $w\in \L$ is a Weyl vector 
of some $\Roots_{\L}\sphyp$-chamber
if and only if
$w$ is a non-zero primitive vector of norm $0$
contained in $\clPPP_{\L}$ such that 
$\gen{w}\sperp/\gen{w}$ is isomorphic to $\Lambda$,
where $\gen{w}\sperp$ is the orthogonal complement of the primitive submodule $\gen{w}$ of $\L$.
\end{theorem}
Therefore
$w_0:=f_U$
is the Weyl vector of an $\Roots_{\L}\sphyp$-chamber 
$\DDD_0$.
%Let $\DDD_0$ be the  $\Roots_{\L}\sphyp$-chamber corresponding to $w_0$.
For simplicity, we denote vectors of $\L\tensor\R=(U\oplus\Lambda)\tensor\R$ by 
$$
[s, t, y]:=s f_U +t z_U +y, \quad\textrm{where}\quad s, t\in \R,\;\; y\in \Lambda\tensor\R.
$$
Then we have
$$
\Delta_{\Roots_{\L}}(\DDD_0)=\set{r_{\lambda}}{\lambda\in \Lambda},
\quad\textrm{where}\quad
%r_{\lambda}:=\left[-\frac{\lambda^2}{2}, 1, \lambda\right].
r_{\lambda}:=\left[-\lambda^2/2, 1, \lambda\right].
$$
Conway, Parker and Sloane~\cite{MR660415} (\cite[Chapter 23]{MR1662447})
proved that the covering radius of the Leech lattice is $\sqrt{2}$
(see also Borcherds~\cite{MR791880}).
A point $c$ of  $\Lambda\tensor\R$ is called a \emph{deep hole}
if $-(c-\lambda)^2\ge 2$ holds for  any  $\lambda\in \Lambda$.
In Lemma 4.4 of~\cite{MR913200},
Borcherds observed the following:
\begin{lemma}\label{lem:property_d_L_26}
Let $b$ be a point of $\pi_{\L} (\closure{\DDD}_0)\cap\bdrclH_{\L}$.
Then either $b=\pi_{\L}(w_0)$ or there exists a deep hole
$c$ such that
$$
b=\pi_{\L}(v_c), \quad\textrm{where}\quad
%v_c:=\left[-\frac{c^2}{2}+1, 1, c\right].
v_c:=\left[-{c^2}/{2}+1, 1, c\right].
$$
In particular,
the set $\pi_{\L} (\closure{\DDD}_0)\cap\bdrclH_{\L}$
is contained in $\bdrclHQ_{\L}$.
\end{lemma}
\begin{proof}
Note that $\clPPP_{\L}$ is contained in $\shortset{x\in \L\tensor\R}{\intM{w_0, x}{}\ge 0}$,
and that $x\in \clPPP_{\L}$
satisfies $\intM{x, w_0}{}=0$ if and only if $x$ belongs to the half-line $\R_{\ge 0} w_0$.
%\cap [w_0]\sperp$ is the half-line $[w_0]_+$.
%$[w_0]_+=\shortset{\alpha w_0}{\alpha\in \R_{\ge 0}}$,
%where $[w_0]\sperp=\shortset{x\in \L\tensor\R}{\intM{w_0, x}{}=0}$.
Suppose that $b=\pi_{\L}(u)$,
where $u=[s, t, y]$ is a non-zero vector of norm $0$
in   $\closure{\DDD}_0$.
Since $u\in \clPPP_{\L}$, we have $t=\intM{w_0, u}{}\ge 0$,
and $t=0$ holds if and only if $\R_{\ge 0} w_0= \R_{\ge 0} u$.
Hence $t=0$ implies $b=\pi_{\L}(w_0)$.
Suppose that $t>0$.
We can assume that $t=1$.
Since $u^2=0$, we have $2s-2+y^2=0$.
Since $u\in \closure{\DDD}_0$,
we have
$$
\intM{u, r_{\lambda}}{}=-\frac{(y-\lambda)^2}{2}-1\ge 0\;\;\textrm{for any}\;\; \lambda\in \Lambda.
$$
Therefore $y\in \Lambda\tensor\R$ is a deep hole $c$ and $u=v_c$ holds.
Let $p_1, \dots, p_N$ be the points of $\Lambda$ nearest to $c$;
that is, $p_1, \dots, p_N$  are the points of $\Lambda$ satisfying  $(p_i-c)^2=-2$.
Then their differences $p_i-p_j$ span $\Lambda\tensor\Q$,
and $c$ is the intersection point of the bisectors of
distinct  two points of $p_1, \dots, p_N$.
Hence $c$ belongs to  $\Lambda\tensor\Q$, and we have   $\pi_{\L}(u)\in \bdrclHQ_{\L}$.
%
%Since 
%every deep hole  belongs to  $\Lambda\tensor\Q$ 
%by  Conway, Parker and Sloane~\cite[Chapters 23]{MR1662447} %~\cite{MR660415}
%or Conway and Sloane~\cite{MR661720} (\cite[Chapters 23 and  24]{MR1662447}),
%or Conway and Sloane~\cite[Chapters 24]{MR1662447}, 
%we have   $\pi_{\L}(u)\in \bdrclHQ_{\L}$.
\end{proof}
\begin{remark}\label{rem:deepholes}
The coordinates of  deep holes of $\Lambda$ are explicitly given  in 
 Conway, Parker and Sloane~\cite{MR660415}.
\end{remark}
\begin{remark}\label{rem:UE83}
We have an isomorphism 
$$
\L\cong  U\oplus E_8 \oplus E_8 \oplus E_8.
$$
The vector $w_E\in U\oplus E_8^3$ given by
$[30, 1, 1^8, 1^8, 1^8]$ 
in terms of the dual basis $f\dual, z\dual$, $e_1\dual, \dots, e_8^{\prime\prime\vee}$
is a Weyl vector of $\L$.
Indeed, let $u\in U\oplus E_8^3$ be the vector given by
$[1, -2, 0^8, 0^8, 0^8]$ 
in terms of the dual basis above.
Since $w_E^2=0$ and $\intM{w_0, u}{}=1$,
the vectors $w_E$ and $u$ span an even  hyperbolic unimodular lattice of rank $2$,
and its orthogonal complement  in $U\oplus E_8^3$
is isomorphic to $\gen{w_E}\sperp/\gen{w_E}$.
Calculating a Gram matrix of this orthogonal complement 
and using Algorithm~\ref{algo:QLc}, we  confirm that $\gen{w_E}\sperp/\gen{w_E}$ has no $(-2)$-vectors,
and therefore is isomorphic to $\Lambda$.
\end{remark}
\section{Generalized Borcherds' method}\label{sec:Borcherds}
Suppose that $S$ is an even hyperbolic lattice
with a fixed positive cone $\PPP_S$,
and let $G$ be  a subgroup of $\OG\sp+(S)$ with finite index
satisfying the existence of membership algorithm~$\propVG$ in Section~\ref{sec:chamber} with $L$ replaced by $S$.
We present an algorithm that calculates a set of generators of
$\aut_G(\Nc)=\shortset{g\in G}{\Nc^g=\Nc}$ for a given $\RRR_S\sphyp$-chamber $\Nc$
under the assumptions~$\assumpone$,~$\assumptwo$ and~$\assumpthree$ below.
\par
%\medskip
We recall the definition of the discriminant form of an even lattice.
See Nikulin~\cite{MR525944} for the details.
For an even lattice $L$,
the \emph{discriminant group} $\discgr{L}:=L\dual/L$ is  equipped with a non-degenerate quadratic form
$$
q_L \colon \discgr{L}\to \Q/2\Z, \qquad x \bmod L\mapsto x^2 \bmod 2\Z,
$$
 which is called the \emph{discriminant form of $L$}.
 Let $\OG(q_L)$ denote the group of automorphisms of $(\discgr{L}, q_L)$.
We have a natural homomorphism
$$
\eta_L \colon \OG(L)\to \OG(q_L).
$$
For an isomorphism $\delta\colon (A_1, q_1)\isom (A_2, q_2)$
of discriminant forms,
we denote by
$$
\delta_*\colon \OG(q_1)\isom \OG(q_2)
$$
the induced isomorphism on the automorphism groups.
\par
%\medskip
Let $n$ be $10$, $18$ or $26$.
As in Section~\ref{sec:Conway}, 
we denote by  $\L$  an even unimodular hyperbolic lattice of rank $n$.
We assume that $S$ satisfies the following embeddability  condition:
\begin{itemize}
\item[$\assumpone$]\quad 
\textrm{$S$ is primitively embedded into $\L$.}
\end{itemize}
Let $R$ denote the orthogonal complement of $S$ in ${\L}$.
Then $R$ is an even negative-definite lattice.
We assume that
\begin{itemize}
\item[$\assumptwo$]
\quad if $n=26$, the lattice $R$ cannot be embedded into the Leech lattice $\Lambda$.
\end{itemize}
For example, if $\Roots_R$ is non-empty, then $\assumptwo$ is satisfied,
because $\Roots_{\Lambda}=\emptyset$.
In fact,
in all examples that are treated in this paper, 
the assumption $\assumptwo$ is verified  by showing $\Roots_R\ne \emptyset$.
\par
%\medskip
We denote by
$$
x\mapsto x_S \quand x\mapsto x_R
$$
the orthogonal projections from ${\L}\tensor\R$ to $S\tensor\R$ and $R\tensor\R$,
respectively.
Since ${\L}$ is contained in $S\dual\oplus R\dual$,
the images of ${\L}$ by these projections are contained in $S\dual$ and $R\dual$,
respectively.
Since ${\L}$ is unimodular, the result of Nikulin~\cite[Proposition 1.6.1]{MR525944} implies that
the subgroup ${\L}/(S\oplus R)$ of $A_S\oplus A_R$ is the graph of an isomorphism
$$
\delta_{\L} \colon (A_S, q_S)\isom (A_R, -q_R).
$$
We assume that the subgroup $G$ of $\OG^+(S)$ satisfies 
 the following liftability condition~(see Proposition~\ref{prop:liftg} below):
\begin{itemize}
\item[$\assumpthree$]\quad
$\delta_{\L*}(\eta_S(G))\subset \Image \eta_R$, \;\; 
\textrm{where $\delta_{\L*}\colon \OG(q_S)\isom \OG(q_R)$ is induced by $\delta_{\L}$}.
\end{itemize}
For example, if $G$ is contained in $\eta_S\inv(\{\pm 1\})$, then~$\assumpthree$ is satisfied.
\par
%\medskip
%
Let $\PPP_{\L}$ be 
 the positive cone of $\L$ that contains the fixed positive cone $\PPP_S$ of $S$.
Let $r$ be an element of $\Roots_{\L}$.
Then the hyperplane $(r)\sperp\in \HHH_{\L}$ of $\PPP_{\L}$ intersects $\PPP_S$ if and only if
$r_S^2<0$ holds,
and in this case,
the hyperplane $(r_S)\sperp\in  \HHH_S$ of $\PPP_S$ is equal to the intersection $\PPP_S\cap (r)\sperp$.
We put
$$
\RtsLS:=\set{r_S}{r\in \Roots_{\L},\;\; r_S^2<0}, 
$$
and show that the subset $\RtsLS$ of $\NNN_S\cap S\dual$ has the properties~$\propVone$-$\propVfour$
given in Section~\ref{sec:chamber} with $L$ replaced by $S$.
\begin{proposition}
If $v\in  \RtsLS$,  then $-v^2\le 2$.
In particular, $\RtsLS$ satisfies $\propVone$.
\end{proposition}
\begin{proof}
Since  $R$ is negative-definite
and $r_S^2+r_R^2=-2$ holds for any $r\in \Roots_{\L}$, 
we have $-2\le v^2$ for any $v=r_S\in \RtsLS$.
\end{proof}
By Lemma~\ref{lem:locallyfinite},  the family of hyperplanes
\begin{eqnarray*}
\RtsLS\sphyp &=&\set{(r_S)\sperp\in \HHH_S}{r\in \Roots_{\L},\;\; r_S^2<0} \\
&=&\set{\PPP_S\cap (r)\sperp}{r\in \Roots_{\L},\;\; \PPP_S\cap (r)\sperp\ne \emptyset}
\end{eqnarray*}
is locally finite in $\PPP_S$.
\begin{proposition}\label{prop:liftg}
Let $g$ be an element of $G$.
Then there exists an element  $\tilde {g}\in \OG\sp+({\L})$
that leaves $S$ invariant and is equal to $g$ on $S$.
\end{proposition}
\begin{proof}
By~$\assumpthree$,
there exists an element $h\in \OG(R)$ such that $\delta_{\L*}(\eta_S(g))=\eta_R(h)$.
Then the action of $(\eta_S(g), \eta_R(h))$ on $A_S\oplus A_R$ preserves 
the graph ${\L}/(S\oplus R)$  of $\delta_{\L}$,
and hence the action of $(g, h)$ on $S\dual\oplus R\dual$ preserves 
 ${\L}\subset S\dual\oplus R\dual$.
The restriction $\tilde {g}\in \OG({\L})$  of $(g, h)$ to ${\L}$
belongs to  $\OG\sp+({\L})$,
because $\tilde{g}$ leaves  $\PPP_S\subset \PPP_{\L}$ invariant.
Thus we obtain  a desired lift $\tilde {g}\in \OG\sp+({\L})$. 
\end{proof}
\begin{proposition}
The action of $G$ on $\NNN_S\cap S\dual$ leaves 
$\RtsLS$  invariant;  that is, 
$\RtsLS$ satisfies $\propVtwo$.
\end{proposition}
\begin{proof}
For any $v\in \RtsLS$ and $g\in G$,
we have $r\in \Roots_{\L}$ such that $v=r_S$
and a lift $\tilde {g}\in \OG^+({\L})$ of $g$.
Then  $v^g=(r^{\tilde g})_S\in \RtsLS$ holds.
\end{proof}
By the definition of $\RtsLS$, 
every $\RtsLS\sphyp$-chamber $D$ is written as 
$$
D=\DDD\cap\PPP_S
$$
by some $\Roots_{\L}\sphyp$-chamber $\DDD$.
\begin{definition}\label{def:Snondeg}
An $\Roots_{\L}\sphyp$-chamber $\DDD$
is said to be \emph{$S$-nondegenerate} if  $D=\DDD\cap\PPP_S$
is an $\RtsLS\sphyp$-chamber.
In other words,
$\DDD$
is $S$-nondegenerate
if and only if
$\DDD\cap\PPP_S$ contains a non-empty open subset of $\PPP_S$.
\end{definition}
\begin{definition}\label{def:Sweylvector}
Let $D$ be an $\RtsLS\sphyp$-chamber, 
and let $\DDD$ be an $\Roots_{\L}\sphyp$-chamber such that $D=\DDD\cap\PPP_S$.
By certain abuse of terminology,
we say that the Weyl vector $w\in {\L}$ of  $\DDD$ is 
a \emph{Weyl vector of $D$}.
\end{definition}
\begin{remark}\label{rem:DDDs}
The set $\shortset{r\in \Roots_{\L}}{\PPP_S\subset (r)\sperp}$
is equal to $\Roots_{R}$.
In particular, the number of hyperplanes $(r)\sperp\in \Roots_{\L}\sphyp$
passing through a point $v\in \PPP_S$ is at least $|\Roots_{R}|/2$,
and is equal to $|\Roots_{R}|/2$ if $v$ is general.
We remark two consequences of this fact.
\par
Let $D$ be an $\RtsLS\sphyp$-chamber.
Then there exists a canonical one-to-one correspondence
between the set of $\Roots_{\L}\sphyp$-chambers $\DDD$ satisfying  $D=\DDD\cap\PPP_S$
and the set of connected components of
$$
(R\tensor\R) \setminus \bigcup_{\rho\in \Roots_R} [\rho]\sperp,
\quad \textrm{where}\quad
[\rho]\sperp:=\shortset{x\in R\tensor\R}{\intM{x, \rho}{R}=0}.
$$
If $\Roots_{\L}\sphyp$-chambers $\DDD$ and $\DDD\sprime$ satisfy $D=\DDD\cap\PPP_S=\DDD\sprime\cap\PPP_S$,
then there exists a sequence of reflections $s_i$ ($i=1, \dots, N$)
of $\L$ with respect to $\rho_i\in \Roots_R\subset \Roots_{\L}$ such that their product 
$s_1\cdots s_N$ maps $\DDD$ to $\DDD\sprime$.
Then the Weyl vector $w$ of $\DDD$ is mapped to the Weyl vector $w\sprime$ of $\DDD\sprime$
by $s_1\cdots s_N$.
Since each $(-2)$-vector $\rho_i$ is contained in $R$,
we  have $w_S=w\sprime_S$.
Therefore,
for an $\RtsLS\sphyp$-chamber $D$,
 $w_S\in S\dual$ is independent of the choice of a Weyl vector $w$ of $D$.
%
%Note that, however,  
%it is \emph{not} true in general that
%$w_S$ belongs to  $D$.
\par
Let $\DDD$ be an $\Roots_{\L}\sphyp$-chamber,
and let
 $v$ be a vector in $\DDD\cap \PPP_S$.
If the number of walls of  $\DDD$ 
 passing though $v$ is equal to  
$|\Roots_{R}|/2$,
then a small neighborhood of $v$ in $\PPP_S$ is  contained in $\DDD\cap \PPP_S$,
and hence $\DDD$ is $S$-nondegenerate,
and $D=\DDD\cap \PPP_S$ contains $v$ in its interior.
\end{remark}
Next we consider the property~$\propVthree$ for $\RtsLS$.
Let $w$ be a Weyl vector of $\L$,
and let $\DDD$ be the corresponding $\Roots_{\L}\sphyp$-chamber
(not necessarily $S$-nondegenerate).
Recall that the  $\Roots_{\L}$-minimal defining set
$\Delta_{\Roots_{\L}}(\DDD)$ of $\DDD$ is equal to $\shortset{r\in \Roots_{\L}}{\intM{w, r}{\L}=1}$.
We put  
\begin{equation}\label{eq:Deltasprime} 
\Delta_{w}:=\set{r\in \Delta_{\Roots_{\L}}(\DDD)}{r_S^2<0}=\set{r\in \Delta_{\Roots_{\L}}(\DDD)}{(r)\sperp\cap\PPP_S\ne\emptyset}
\end{equation}
and
\begin{equation} 
\pr_{S}(\Delta_w):=\shortset{r_S}{r\in\Delta_{w}}.
\end{equation}
Then
we have $\DDD\cap\PPP_S=\Sigma_{S} (\pr_S(\Delta_w))\cap \PPP_S$.
Therefore, if $\DDD$ is $S$-nondegenerate, then $\pr_{S}(\Delta_w)$ is a defining set of 
the $\RtsLS\sphyp$-chamber $D=\DDD\cap\PPP_S$.
\begin{proposition}\label{prop:proerty_c}
For any Weyl vector $w$ in $\L$, the set $\Delta_{w}$ is finite.
In particular, any $\RtsLS\sphyp$-chamber $D$ has a finite defining set,
and hence $\RtsLS$ satisfies $\propVthree$.
\end{proposition}
\begin{proof}
First we show  $w_S^2>0$.
Note that
$w^2=w_S^2+w_R^2$ and $w_R^2\le 0$.
Hence $w_S^2>0$ holds when $n=10$ or $n=18$
by Theorem~\ref{thm:Weyl}.
Suppose that $n=26$ and $w_S^2=0$.
Then $w_R=0$ and $w=w_S$ hold.
Therefore $\gen{w}\sperp/\gen{w}$ would contain $R$,
which contradicts Theorem~\ref{thm:ConwaySloane} and the assumption~$\assumptwo$. 
\par
We denote by $d_R$ the order of the discriminant group $A_R$.
Then $d_Rv^2\in \Z$ and $d_R^2v^2\in2\Z$ hold for any $v\in R\dual$.
We put
\begin{equation}\label{eq:nR}
n_R:=\set{c\in \Q}{d_Rc\in \Z, \;\;d_R^2c\in 2\Z, \;\; -2< c\le 0},
\end{equation}
which is obviously finite.
Since $R$ is negative-definite, 
the set 
\begin{equation}\label{eq:Rdualc}
R\dual[c]:=\set{v\in R\dual}{v^2=c}
\end{equation}
is finite for each $c\in n_R$.
In particular, the set
\begin{equation}\label{eq:aRc}
a_R[c]:=\set{\intM{w_R, v}{R}}{v\in R\dual[c]}
\end{equation}
is finite for each $c\in n_R$. 
%\par
Suppose that $r\in\Delta_w$,  so that $r_S^2<0$.
%$r\in \Delta_{\Roots_{\L}}(\DDD)$ satisfies $r_S^2<0$.
Note that 
$$
r^2=r_S^2+r_R^2=-2\quand \intM{w, r}{{\L}}=\intM{w_S, r_S}{S}+\intM{w_R, r_R}{R}=1.
$$
Since $r_R^2\le 0$,
we have $-2\le r_S^2<0$ and $-2<r_R^2\le 0$.
Hence $n\sprime:=r_R^2$ belongs to  $n_R$,
and $r_R$ is an element of  $R\dual[n\sprime]$.
Since $w_S^2>0$,
the quadratic part of the inhomogeneous quadratic form $x\mapsto x^2$ on the affine hyperplane
$$
\set{x\in S\tensor\R}{\intM{w_S, x}{S}=b}
$$
of $S\tensor\R$
is negative-definite for any $b\in \R$.
 If we put $a\sprime:=\intM{w_R, r_R}{R} \in a_R[n\sprime]$, then  $r_S$ belongs to 
 the finite set
\begin{equation}\label{eq:Sdualna}
S\dual[n\sprime, a\sprime]:=\set{v\in S\dual}{\intM{w_S, v}{S}=1-a\sprime,\;\; v^2=-2-n\sprime}.
\end{equation}
Since $n_R$, $R\dual[n\sprime]$, $a_R[n\sprime]$  and $S\dual[n\sprime, a\sprime]$ are finite,
$\Delta_{w}$ is finite.
\end{proof}
We state the above proof in the form of an algorithm.
\begin{algorithm}\label{algorithm:Deltaw}
Suppose that a Weyl vector $w\in \L$  is given.
This algorithm calculates the set $\Delta_{w}$
defined by~\eqref{eq:Deltasprime}.
\par
\step{0}.
We calculate $w_S\in S\dual$, $w_R\in R\dual$,
and the set $n_R$ defined by~\eqref{eq:nR}.
We set $\Delta\sprime:=\{\}$.
\par
\step{1}.
For each $c\in n_R$, we calculate 
$R\dual[c]$ defined by~\eqref{eq:Rdualc} by  Algorithm~\ref{algo:QLc},
and calculate $a_R[c]$ defined by~\eqref{eq:aRc}.
\par
\step{2}.
Using Algorithm~\ref{algo:QLcab},
we calculate the finite set 
$S\dual[n\sprime, a\sprime]$
defined by~\eqref{eq:Sdualna}
for each pair of $n\sprime\in n_R$ and $a\sprime\in a_R[n\sprime]$.
\par
\step{3}.
For each triple 
$n\sprime\in n_R$, $v_R\in R\dual[n\sprime]$
and $v_S\in S\dual[n\sprime, a\sprime]$,  where $a\sprime=\intM{w_R, v_R}{R}$,
we determine whether $v_S+v_R \in S\dual\oplus R\dual$ belongs to  ${\L}$ or not,
and if $v_S+v_R\in {\L}$,
then we append $r:=v_S+v_R$ to $\Delta\sprime$.
\par
\step{4}.
Output $\Delta\sprime$ as $\Delta_w$.
\end{algorithm}
We give a criterion of the $S$-nondegeneracy of a given $\Roots_{\L}\sphyp$-chamber.
\begin{criterion}\label{criterion:Snondeg}
Let $\DDD$ be an $\Roots_{\L}\sphyp$-chamber with the Weyl vector $w$.
Then $\DDD$ is $S$-nondegenerate if and only if
there exists a vector  $v\in \PPP_S$ that satisfies the finite number of strict inequalities 
\begin{equation}\label{eq:ineqSnondeg}
\intM{v, r_S}{S}>0\quad\textrm{for any $r\in \Delta_w$}.
\end{equation}
If $v\in \PPP_S$ satisfies these inequalities,
then the  $\RtsLS\sphyp$-chamber $\DDD\cap\PPP_S$ contains $v$ in its interior.
\end{criterion}
Next we consider the property~$\propVfour$ for $\RtsLS$.
\begin{proposition}\label{prop:proerty_d}
For any  $\RtsLS\sphyp$-chamber $D$,
we have $\pi_S(\clD)\cap\bdr\clH_S\subset \bdrclHQ_S$.
In particular, $\RtsLS$ satisfies $\propVfour$.
\end{proposition}
\begin{proof}
Let  $\DDD$ be an $\Roots_{\L}\sphyp$-chamber such that $D=\DDD\cap\PPP_S$.
Then we have $\clD =\closure{\DDD}\cap\clPPP_S$.
Since $\bdrclHQ_S=\bdrclH_S\cap\bdrclHQ_{\L}$ under the canonical inclusion  $\clH_S\inj \clH_{\L}$,
the assertion follows from Proposition~\ref{prop:bdrQ}.
\end{proof}
Thus 
the subset $\RtsLS$ of $\NNN_S\cap S\dual$  has 
the properties~$\propVone$-$\propVfour$.
Hence we can apply Algorithms~\ref{algorithm:mindefset},~\ref{algorithm:autG}
and~\ref{algorithm:Gequiv} described in Section~\ref{sec:chamber}
to  $\VVV=\RtsLS$. 
\begin{algorithm}\label{algorithm:aSnS}
Suppose that a Weyl vector $w\in \L$ of 
an $\RtsLS\sphyp$-chamber $D$ is given.
This algorithm calculates 
the primitively minimal defining set $\Delta_{S\dual}(D)$ of $D$.
\par
\step{0}.
We calculate $\Delta_w$ by Algorithm~\ref{algorithm:Deltaw}.
\par
\step{1}.
We calculate the defining set $\pr_S(\Delta_{w})$ of $D$.
\par
\step{2}.
We then calculate  $\Delta_{S\dual}(D)$ from $\pr_S(\Delta_{w})$ by Algorithm~\ref{algorithm:mindefset}.
\end{algorithm}
\begin{remark}\label{rem:prSDelta}
Let $D$ and $D\sprime$ be $\RtsLS\sphyp$-chambers.
By Algorithms~\ref{algorithm:autG},~\ref{algorithm:Gequiv} and~\ref{algorithm:aSnS},
we can calculate $\aut_G(D)$ and determine whether $D$ and $D\sprime$ are $G$-congruent or not 
from Weyl vectors $w$ of $D$ and $w\sprime$ of $D\sprime$.
Note that, by Remark~\ref{rem:DDDs},
the defining set $\pr_S(\Delta_w)$ of $D$ is independent of the choice of the Weyl vector $w$.
Moreover, by Proposition~\ref{prop:liftg},
if $g\in \aut_G(D)$, then $g$ preserves $\pr_S(\Delta_w)$,
and if $D\sprime=D^g$, then
$g$ maps $\pr_S(\Delta_w)$ to $\pr_S(\Delta_{w\sprime})$ bijectively.
Hence, when we apply Algorithms~\ref{algorithm:autG} and~\ref{algorithm:Gequiv}
to $\RtsLS\sphyp$-chambers, 
we can use $\pr_S(\Delta_w)$ and $\pr_S(\Delta_{w\sprime})$
in the place of the primitively defining sets $\Delta_{S\dual}(D)$ and $\Delta_{S\dual}(D\sprime)$.
\end{remark}
Let $w$ be a  Weyl vector of 
an $\RtsLS\sphyp$-chamber $D$,
and let 
$v$ be an element of  $\Delta_{S\dual}(D)$.
Then $(v)\sperp$ is a wall of $D$.
We calculate the $\RtsLS\sphyp$-chamber that is adjacent to $D$ across $(v)\sperp$.
First we prepare an auxiliary algorithm.
\begin{algorithm}\label{algorithm:passing}
Suppose that $v\in \NNN_S\cap S\dual$ is given.
We regard the hyperplane $(v)\sperp$ of $\PPP_S$ as a linear subspace of the larger space $\PPP_{\L}$;
that is, 
$(v)\sperp$  is of codimension 
$\rank R+1$ in $\PPP_{\L}$, whereas $(r)\sperp$ for $r\in \Roots_{\L}$ is of codimension $1$ in $\PPP_{\L}$.
This algorithm calculates the set 
\begin{equation}\label{eq:Pv}
P_v:=\set{r\in \Roots_{\L}}{(v)\sperp\subset (r)\sperp}=\set{r\in \Roots_{\L}}{r_S\in \R v}.
\end{equation}
We set $P:=\{\}$.
There exist only a finite number of 
$\alpha\in \Q$ such that
$\alpha v \in S\dual$ and $\alpha^2 v^2\ge -2$.
For each such rational number $\alpha$
and
each $u\in R\dual[c]$, where   $c=-2-\alpha^2 v^2$
and $R\dual[c]$ defined by~\eqref{eq:Rdualc},
we determine whether $\alpha v+u\in S\dual\oplus R\dual$ belongs to ${\L}$ or not,
and if  $\alpha v+u\in {\L}$,
then we append $r:=\alpha v+u\in {\L}$ to $P$.
Then we output $P$ as $P_v$.
(Remark that $P_v$ includes the subset  $\Roots_{R}$ of $\Roots_{\L}$.)
\end{algorithm}
%
%
%We put $[v]:=\R v\subset L\tensor\R$,  and consider the linear subspace
We consider the linear subspace
$$
V:=\R v\oplus (R\tensor \R)
$$
of $\L\tensor\R$.
% is the $1$-dimensional linear subspace of ${\L}\tensor \R$ generated by $v$.
Let $\intM{\phantom{|}, \phantom{|}}{V}\colon V\times V\to \R$ denote 
the restriction of $\intM{\phantom{|}, \phantom{|}}{\L}$ to $V$.
Note that $\intM{\phantom{|}, \phantom{|}}{V}$ is negative-definite.
Let 
$x\mapsto x_V$ 
denote the orthogonal projection from ${\L}\tensor \R$ to $V$.
Each element $r$ of $P_v$ defined by~\eqref{eq:Pv} belongs to $V$.
%because $(r_S)\sperp=(v)\sperp$ implies that $r_S$ is a multiple of $v$.
Hence $r_V=r$ holds for any $r\in P_v$.
We denote by $\{\DDD_0, \dots, \DDD_m\}$
the set of  $\Roots_{\L}\sphyp$-chambers containing 
the hyperplane $(v)\sperp$ of $\PPP_S$ 
(that is, 
the linear subspace  of $\PPP_{\L}$
with codimension $\rank R+1$), 
and  put
$$
\DDD\spcirc_{j, V}:=
\set{x\in V}{\intM{x, r}{V}>0 \;\;\textrm{for any}\;\; r \in P_v\cap \Delta_{\Roots_{\L}}(\DDD_j)}.
$$
Then $\DDD_j\mapsto \DDD\spcirc_{j, V}$ gives
a one-to-one correspondence
from $\{\DDD_0, \dots, \DDD_m\}$ 
to the set of connected components
of
$$
V\;\setminus\; \bigcup_{r \in P_v}[r]\sperp_V, 
\;\;\;\;\textrm{where}\;\;\;\;
[r]_V\sperp:=\set{x\in V}{\intM{x, r}{V}=0}.
$$
Let $w_j$ be the Weyl vector of $\DDD_j$.
By renumbering $\DDD_0, \dots, \DDD_m$,
we can assume that
$D=\PPP_S\cap \DDD_0$ and  
that $w_0$ is the given Weyl vector $w$ of $D$.
Since 
$\intM{w_{j,V}, r}{V}=\intM{w_{j}, r}{\L}=1$
for any $r\in P_v\cap\Delta_{\Roots_{\L}}(\DDD_j)$,
the vector $w_{j, V}$ belongs to  $\DDD\spcirc_{j, V}$.
There exists an $\Roots_{\L}\sphyp$-chamber $\DDD_{\opp}$ among $\{\DDD_0, \dots, \DDD_m\}$ such that
$$
\DDD\spcirc_{\opp, V}=-\DDD\spcirc_{0, V}.
$$
Then 
$\PPP_S\cap \DDD_{\opp}$
is the $\RtsLS\sphyp$-chamber  $D\sprime$ adjacent to $D$ across $(v)\sperp$.
We calculate the Weyl vector of $\DDD_{\opp}$.
Let $u$ be a sufficiently general vector of ${\L}\tensor\Q$,
and let $\varepsilon$ be a sufficiently small positive real number.
Consider the oriented line segment
$$
p(t):=(1-t) w_V+t(-w_V+\varepsilon u_V)\qquad (0\le t\le 1)
$$
in $V$ from $w_{V}=w_{0, V}\in \DDD\spcirc_{0, V}$ 
to $-w_V+\varepsilon u_V\in \DDD\spcirc_{\opp, V}$.
Let 
$P\sprime_v=\{r_1, \dots, r_N\}\subset P_v$
be a complete set of representatives
of  $P_v/\{\pm 1\}$.
For $i=1, \dots, N$,
let $t_i$ be the value of $t$ such that $p(t_i) \in [r_i]_V\sperp$.
Since $\intM{x, r_i}{\L}=\intM{x_V, r_i}{V}$ for any
$x\in {\L}\tensor\R$,
we have
$$
t_i=\left(2-\varepsilon \frac{\intM{u, r_i}{\L}}{\intM{w, r_i}{\L}}\right)\inv.
$$
Since $u$ is general,
we can assume that $t_1, \dots, t_N$ are distinct.
We put the numbering of the elements $r_1, \dots, r_N$ of $P_v\sprime$ so that
$t_1<\dots<t_N$ holds.
Let $s_i\in \OG^+({\L})$ denote the reflection with respect to 
$r_i$.
Then $\DDD_0$ and $\DDD_{\opp}$ are related by
$$
\DDD_{\opp}=\DDD_{0}^{s_1 s_2\dots s_N}.
$$
Therefore $w^{s_1 s_2\dots s_N}$ is the Weyl vector of $\DDD_{\opp}$.
By this consideration, we obtain the following:
\begin{algorithm}\label{algorithm:adjacent}
Suppose that a Weyl vector $w\in\L$ of 
an $\RtsLS\sphyp$-chamber $D$ and
an element $v$ of $\Delta_{S\dual}(D)$
are  given.
This algorithm calculates a Weyl vector $w\sprime$ of 
the $\RtsLS\sphyp$-chamber $D\sprime$ adjacent to
$D$ across the wall $(v)\sperp$.
\par
We calculate the set
$P_v$
by Algorithm~\ref{algorithm:passing},
and choose  a complete set of representatives 
$P\sprime_v=\{r_1, \dots, r_N\}$
of  $P_v/\{\pm 1\}$.
We also choose a vector $u$ of ${\L}\tensor\Q$
such that $i\ne j$ implies 
${\intM{u, r_i}{\L}}/{\intM{w, r_i}{\L}} \ne  {\intM{u, r_j}{\L}}/{\intM{w, r_j}{\L}}$.
We sort the elements $r_i$  of $P\sprime_v$ so that
$$
i<j\;\;\Longrightarrow\;\;  
\frac{\intM{u, r_i}{\L}}{\intM{w, r_i}{\L}} <  \frac{\intM{u, r_j}{\L}}{\intM{w, r_j}{\L}}
$$
holds.
Then  $w^{s_1 s_2\dots s_N}$
is a Weyl vector of $D\sprime$,
where $s_i\in \OG^+(\L)$ is the reflection with respect to 
$r_i$.
\end{algorithm}
\section{The main algorithm}\label{sec:main}
We present our main algorithm,
and prove its termination and correctness. 
Let $G$, $S$ and $\L$ be as in the previous section.
Let $\Nc$ be an $\Roots_S\sphyp$-chamber.
Since $\Roots_S$ is contained in $\RtsLS$,
the $\Roots_S\sphyp$-chamber $\Nc$ is a union of $\RtsLS\sphyp$-chambers.
We fix  an $\RtsLS\sphyp$-chamber $D_0$
contained in $\Nc$.
An \emph{$\Nc$-chain}  %of $\RtsLS\sphyp$-chambers 
is a finite sequence 
$$
D\spar{0}, D\spar{1}, \dots, D\spar{l}
$$
of $\RtsLS\sphyp$-chambers  contained in $\Nc$ such that 
$D\spar{a+1}$ is adjacent to $D\spar{a}$ for each $a$.
The \emph{length} of an $\Nc$-chain  $D\spar{0}, \dots, D\spar{l}$ is 
defined to be $l$.
Let  $D$ be an $\RtsLS\sphyp$-chamber contained in $\Nc$.
Since $\Nc$ is connected,
there exists an $\Nc$-chain  $D\spar{0},  \dots, D\spar{l}$
such that $D\spar{0}=D_0$ and $D\spar{l}=D$.
The \emph{level  of $D$} is defined to be the minimum  of the lengths of all $\Nc$-chains 
from $D_0$ to $D$.
%$D\spar{0},  \dots, D\spar{l}$
%such that $D\spar{0}=D_0$ and $D\spar{l}=D$.
%
\par
In the actual execution of the following algorithm,
we use a Weyl vector to store an $\RtsLS\sphyp$-chamber 
in the computer memory.
Thus, for example,  the set $\Chams$ is realized  as a set of vectors of $\L$ in 
the computer.
\begin{algorithm}\label{algo:main}
Suppose that 
a Weyl vector $w_0$ of an $\RtsLS\sphyp$-chamber $D_0$
is given.
Let $\Nc$ be the $\Roots_{S}\sphyp$-chamber containing $D_0$.
This algorithm calculates
a finite set $\varGamma$ of generators of $\autG(\Nc)$,
a finite set  $\Chams$ of $\RtsLS\sphyp$-chambers contained in $\Nc$
such that the union 
\begin{equation}\label{eq:FN}
F_{\Nc}:=\bigcup_{D_i\in \Chams} D_i
\end{equation}
satisfies the property (2) in Proposition~\ref{prop:adj} below,
and a set $\BBB$ of $(-2)$-vectors of $S$  that satisfies the property (3) in Proposition~\ref{prop:adj}.
In fact,
the set  $\Chams$ is a complete set of representatives of $G$-congruence classes
of $\RtsLS\sphyp$-chambers,
and it is a union of non-empty subsets $\Chams_{\ell}$,
where elements of $\Chams_{\ell}$ are of level $\ell$.
\par
\smallskip

\step{0}.
Set   $\varGamma$ to be $\{\}$,
$\Chams_0$ to be $\{D_0\}$, and $\BBB$ to be $\{\}$.

\step{1}.
Calculate the set $\Delta_{S\dual}(D_0)$ from $w_0$ by Algorithm~\ref{algorithm:aSnS}.

\step{2}.
Execute $\procadj(0)$,
where $\procadj(\ell)$ is the following procedure,
which calls $\procadj(\ell+1)$ at the last step  if the condition for termination  is  not fulfilled.
\par
\smallskip 

{\it The procedure $\procadj(\ell)$.}
Suppose  that,
for  each $\lambda=0, \dots, \ell$,
a non-empty finite set $\Chams_{\lambda}$ of $\RtsLS\sphyp$-chambers with the following properties 
has been calculated.
\begin{itemize}
\item[\textrm{[D1]}] Each $D\in \Chams_{\lambda}$ is contained in $\Nc$,  and is of level $\lambda$, and 
\item[\textrm{[D2]}] if $D, D\sprime \in \cup_{\lambda=0}^{\ell}\Chams_{\lambda}$ are distinct,
then $D$ and $D\sprime$ are not $G$-congruent.
\end{itemize}
Suppose also that, for each $D \in \cup_{\lambda=0}^{\ell} \Chams_{\lambda}$,
we have calculated  $\Delta_{S\dual} (D)$.
This procedure calculates  $\Chams_{\ell+1}$ and $\Delta_{S\dual} (D)$  for each $D\in \Chams_{\ell+1}$.
\par
Let $D_{k+1}, \dots, D_{k+m}$ be the elements of $\Chams_{\ell}$, 
where $k$ is the sum of the cardinalities of $\Chams_{\lambda}$ with $\lambda<\ell$
(we put $k=-1$ if $\ell=0$), 
and $m$ is the cardinality of $\Chams_{\ell}$.
\begin{itemize}
\item[1.]
Put  $\Chams\sprime:=\{\}$ and set $l:=k+m+1$.
\item[2.]
For each $D_{i}\in \Chams_{\ell}$,
we make the following calculation.
\begin{itemize}
\item[2-1.]
Calculate the finite group $\autG(D_{i})$ from $\Delta_{S\dual}(D_{i})$
by Algorithm~\ref{algorithm:autG}, 
and append a finite set of generators  of $\autG(D_{i})$ to $\varGamma$.
\item[2-2.]
Note that $\autG(D_{i})$ acts on $\Delta_{S\dual}(D_{i})$.
Decompose   $\Delta_{S\dual}(D_{i})$ into the $\autG(D_{i})$-orbits
$o_1, \dots, o_t$.
Since $\Roots_S\sphyp$ is  $G$-invariant,
the set  $o_{\nu}\sphyp:=\shortset{(v)\sperp}{v\in o_{\nu}}$ 
is either disjoint from  $\Roots_S\sphyp$ or entirely contained in $\Roots_S\sphyp$.
Let $v$ be an element of $o_{\nu}$.
Since $v$ is primitive in $S\dual$,
we have $o_{\nu}\sphyp\subset \Roots_S\sphyp$
if and only if there exists a positive  integer $\alpha$ such that $\alpha^2 v^2=-2$ and $\alpha v\in S$.
\item[2-3.]
For each orbit $o_{\nu}$ such that $o_{\nu}\sphyp$ is  disjoint from  $\Roots_S\sphyp$,
we make the following calculation.
\begin{itemize}
\item[2-3-1.] Choose a  vector  $v \in o_{\nu}$, and calculate 
a Weyl vector $w\sprime\in {\L}$ of the $\RtsLS\sphyp$-chamber
$D\sprime$ adjacent to $D_{i}$ across the wall $(v)\sperp$
by Algorithm~\ref{algorithm:adjacent}.
We then calculate 
$\Delta_{S\dual}(D\sprime)$ 
 by Algorithm~\ref{algorithm:aSnS}.
Since $(v)\sperp\notin \Roots_S\sphyp$ and $D_i\subset \Nc$,
we have 
$D\sprime\subset \Nc$.
Since $D_i$ is of level $\ell$,
we see that $D\sprime$ is of level $\le \ell+1$.
\item[2-3-2.] 
By Algorithm~\ref{algorithm:Gequiv}, determine whether $D\sprime$ is $G$-congruent to  an $\RtsLS\sphyp$-chamber 
$D\spprime$ in 
$$
\tilde{\Chams}:=\Chams_0\cup\dots\cup \Chams_{\ell}\cup \Chams\sprime=\{D_0, D_1, \dots, D_{l-1}\}
$$
 or not.
If $D\sp{\prime}=D\sp{\prime\prime \, h}$ for some $D\spprime\in \tilde{\Chams}$ and $h\in G$, 
then  append $h$ to $\varGamma$.
If there exist no such $D\spprime\in \tilde{\Chams}$ and $h\in G$, then  
$D\sprime$ represents a new $G$-congruence class and its level is $\ell+1$, and hence we put 
$D_{l}:=D\sprime$, 
append $D_{l}$ to $\Chams\sprime$ and increment $l$ by $1$.
%Note that  $\tilde{\Chams}$ still has the property that distinct elements of $\tilde{\Chams}$ 
%are not $G$-congruent.
\end{itemize}
\item[2-4.]
For each orbit $o_{\nu}$ such that $o_{\nu}\sphyp\subset \Roots_S\sphyp$,
choose a vector $v\in o_{\nu}$,
find  a positive integer $\alpha$ 
such that $r:=\alpha v \in \Roots_{S}$,
and append $r$ to $\BBB$.
\end{itemize}
\item[3.] If $\Chams\sprime\ne \emptyset$, then 
put $\Chams_{\ell+1}:=\Chams\sprime$, which has the properties [D1] and [D2] above, 
 and  execute $\procadj(\ell+1)$.
If $\Chams\sprime= \emptyset$, then 
put 
$$
\Chams:=\Chams_0\cup \Chams_1\cup\cdots\cup  \Chams_{\ell},
$$
and terminate.
\end{itemize}
\end{algorithm}
\begin{proposition}\label{prop:terminates}
Algorithm~\ref{algo:main} terminates.
\end{proposition}
\begin{proof}
By construction,
any two distinct $\RtsLS\sphyp$-chambers in 
$\Chams_0\cup\dots\cup \Chams_{\ell} \cup \Chams\sprime$ are not $G$-congruent
during the calculation.
Thus Proposition~\ref{prop:terminates} follows from Theorem~\ref{thm:finiteGequiv}.
\end{proof}
\begin{proposition}\label{prop:adj}
{\rm (1)} 
The group  $\autG(\Nc)$ is generated by $\varGamma$.

{\rm (2)} 
For any $v\in \Nc$, there exists  an element $g\in \autG(\Nc)$  such that $v^g\in F_{\Nc}$,
where $F_{\Nc}$ is defined by~\eqref{eq:FN}.
%If $v\in \Nc$ is general, then the number of $g\in \autG(\Nc)$ satisfying $v^g\in F_{\Nc}$
%is finite.

{\rm (3)} 
Let $r$ be an element of the $\Roots_S$-minimal defining set $\Delta_{\Roots_S}(\Nc)$
of $\Nc$.
Then there exists  an element $g\in \autG(\Nc)$  such that $r^g\in \BBB$.
\end{proposition}
\begin{proof} 
Since each $D_i\in \Chams$ is contained in $\Nc$,
we have $\autG(D_i)\subset \autG(\Nc)$ by~\eqref{eq:intpt},
and hence all elements of $\varGamma$ appended in Step 2-1 is an element of $\autG(\Nc)$.
If $h\in G$ is appended to $\varGamma$ in Step 2-3-2,
then we have $D\sp{\prime}=D\sp{\prime\prime \, h}$ for some $\RtsLS\sphyp$-chambers $D\sprime$ and $D\spprime$
contained in $\Nc$, and hence $h\in \autG(\Nc)$ by~\eqref{eq:intpt}.
Therefore the subgroup $\gen{\varGamma}$ of $G$ generated by $\varGamma$ is contained in $\autG(\Nc)$.
\par
To prove the rest of Proposition~\ref{prop:adj}, 
it is enough to show the following.
This claim also proves that $\Chams$ is a complete set of representatives of 
$G$-congruence classes of $\RtsLS\sphyp$-chambers contained in $\Nc$.
\begin{claim}\label{claim}
For an arbitrary $\RtsLS\sphyp$-chamber $D$ contained in $\Nc$,
 there exists an element  $\gamma\in \gen{\varGamma}$
such that $D^\gamma\in \Chams$.
\end{claim}
Indeed, suppose that Claim~\ref{claim} is proved.
Let $g$ be an arbitrary element of $\autG(\Nc)$.
Since $D_0^g\subset \Nc$, there exists an element $\gamma\in \gen{\varGamma}$
such that  $(D_0^g)^\gamma$ is equal to some $D_i \in \Chams$.
Since  $D_i$ and $D_0$ are $G$-congruent and $D_i, D_0\in \Chams$, we have $D_0=D_i$.
Therefore  $g\gamma \in \autG(D_0)\subset \gen{\varGamma}$ follows,
and hence we have $g\in \gen{\varGamma}$.
Let $v\in \Nc$ be an arbitrary vector,
and let $D$ be an $\RtsLS\sphyp$-chamber containing $v$ and contained in $\Nc$.
By Claim~\ref{claim},
we have $\gamma\in \gen{\varGamma}$
such that 
$D^\gamma=D_i \in \Chams$. 
Then we have 
$v^\gamma\in F_{\Nc}$.
Suppose that $r\in \Delta_{\Roots_S}(\Nc)$.
Then there exist an $\RtsLS\sphyp$-chamber $D$ contained in $\Nc$ and 
a vector $v\in \Delta_{S\dual}(D)$ 
such that $r=\alpha v$ for some positive integer $\alpha$.
By Claim~\ref{claim},
we have $\gamma\in \gen{\varGamma}$
such that  $D^\gamma=D_i \in \Chams$.
Then  $v^\gamma \in \Delta_{S\dual}(D_i)$ is contained in an $\aut_{G}(D_i)$-orbit  $o_{\nu}$
such that $o_{\nu}\sphyp\subset\Roots_{S}\sphyp$.
Hence there exists  an element $\gamma\sprime\in \aut_{G}(D_i)$ such that
$r^{\gamma \gamma\sprime}$ is appended to $\BBB$ in Step 2-3-3.
\par
Now we prove Claim~\ref{claim}.
We fix an $\RtsLS\sphyp$-chamber $D$ contained in $\Nc$,
and prove the existence of  $\gamma\in \gen{\varGamma}$ such that $D^\gamma\in \Chams$.
An $\Nc$-chain $D\spar{0}, D\spar{1}, \dots, D\spar{l}$ of $\RtsLS\sphyp$-chambers 
is said to be \emph{$D$-admissible} if $D\spar{0}$ belongs to $\Chams$
and there exists  an element  $\gamma\in \gen{\varGamma}$ such that $D\spar{l}=D^{\gamma}$.
Since  $\Nc$ is connected, 
there exists at least one  $D$-admissible $\Nc$-chain.
It is enough to show that
there exists a $D$-admissible $\Nc$-chain of  length $0$.
We suppose that the  $D$-admissible $\Nc$-chain with \emph{minimal} length 
$$
D\spar{0}=D_i, \; D\spar{1}, \; \dots, \; D\spar{l}=D^{\gamma}\qquad(D_i\in \Chams, \; \gamma\in \gen{\varGamma})
$$
is of length $l>0$, and 
derive a contradiction.
Let $v\sprime\in \Delta_{S\dual}(D_i)$ be the vector such that $(v\sprime)\sperp$ is the wall between 
$D\spar{0}=D_i$ and $D\spar{1}$.
Since $D_i$ and $D\spar{1}$ are contained in $\Nc$,
the $\autG(D_i)$-orbit
 $o_{\nu}\subset \Delta_{S\dual}(D_i)$   containing $v\sprime$
 satisfies $o_{\nu}\sphyp\cap \Roots_{S}\sphyp=\emptyset$.
 %is disjoint from $\Roots_S$.
Let $v$ be the  vector of 
 $o_{\nu}$ chosen in Step 2-3-1,
and let $g\in \autG(D_i)$ be an element that maps $v\sprime$ to $v$.
Then $D^{(1)g}$ is the $\RtsLS\sphyp$-chamber adjacent to $D\spar{0}=D_i$
across $(v)\sperp$.
Since $g\in \gen{\varGamma}$, 
the $\Nc$-chain
$$
D\spar{0}=D_i, \, D^{(1)g}, \, \dots, \, D^{(l)g}=D^{\gamma g}
$$
is $D$-admissible.
By the minimality of $l$,
$D^{(1) g}$ does not belong to  $\Chams$.
Hence, in Step 2-3-2, the  $\RtsLS\sphyp$-chamber
$D^{(1) g}$ is not appended to $\Chams\sprime$,
which means that there exist a chamber $D\spprime\in\Chams$ and and an element $h\in G$
such that $D\sp{(1) g h}=D\spprime$.
The element  $h\inv$ is appended to $\varGamma$ in Step 2-3-2,
and hence $\gamma  g h\in \gen{\varGamma}$.
Then
$$
D^{(1) g h}=D\spprime, \, D^{(2) g h}, \, \dots, D^{(l) g h}=D^{\gamma g h}
$$
is a $D$-admissible $\Nc$-chain of length $l-1$,
which is a contradiction.
\end{proof}
\begin{remark}\label{rem:weakfund}
For $D_i\in \Chams$,
let $E(D_i)\subset D_i$ be a fundamental domain of the action of the finite group $\aut_{G}(D_i)$
on $D_i$. 
Then their union $\bigcup_{D_i\in \Chams} E(D_i)$
is a fundamental domain of the action of $\aut_{G}(\Nc)$ on $\Nc$.
In particular,
if we have $|\aut_{G}(D)|=1$ for any $D\in \Chams$, 
then $F_{\Nc}$ is a fundamental domain of the action of $\aut_{G}(\Nc)$ on $\Nc$.
\end{remark}
\begin{remark}
When 
 Algorithm~\ref{algo:main} is applied to the case $G=\OG^+(S)$, we see that $\OG^+(S)$
is generated by $\varGamma$ and the reflections $\shortset{s_r}{r\in \BBB}$.
\end{remark}
\begin{remark}\label{rem:Rroots}
Suppose that $n=26$ and that $R$ contains  a root lattice
as a sublattice of finite index.
By  Borcherds~\cite[Lemma 5.1]{MR913200},
%there exists only one $\OG\sp+(S)$-congruence class of $\RtsLS\sphyp$-chambers,
there exist only a small number of  $\OG\sp+(S)$-congruence classes of $\RtsLS\sphyp$-chambers. 
%and hence $|\Chams|$ is at most the index of $G$ in $\OG\sp+(S)$. %added on April 21
%Hence Algorithm~\ref{algo:main} terminates quickly in this case.
\end{remark}
\begin{remark}\label{rem:speedup}
By Remark~\ref{rem:prSDelta}, 
we can make Step 2-3 faster;
namely, we can determine whether $D\sprime$ represents a new $G$-congruence class or not by
using $\pr_S(\Delta_{w\sprime})$ instead of $\Delta_{S\dual}(D\sprime)$.
\end{remark}
\section{The automorphism group of a $K3$ surface}\label{sec:AutX}
In this section, we review the classical theory of the automorphism groups of  $K3$ surfaces.
Let $X$ be an algebraic  $K3$ surface
with the N\'eron-Severi lattice $S_X$
of rank $>1$.
Then $S_X$ is an even hyperbolic lattice.
Let $\PPP(X)$ denote the positive cone of $S_X$ containing an ample class.
The nef cone $\Nef(X)$ of $X$ is defined by
$$
\Nef(X):=\set{v\in S_X\tensor\R}{\intM{v, [C]}{}\ge 0\;\;\textrm{for any curve $C$ on $X$}},
$$
where $[C]\in S_X$ is the class of a curve $C$.
Note that, by Kleiman's criterion, %~\cite{MR0206009},
$\Nef(X)$ is contained in the closure $\clPPP(X)$ of $\PPP(X)$ in $S_X\tensor\R$.
We put
$$
\Nc(X):=\Nef(X)\cap \PPP(X)=\Nef(X)\setminus (\Nef(X) \cap \bdr \clPPP(X)).
$$
%(Hence $\Nef(X)\setminus \Nc(X)$ is the union of  all $\R_{\ge 0} f_{\phi}$,
%where $\phi\colon X\to \P^1$ is a genus $1$ fibration and $f_{\phi}\in S_X$ is the class of a fiber of $\phi$.)
It is known that $\Nc(X)$ is an $\Roots_{S_X}\sphyp$-chamber,
and that the $\Roots_{S_X}$-minimal defining set $\Delta_{\Roots_{S_X}}(\Nc(X))$
consists of the classes of smooth rational curves on $X$~(see, for example, Rudakov-Shafarevich~\cite{MR633161}).
For simplicity, we put
$$
\aut(\Nc(X)):=\set{g\in \OG^+({S_X})}{\Nc(X)^g=\Nc(X)}=\aut_{\OG^+({S_X})}(\Nc(X)).
$$
By a \emph{$(-2)$-wall},
we mean a wall bounding $\Nc(X)$;
that is, a wall $([C])\sperp$, where $C$ is a smooth rational curve on $X$.
\subsection{Complex $K3$ surfaces}\label{subsec:AutXC}
Suppose that  $X$ is defined over $\C$.
With the cup-product, 
the second cohomology group $H:=H^2(X, \Z)$ is an even unimodular lattice of signature
$(3, 19)$.
Let $T_X$ denote the orthogonal complement of ${S_X}$ in $H$,
which we call the \emph{transcendental lattice} of $X$.
We  regard a non-zero holomorphic $2$-form $\omega_X$ on $X$ as a vector of ${T_X}\tensor\C$, and put
\begin{equation}\label{eq:CX}
C_{X}:=\set{g\in \OG({T_X})}{\omega_X^g=\lambda\, \omega_X\;\;\textrm{for some}\;\; \lambda\in \C\sptimes}.
\end{equation}
%
%is known to be   finite and  cyclic~(see, for example, Sterk~\cite{MR786280}).
Since $H$ is  unimodular,
the subgroup $H/({S_X}\oplus {T_X})$ of the discriminant group $A_{S_X}\oplus A_{T_X}$
of ${S_X}\oplus {T_X}$ 
is the graph of an isomorphism
$$
\delta_{ST}\colon (A_{S_X}, q_{S_X})\;\isom\;(A_{T_X}, -q_{T_X})
$$
by  Nikulin~\cite[Proposition~1.6.1]{MR525944}, 
which induces an isomorphism of the automorphism groups 
$\delta_{ST*}\colon \OG(q_{S_X})\isom \OG(q_{T_X})$
of discriminant forms.
Recall that, for an even lattice $L$,   we have a natural homomorphism $\eta_L\colon \OG(L)\to \OG(q_L)$.
\begin{theorem}[Piatetski-Shapiro and Shafarevich~\cite{MR0284440}]\label{thm:aut} 
Via the natural actions of $\Aut(X)$ on the lattices ${S_X}$ and  ${T_X}$, 
the automorphism group $\Aut(X)$ of $X$ is identified with 
$$
\set{(g, h)\in \aut(\Nc(X))\times C_{X}}{\delta_{ST*}(\eta_{S_X}(g))=\eta_{T_X}(h)}.
$$
\end{theorem}
Since $\OG(q_{T_X})$ is finite,
the subgroup
\begin{equation}\label{eq:GXcomplex}
G_X:=\set{g\in \OG^+({S_X})}{\delta_{ST*}(\eta_{S_X}(g)) \in \eta_{T_X}(C_X)}
\end{equation}
of $\OG^+({S_X})$
has  finite index.
\begin{corollary}\label{cor:KerIm}
The kernel of the natural homomorphism $\varphi_X\colon \Aut(X)\to \OG(S_X)$
is isomorphic to $\Ker (\eta_{T_X}) \cap {C_X}$, and its image 
is equal to 
$$
\aut_{G_X}(\Nc(X))=\shortset{g\in G_X}{\Nc(X)^g=\Nc(X)}.
$$
\end{corollary}
\subsection{Supersingular $K3$ surfaces in odd characteristics}\label{subsec:AutssK3}
Suppose that  $X$ is a supersingular $K3$ surface defined over an algebraically closed field
$k$ of \emph{odd} characteristic $p$.
By Artin~\cite{MR0371899},
we know that the discriminant group  $\discgr{{S_X}}$
is a $p$-elementary abelian group of rank $2\sigma$,
where $\sigma$ is a positive integer $\le 10$,
which is called the \emph{Artin invariant of $X$}.
The $\F_p$-vector space 
$S_0:=p S_X\dual /p{S_X}$ of dimension $2\sigma$ has a 
natural quadratic form 
$$
%Q_0\;\;:\;\; px \bmod p{S_X} \mapsto px^2 \bmod p \quad (x\in S_X\dual)
Q_0\colon  px \bmod p{S_X} \mapsto px^2 \bmod p \quad (x\in S_X\dual)
$$
that takes values in $\F_p$.
We denote by $\OG(Q_0)$ the finite group of automorphisms of $(S_0, Q_0)$.
We have a natural homomorphism
$\OG({S_X}) \to \OG(Q_0)$.
We denote by  $\varphi\colon S_0\tensor k\to S_0\tensor k$
the map $\id_{S_0}\tensor F_k$,
where $F_k$ is the Frobenius map of $k$.
Let $c_{\DR}\colon {S_X}\to H^2_{\DR}(X/k)$ denote the Chern class map.
Then the kernel $\Ker (\bar{c}_{\DR})$ 
of the induced homomorphism $\bar{c}_{\DR}\colon {S_X}\tensor k \to H^2_{\DR}(X/k)$
from ${S_X}\tensor k=S_X/p S_X$ to $H^2_{\DR}(X/k)$
is contained in $S_0\tensor k$.
(Note that we have $p S_X\dual \subset S_X$.)
The subspace 
$$
K:=\varphi\inv(\Ker (\bar{c}_{\DR}))
$$
of $S_0\tensor k$ is called the 
\emph{period  of $X$}.
\begin{theorem}[Ogus~\cite{MR563467},~\cite{MR717616}]\label{thm:ssK3aut}
Via the natural action of $\Aut(X)$ on  ${S_X}$, 
the automorphism group $\Aut(X)$ of $X$ is identified with 
$$
\set{g\in \aut(\Nc(X))}{K^g=K}.
$$
\end{theorem}
Since $\OG(Q_0)$ is finite,
the subgroup
\begin{equation}\label{eq:GXss}
G_X:=\set{g\in \OG^+({S_X})}{K^g=K}
\end{equation}
of $\OG^+({S_X})$
has  finite index.
\begin{corollary}
The natural homomorphism $\Aut(X)\to \OG(S_X)$
is injective, and its 
 image  is equal to $\aut_{G_X}(\Nc(X))$.
 \end{corollary}
\section{Geometric application of Algorithm~\ref{algo:main}}\label{sec:geometricapplication}
Let $X$ be a complex algebraic $K3$ surface or a supersingular $K3$ surface in
odd characteristic.
Let $S_X$ be the N\'eron-Severi lattice of $X$,
and let $\Nc(X)$ be the $\Roots_{S_X}\sphyp$-chamber $\Nef(X)\cap \PPP(X)$
in the positive cone $\PPP(X)$ containing an ample class.
%obtained by taking  the intersection with  the nef cone $\Nef(X)$ of $X$.
Let $G_X$ be the subgroup of $\OG^+(S_X)$ defined by~\eqref{eq:GXcomplex} or~\eqref{eq:GXss}.
%and let $\varphi_X\colon \Aut(X)\to \OG(S_X)$ be the natural homomorphism.
Applying Algorithm~\ref{algo:main}
to the case $S=S_X$, $G=G_X$ and $\Nc=\Nc(X)$, 
we can calculate a finite set of generators of 
$$
\Image(\varphi_X\colon \Aut(X)\to \OG(S_X))=\aut_{G_X}(N (X))
$$
and %a weak fundamental domain of the action of $\Aut(X)$ on $\Nc(X)$,
a closed domain $F_{\Nc(X)}$ of $\Nc(X)$ with the properties given in Proposition~\ref{prop:adj}, 
provided that the following hold:
\begin{itemize}
\item[(1)] 
The subgroup $G_X$ of $\OG^+(S_X)$ satisfies  
the condition $\propVG$ of the existence of a membership algorithm 
in Section~\ref{sec:chamber}.
\item [(2)]
We can find a primitive embedding of  $S_X$
into an even unimodular  hyperbolic lattice $\L$ of rank $10$, $18$ or $26$
such that the orthogonal complement $R$ of $S_X$ in $\L$ 
satisfies $\assumptwo$ and~$\assumpthree$  in Section~\ref{sec:Borcherds}. 
\item[(3)]
 We can find a Weyl vector of an $\RtsLS\sphyp$-chamber $D_0$
contained in the $\Roots_{S}\sphyp$-chamber $\Nc(X)$.
\end{itemize}
We discuss  these requirements for a complex $K3$ surface $X$.
Let $\rho_X$ denote the Picard number of $X$,
and let $T_X$ denote the transcendental lattice of $X$.
\par
%\medskip
%
\subsection{Requirement (1).}\label{subsec:req1}
By the definition~\eqref{eq:GXcomplex} of $G_X$,
we have a membership algorithm for the subgroup $G_X$ of $\OG^+(S_X)$
if we have explicit descriptions of 
the isomorphism $\delta_{ST}\colon (A_{S_X}, q_{S_X})\isom (A_{T_X}, -q_{T_X})$
induced by the even unimodular overlattice $H^2(X, \Z)$ of $S_X\oplus T_X$,
the homomorphisms 
$\eta_{S_X}\colon \OG(S_X)\to \OG(q_{S_X})$, $\eta_{T_X}\colon \OG(T_X)\to \OG(q_{T_X})$,
and the subgroup $C_X$ of $\OG(T_X)$.
\par
Suppose that $C_X=\{\pm 1\}$.
Then the condition $g\in G_X$ is reduced to the condition $\eta_{S_X}(g)\in\{\pm 1\}$,
and hence all we need to do is to calculate 
the homomorphism $\eta_{S_X}$.
This assumption $C_X=\{\pm 1\}$
holds in many cases.
For example, if  $\rho_X<20$
and the period $\omega_X$ is generic in $T_X\tensor \C$,  then $C_X=\{\pm 1\}$.
Indeed, since the eigenspaces in  $T_X\tensor\C$  of any $g\in \OG(T_X)\setminus \{\pm 1\}$ are proper subspaces,
 a period $\omega_X$ for which $C_X\ne \{\pm 1\}$ must lie in a countable union of proper subspaces of $T_X\tensor\C$.
\subsection{Requirement (2).}\label{subsec:req2}
We have the following:
\begin{proposition}\label{prop:embeddable}
For a complex algebraic $K3$ surface $X$,
the lattice $S_X$  has a primitive embedding into an even unimodular  hyperbolic lattice
$\L_{26}$
of rank $26$.
\end{proposition}
\begin{proof}
Recall that  $T_X$ is  an even lattice of signature $(2, 20-\rho_X)$
such that $(A_{T_X}, q_{T_X})$ is isomorphic to $(A_{S_X}, -q_{S_X})$.
By  Nikulin~\cite[Theorem~1.10.1]{MR525944},
the existence of the lattice $T_X$ implies 
the existence  of an even lattice $R$ of signature $(0, 26-\rho_X)$
with $(A_R, q_R)\cong (A_{T_X}, q_{T_X})$.
Hence,
by  Nikulin~\cite[Proposition~1.6.1]{MR525944},
$S_X$ can be embedded primitively into an even unimodular hyperbolic lattice 
of rank $26$ 
with $R$ being the orthogonal complement.
Since 
an even unimodular hyperbolic lattice 
of rank $26$ is unique up to isomorphism,
Proposition~\ref{prop:embeddable} follows.
\end{proof}
 Therefore $S_X$ always satisfies the condition
$\assumpone$ in Section~\ref{sec:Borcherds}.
It is, however,  difficult in general
to obtain  a primitive embedding $S_X\inj \L$ explicitly.
In fact, by Nikulin~\cite[Proposition~1.6.1]{MR525944},
this problem is equivalent to construct a negative-definite lattice $R$ of
rank equal to $\rank\L-\rank S_X$ such that
$(A_R, q_R)\cong (A_{S_X}, -q_{S_X})$ holds.
For the algorithm to construct an integral lattice in a given genus explicitly,
see Conway and Sloane~\cite[Chapter 15]{MR1662447}.
\begin{remark}
In the case where $X$ is supersingular,
we can use the table~\cite{MR1118842} of  positive-definite integral lattices of rank $4$.
\end{remark}
Once a primitive embedding $S_X\inj \L$ is obtained,
we can calculate a Gram matrix of $R$ of the orthogonal complement,
the isomorphism $\delta_{\L}\colon (A_{S_X}, q_{S_X})\isom (A_R, -q_R)$ of discriminant forms induced by $\L\subset S_X\dual\oplus R\dual$,
and the induced isomorphism $\delta_{\L*}\colon \OG(q_{S_X})\isom \OG(q_{R})$.
Since both $R$ and $\Lambda$ are negative-definite,
we can enumerate all embeddings of $R$ into $\Lambda$,
and hence the condition~$\assumptwo$ can be checked.
In practice, ~$\assumptwo$ is often verified simply by showing that $\Roots_{R}\ne \emptyset$.
Note that $\OG(R)$ is a finite group,
and hence its image  by $\eta_R\colon \OG(R)\to \OG(q_R)$ can be  calculated.
By the definition of $G_X$, the subgroup  $\eta_{S_X}(G_X)$ of $\OG(q_{S_X})$
is equal to
$$
\Image \eta_{S_X}\cap \delta_{ST*}\inv (\eta_{T_X}(C_X)).
$$
Hence, if $\delta_{\L*}\inv ( \Image \eta_{R})$ contains $\delta_{ST*}\inv (\eta_{T_X}(C_X))$,
then  the liftability condition~$\assumpthree$ is satisfied.
This sufficient condition for $\assumpthree$ is fulfilled
in the following cases that occur frequently:
the case where  $\eta_{R}$ is surjective,
or the case where $C_X=\{\pm 1\}$.
\subsection{Requirement (3).}\label{subsec:req3}
In order to find an initial $\RtsLS\sphyp$-chamber $D_0$
contained in $\Nc(X)$, 
we can employ one of the following two methods. % or its variant.
\par
Let $\DDD_0$
be the $\Roots_{\L}\sphyp$-chamber
corresponding to 
the standard Weyl vector $w_0$ given in Section~\ref{sec:Conway}.
We choose an interior point $u_0$ of $\DDD_0$.
When  $\rank \L$ is $10$ or $18$,
$w_0$ is an interior point of $\DDD_0$.
When  $\rank \L$ is $26$ and $w_0=f_U$,
we can use
$$
u_0:=3f_U+z_U \in \L=U\oplus \Lambda,
$$
because we have $u_0^2=4$, $\intM{u_0, w_0}{\L}=1$ and 
$\intM{u_0, r_{\lambda}}{\L}=1-\lambda^2/2>0$ for all $\lambda\in \Lambda$.
\par
Next we find an ample class $a_0\in S_X$.
The method of this step depends, of course,
on what kind of geometric information of $X$ is available.
%and hence we cannot present an algorithm that works in arbitrary cases.
Once an ample class $a_0$ is obtained,
 we can determine whether a given vector $v\in S_X$
 is ample or not by means of the following criterion and Algorithms~\ref{algo:QLcab} and~\ref{algo:isnef}.
 Note that a vector of $S_X$ is ample if and only if it  belongs to the interior of $N(X)$.
 Therefore
a vector $v\in S_X$ is ample if and only if
 \begin{itemize}
 \item[(i)] $v^2>0$ and $\intM{v, a_0}{}>0$, so that $v\in \PPP(X)$,
 \item[(ii)] the set $\shortset{r\in \Roots_{S_X}}{\intM{v, r}{}=0}$ is empty, and
  \item[(ii)] the set $\shortset{r\in \Roots_{S_X}}{\intM{v, r}{}<0,\, \intM{a_0, r}{}>0}$ is empty, so that the line segment in $\PPP(X)$ connecting $a_0$ and $v$
  does not intersect any hyperplane $(r)\sperp$ perpendicular to some $(-2)$-vector $r$.
 \end{itemize}
 For example,
 we can produce many ample classes $Aa_0+u$
 by putting $A$ to be  sufficiently large integers
 and $u\in S_X$ to be  vectors with relatively small coordinates.
% (For a fixed vector $u$ of perturbation, the vector $Aa_0+u$ is ample for a sufficiently large $A$.)
 \par
We choose a general  ample class $a\in S_X$.
(The cases where $a$ is not general enough so that  we have to re-choose $a$ are indicated below.)
Let $i_0\colon S_X\inj \L$ be  a primitive embedding.
We consider the oriented line segment
$$
\ell(t):=(1-t)\ i_0(a)+t\ u_0\quad ( 0\le t\le 1)
 $$
 in $\PPP_{\L}$
 from $i_0(a)$ to $u_0$.
By Algorithm~\ref{algo:isnef}, 
we calculate the finite set
$$
\set{r\in \Roots_{\L}}{\intM{u_0, r}{\L}>0,\;\; \intM{i_0(a), r}{\L}<0}=\{r_1, \dots, r_N\}
$$
and sort the elements of this set in such a way that
$$
t_1\le t_2<\dots\le t_N, \quad\textrm{where $t_i$ satisfies $\intM{\ell(t_i), r_i}{\L}=0$}.
$$
If  $t_1, \dots, t_N$ are not distinct, 
 we re-choose $a$.
 We assume that $t_1, \dots, t_N$ are distinct.
Then the oriented line segment $\ell$ intersects 
the hyperplanes 
$(r_1)\sperp$, \dots, $(r_N)\sperp$ in this order.
We replace the embedding $i_0$ by
$$
i:=s_{r_N}\circ \cdots \circ s_{r_1}\circ i_0,
$$
and consider $S_X$ as a primitive sublattice of $\L$ by $i$.
Then the ample class $a$ is contained in $\DDD_0\cap \PPP(X)$.
We check  whether the inequalities~\eqref{eq:ineqSnondeg} in Criterion~\ref{criterion:Snondeg} 
are satisfied for $v=a$ and $w=w_0$.
If not,
then we re-choose $a$ and repeat the process  again.
If~\eqref{eq:ineqSnondeg} is satisfied, 
then $D_0:=\DDD_0\cap \PPP(X)$ is an $\RtsLS\sphyp$-chamber
containing $a$ in its interior.
In particular, $D_0$ is contained in $\Nc(X)$ and 
the standard Weyl vector $w_0$ is a Weyl vector of $D_0$.
\par
The other method does not necessarily succeed,
but  is practically useful.
In fact, we have used this method in~\cite{MR3286672} and~\cite{KondoShimada}.
Suppose that $\rank\L=26$.
We choose an ample class $a\in S_X$
that is primitive in $S_X$.
By Algorithm~\ref{algo:QLc},
we can calculate the list of all vectors $v\in R$
such that $a^2+v^2=0$.
For each $v$ in this list,
we determine whether the vector $w:=a+v\in S_X\oplus R$ of $\L$
is a Weyl vector or not by Theorem~\ref{thm:ConwaySloane};
we check whether $\gen{w}\sperp/\gen{w}$
is a negative-definite  unimodular lattice with no $(-2)$-vectors  by applying Algorithm~\ref{algo:QLc}
to a Gram matrix of $\gen{w}\sperp/\gen{w}$.
Suppose that a Weyl vector $w_0:=a+v_0$ of this form is found,
and let $\DDD_0$ be the corresponding $\Roots_{\L}\sphyp$-chamber.
We check  whether the inequalities~\eqref{eq:ineqSnondeg} in Criterion~\ref{criterion:Snondeg} 
are satisfied for $v=a$ and $w=w_0$.
%If not, we re-choose $a$ and repeat the process  again.
If~\eqref{eq:ineqSnondeg} is satisfied, 
then $D_0:=\DDD_0\cap \PPP(X)$ is an $\RtsLS\sphyp$-chamber
containing $a$ in its interior.
In particular, $D_0$ is contained in $\Nc(X)$ and 
the  Weyl vector $w_0$ is a Weyl vector of $D_0$.
\section{Complex elliptic $K3$ surfaces with Picard number $3$}\label{sec:rho3}
We demonstrate  Algorithm~\ref{algo:main} on certain complex $K3$ surfaces with Picard number $3$,
because, 
in this case,
we can draw  pictures of the closed domain $F_{\Nc(X)}$  defined by~\eqref{eq:FN} in the hyperbolic plane.
\par
%\medskip
Let $X$ be a complex $K3$ surface with Picard number $3$
and with 
a Jacobian fibration
$$
\phi\colon X\to\P^1
$$
whose Mordell-Weil group $\MW_{\phi}$ is of rank $1$.
We assume  that 
the period $\omega_X$ of $X$ is generic in $T_X\tensor \C$,
which implies that the group $C_X$ defined by~\eqref{eq:CX}
is equal to $\{\pm 1\}$.
We denote by $f_{\phi}\in S_X$ the class of a fiber of $\phi$ and 
by $z_{\phi}\in S_X$ the class of the zero section of $\phi$. 
Then there exists a vector $v_3\in S_X$,
unique up to sign,  such that
$f_{\phi}, z_{\phi}, v_3$ form a basis of $S_X$,
and that the Gram matrix of $S_X$  with respect to $f_{\phi}, z_{\phi}, v_3$ is
$$
M:=\left[
\begin{array}{ccc}
0 & 1 & 0 \\
1 & -2 & 0 \\
0 & 0 & -2k
\end{array}
\right],
$$
where $k:=-v_3^2/2$. 
Since  $\MW_{\phi}$ is of rank $1$,
there exist no reducible fibers of $\phi$ and hence we have $k>1$~(see, for example,~\cite{MR1081832}).
A vector $\xi f_\phi +\eta z_\phi +\zeta v_3$ of $S_X\tensor\R$ is written as $[\xi, \eta, \zeta]$.
The discriminant group $A_{S_X}$ of $S_X$ is 
a cyclic group of order $2k$ generated by
$u_3 \bmod S_X$, where 
$u_3:=[0,0,-1/2k]$.
Since $k>1$, we have $-1\notin \Ker (\eta_{T_X})$.
Therefore, by Corollary~\ref{cor:KerIm} and the assumption $C_X=\{\pm 1\}$,
the natural homomorphism $\Aut(X)\to \OG^+(S_X)$ is injective,
and $\Aut(X)\cong \aut_{G_X}(\Nc(X))$ holds,
where  $G_X$ is defined by~\eqref{eq:GXcomplex}.
We describe the group $G_X$ in terms of matrices.
Note that the class $a:=[2,1,0]\in S_X$ is nef,
and hence $g\in \OG(S_X)$ belongs to $\OG^+(S_X)$ if and only if 
$\intM{a^g, a}{S_X}>0$.
Therefore $G_X$ is canonically identified with 
$$
\set{g\in \GL_3(\Z)}{ g\,M\,{}^t g=M, \;\; a\, g\, M\, {}^ta >0, \;\;  u_3 g\equiv \pm u_3\bmod \Z^3}.
$$
This identification provides us with 
the membership algorithm for $G_X$ in the condition~$\propVG$ in Section~\ref{sec:chamber}.
%In particular, the subgroup $G_X$ of $\OG^+(S_X)$ satisfies~$\propVG$ in Section~\ref{sec:chamber}.
%
\par
We embed $S_X$ into the even unimodular hyperbolic lattice $\L=U\oplus E_8$ of rank $10$ primitively.
Note that $\assumptwo$ is irrelevant in this case,
and $\assumpthree$ holds because $C_X=\{\pm 1\}$.
Let $\DDD_0$ be  the $\Roots_{\L}\sphyp$-chamber  
with the Weyl vector $w_0\in \L=U\oplus E_8$  given in Section~\ref{subsec:10}.
We will find 
a primitive embedding $i\colon S_X\inj \L$  such that
$D_0:=i\inv(\DDD_0)$ is an $\Roots_{\L|S_X}\sphyp$-chamber 
contained in $\Nc(X)$
by the method described in Section~\ref{subsec:req3}.
To find an initial  primitive embedding $i_0$,
it is enough to 
choose a primitive vector $v_3\sprime\in E_8$ such that
$v_3\sp{\prime 2}=-2k$,  and define  $i_0\colon S_X\inj \L$ 
 by
$$
i_0(f_\phi)=f_U, \quad
i_0(z_\phi)=z_U, \quad
i_0(v_3)=v_3\sprime,
$$
where $f_U$ and $z_U$ are the basis of the hyperbolic plane $U$ 
with respect to which the Gram matrix is~\eqref{eq:nonstandardU}. 
As the general ample class $a$, 
we use $[3A, A, -1]$ with $A$ large enough.
%(Or more precisely, we choose the vector $v_3$ in the basis of $S_X$ 
%in such a way that $[3A, A, -1]$ is ample for $A$ large enough.)
%
%
\par
%\medskip
There exist two obvious elements in $\Aut(X)$.
We have a canonical bijection 
$$
\MW_{\phi}\;\;\cong\;\;\set{[k s^2, 1, s]}{s\in\Z}
$$
by sending  sections of $\phi$ to the classes of their images.
The  Mordell-Weil group $\MW_{\phi}\cong \Z$ acts on $X$ as translations.
We also have the inversion automorphism $\iota_X\colon X\to X$
induced by the multiplication by $-1$
on the generic fiber of $\phi$.
Therefore $\Aut(X)$ contains $\MW_{\phi}\semidirectproduct\gen{\iota_X}\cong \Z/2\Z * \Z/2\Z$,
which is generated by the two involutions 
$$
h_1 :=\iota_X= \left[
\begin{array}{ccc}
1 & 0 & 0 \\
0 & 1 & 0\\
0 & 0 & -1
\end{array}\right],
\quad
h_2 := \left[
\begin{array}{ccc}
1 & 0 & 0 \\
k & 1 & -1\\
2k & 0 & -1
\end{array}\right].
$$
%(The automorphism $h_2$ interchanges $[0,1,0]$ and $[k, 1,-1]$.)
%Applying Algorithm~\ref{algo:main} to $S:=S_X$ and $\L:=\textrm{II}_{1, 9}$,
%we will obtain  $h_3, \dots, h_N\in \OG^+(S_X)$
%such that  $\Aut(X)\cong \aut_{G_X}(\Nc(X))$ is generated by $\{h_1, h_2, h_3, \dots, h_N\}$.
%
\par
%\medskip
%
To describe  the closed domain $F_{\Nc(X)}$
satisfying the properties given in Proposition~\ref{prop:adj}, %a weak fundamental domain of the action of $\Aut(X)$ on $\Nc(X)$,
we use the following method.
The norm of $[1,x,y]\in S_X\tensor\R$ is $2x-2x^2-2ky^2$.
Hence, by the map $[1,x,y]\mapsto (x, y)$,
the hyperbolic plane  $\H_{S_X}$ associated with $S_X$
is identified with
$$
H_X:=\set{(x, y)\in \R^2}{(x-1/2)^2+(\sqrt{k} y)^2<1/4}.
$$
The vector $f_{\phi}$  corresponds to the point  $(0, 0)$ of $\closure{H}_X$,
and the hyperplane $(z_{\phi})\sperp$ is given by $x=1/2$.
The Poincar\'e  disk model of the hyperbolic plane  $\H_{S_X}$ is given by  the mapping
\begin{equation}\label{eq:PoincareHX}
(x, y)\mapsto\frac{1-2\,x}{1+\sqrt{2}\,r}\;+\; \sqrt{-1} \frac{2\,\sqrt{k}\, y}{1+\sqrt{2}\,r}, \quad \textrm{where}\;\; r:=\sqrt{2x-2x^2-2ky^2}
\end{equation}
from $H_X$ to $\varDelta:=\shortset{z\in \C}{|z|< 1}$.
This map sends the boundary point corresponding to $f_{\phi}$ to $1\in\bar{\varDelta}$,
and the hyperplane $(z_{\phi})\sperp$ to the imaginary axis in $\varDelta$.
\begin{figure}
\begin{center}
\setlength{\unitlength}{.6truecm}
\begin{picture}(14, 13.3)
\put(1,7){\vector(1,0){12}}
\put(3,1.8){\vector(0,1){10.4}}
\put(7,3){\line(0,1){8}}
\multiput(3,11)(.5,0){8}{\line(1,0){.3}}
\multiput(3,3)(.5,0){8}{\line(1,0){.3}}
\put(3,7){\circle*{.2}}
\put(13.5,6.9){$x$}
\put(2.9, 12.8){$y$}
\put(2.0, 7.5){$f_{\phi}$}
\put(1.5, 11){$\frac{1}{2\sqrt{k}}$}
\put(1.0, 3){$-\frac{1}{2\sqrt{k}}$}
\put(7.2,7.4){$\frac{1}{2}$}
\put(7.2,9.2){$(z_{\phi})\sperp$}
\put(7,7){\circle{8}} 
\end{picture}
\end{center}
\vskip -1.2cm
\caption{$H_X$}
\end{figure}
\par
%\medskip
%
%
%
%\begin{table}
%{\tiny \input chams16 }
%\caption{Chambers for the case $-2k=-16$}\label{table:walls16}
%\end{table}
%
%
%
%
\begin{table}
\begin{center}
{\tiny %../Papers/AlgoAut/nowver/chams22.tex
$$
\renewcommand{\arraycolsep}{4pt}
\begin{array}{lll}
% cham 0 
\begin{array}{lll}
D_{0} &(1,-2,0) &\textrm{$(-2)$-wall}\\ 
   &(0,1,6) &D_{1}\\ 
   &(0,0,-1) &\cong D_{0}\;\textrm{by}\; h_{1}
\end{array}
&
% cham 1 
\begin{array}{lll}
D_{1} &(0,1,5) &D_{3}\\ 
   &(0,-1,-6) &\textrm{$D_{0}$}\\ 
   &(1,-1,6) &D_{2}
\end{array}
&
% cham 2 
\begin{array}{lll}
D_{2} &(1,-2,0) &\textrm{$(-2)$-wall}\\ 
   &(-1,1,-6) &\textrm{$D_{1}$}\\ 
   &(0,1,5) &D_{4}
\end{array}
\\ 
\hline
% cham 3 
\begin{array}{lll}
D_{3} &(0,-1,-5) &\textrm{$D_{1}$}\\ 
   &(1,-1,6) &D_{4}\\ 
   &(0,1,4) &D_{5}
\end{array}
&
% cham 4 
\begin{array}{lll}
D_{4} &(1,-1,5) &D_{6}\\ 
   &(0,1,4) &D_{7}\\ 
   &(-1,1,-6) &D_{3}\\ 
   &(0,-1,-5) &\textrm{$D_{2}$}
\end{array}
&
% cham 5 
\begin{array}{lll}
D_{5} &(0,-1,-4) &\textrm{$D_{3}$}\\ 
   &(1,-1,6) &D_{7}\\ 
   &(0,1,3) &D_{8}
\end{array}
\\ 
\hline
% cham 6 
\begin{array}{lll}
D_{6} &(1,-2,0) &\textrm{$(-2)$-wall}\\ 
   &(-1,1,-5) &\textrm{$D_{4}$}\\ 
   &(0,1,4) &D_{9}
\end{array}
&
% cham 7 
\begin{array}{lll}
D_{7} &(1,-1,5) &D_{9}\\ 
   &(0,1,3) &D_{10}\\ 
   &(-1,1,-6) &D_{5}\\ 
   &(0,-1,-4) &\textrm{$D_{4}$}
\end{array}
&
% cham 8 
\begin{array}{lll}
D_{8} &(0,-1,-3) &\textrm{$D_{5}$}\\ 
   &(1,-1,6) &D_{10}\\ 
   &(0,2,5) &D_{11}
\end{array}
\\ 
\hline
% cham 9 
\begin{array}{lll}
D_{9} &(0,-1,-4) &\textrm{$D_{6}$}\\ 
   &(-1,1,-5) &D_{7}\\ 
   &(1,0,8) &D_{12}
\end{array}
&
% cham 10 
\begin{array}{lll}
D_{10} &(0,-1,-3) &\textrm{$D_{7}$}\\ 
   &(1,1,11) &D_{13}\\ 
   &(-1,1,-6) &D_{8}
\end{array}
&
% cham 11 
\begin{array}{lll}
D_{11} &(0,-2,-5) &\textrm{$D_{8}$}\\ 
   &(1,1,11) &D_{14}\\ 
   &(0,1,2) &\cong D_{11}\;\textrm{by}\; h_{2}
\end{array}
\\ 
\hline
% cham 12 
\begin{array}{lll}
D_{12} &(1,-1,4) &D_{16}\\ 
   &(1,1,11) &D_{15}\\ 
   &(-1,0,-8) &\textrm{$D_{9}$}
\end{array}
&
% cham 13 
\begin{array}{lll}
D_{13} &(1,0,8) &D_{18}\\ 
   &(-1,-1,-11) &\textrm{$D_{10}$}\\ 
   &(0,2,5) &D_{17}
\end{array}
&
% cham 14 
\begin{array}{lll}
D_{14} &(0,1,2) &\cong D_{14}\;\textrm{by}\; h_{2}\\ 
   &(-1,-1,-11) &\textrm{$D_{11}$}\\ 
   &(1,-1,6) &D_{17}
\end{array}
\\ 
\hline
% cham 15 
\begin{array}{lll}
D_{15} &(1,-1,4) &D_{19}\\ 
   &(-1,-1,-11) &\textrm{$D_{12}$}\\ 
   &(0,1,3) &D_{20}
\end{array}
&
% cham 16 
\begin{array}{lll}
D_{16} &(1,-2,0) &\textrm{$(-2)$-wall}\\ 
   &(1,1,11) &D_{19}\\ 
   &(-1,1,-4) &\textrm{$D_{12}$}
\end{array}
&
% cham 17 
\begin{array}{lll}
D_{17} &(0,-2,-5) &\textrm{$D_{13}$}\\ 
   &(-1,1,-6) &D_{14}\\ 
   &(1,0,8) &D_{21}
\end{array}
\\ 
\hline
% cham 18 
\begin{array}{lll}
D_{18} &(1,-1,5) &D_{20}\\ 
   &(0,2,5) &D_{21}\\ 
   &(-1,0,-8) &\textrm{$D_{13}$}
\end{array}
&
% cham 19 
\begin{array}{lll}
D_{19} &(-1,1,-4) &\textrm{$D_{15}$}\\ 
   &(-1,-1,-11) &D_{16}\\ 
   &(2,-1,11) &D_{22}
\end{array}
&
% cham 20 
\begin{array}{lll}
D_{20} &(0,-1,-3) &\textrm{$D_{15}$}\\ 
   &(-1,1,-5) &D_{18}\\ 
   &(1,1,10) &D_{23}
\end{array}
\\ 
\hline
% cham 21 
\begin{array}{lll}
D_{21} &(0,-2,-5) &D_{18}\\ 
   &(-1,0,-8) &\textrm{$D_{17}$}\\ 
   &(1,1,10) &D_{24}
\end{array}
&
% cham 22 
\begin{array}{lll}
D_{22} &(1,-2,0) &\textrm{$(-2)$-wall}\\ 
   &(-2,1,-11) &\textrm{$D_{19}$}\\ 
   &(1,0,7) &D_{25}
\end{array}
&
% cham 23 
\begin{array}{lll}
D_{23} &(-1,-1,-10) &\textrm{$D_{20}$}\\ 
   &(2,1,17) &\cong D_{25}\;\textrm{by}\; h_{3}\\ 
   &(0,2,5) &D_{26}
\end{array}
\\ 
\hline
% cham 24 
\begin{array}{lll}
D_{24} &(1,-1,5) &D_{26}\\ 
   &(2,3,22) &\textrm{$(-2)$-wall}\\ 
   &(-1,-1,-10) &\textrm{$D_{21}$}
\end{array}
&
% cham 25 
\begin{array}{lll}
D_{25} &(1,-2,0) &\textrm{$(-2)$-wall}\\ 
   &(2,1,17) &\cong D_{23}\;\textrm{by}\; h_{3}\\ 
   &(-1,0,-7) &\textrm{$D_{22}$}
\end{array}
&
% cham 26 
\begin{array}{lll}
D_{26} &(0,-2,-5) &\textrm{$D_{23}$}\\ 
   &(2,3,22) &\textrm{$(-2)$-wall}\\ 
   &(-1,1,-5) &D_{24}
\end{array}
\\ 
\hline
\end{array}
$$
 }
\end{center}
\caption{Chambers for the case $-2k=-22$}\label{table:walls22}
\end{table}
\begin{figure}
\begin{center}
\includegraphics[bb=70 140 200 280]{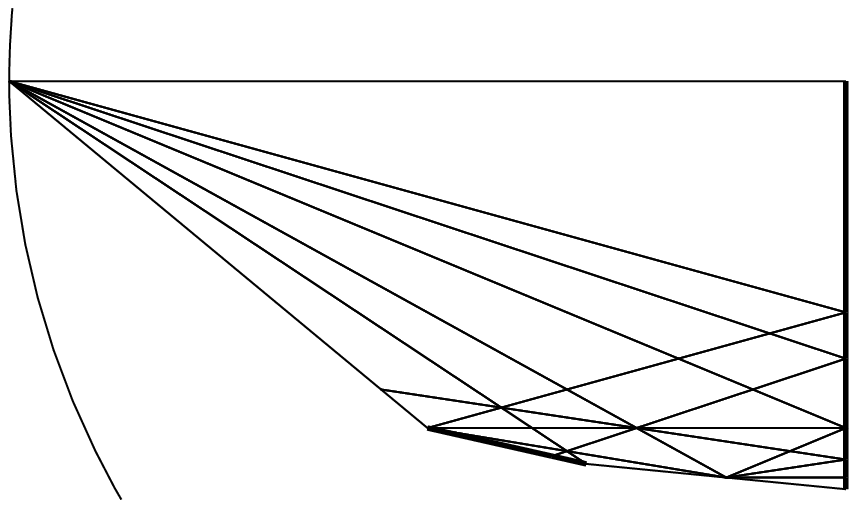}
\end{center}
\caption{Fundamental domain  for the case $-2k=-22$ in  $H_X$}\label{figure:poly22}
\end{figure}
\begin{example}
We present  the result for the case  $-2k=-22$.
Elements of $\L$ is written in terms of the basis $f_U, z_U, e_1, \dots, e_8$.
The primitive  embedding $i\colon S_X\inj \L$ is given by
$$
i(f_\phi)=f_U, \quad
i(z_\phi)=z_U, \quad
i(v_3)=[0, 0, 12, 8, 16, 24, 20, 16, 11, 6].
$$
As the elements of $\Chams$, 
we  obtained twenty-seven $\Roots_{\L|S_X}\sphyp$-chambers $D_0, \dots, D_{26}$.
It turns out that 
$\aut_{G_X}(D_i)=\{1\}$ holds for $i=0, \dots, 26$.
Their walls and adjacency relation are given in Table~\ref{table:walls22},
where $(a,b,c)$ is the wall defined by $a+bx+cy=0$ in $H_X$.
%A \emph{$(-2)$-wall} means a wall that  belongs to $\Roots_{S_X}\sphyp$,
%that is, a wall bounding $\Nc(X)$.
For example,
the chamber $D_{23}$ has three walls 
$\mu_1:-1-x-10y=0$, $\mu_2:2+x+17y=0$ and  $\mu_3:2x+5y=0$.
The chamber adjacent to $D_{23}$ across $\mu_1$ is $D_{20}$,
the chamber adjacent to $D_{23}$ across $\mu_3$ is $D_{26}$,
and 
the chamber $D\sprime$ adjacent to $D_{23}$ across $\mu_2$ is isomorphic to  $D_{25}$ via  
$$
h_{3}:= 
\left[
\begin{array}{ccc}
20 & 9& -3\\
7& 2& -1\\
154& 66& -23
\end{array}
\right].
$$
Hence $\Image \varphi_X$ is generated by 
the three elements $h_1$, $h_2$, $h_3$.
The shape of the union $F:=F_{\Nc(X)}$ of these $D_i$ in $H_X$ 
is  given in Figure~\ref{figure:poly22}.
Since the union $F$ of these $D_i$ has two $(-2)$-walls, 
which is depicted by thick lines in Figure~\ref{figure:poly22},
the set of smooth rational curves on $X$ is decomposed into at most two orbits under the action of $\Aut(X)$. 
The left-hand side of Figure~\ref{figure:Poincare} shows the chambers
$$
F,\;   F^{h_1},\;   F^{h_2},\;   F^{h_3},\;   F^{h_1 h_2},\;   F^{h_2 h_1},\;   F^{h_1 h_3},\;   F^{h_3 h_1},\;   
F^{h_3 h_2},\;    F^{h_2 h_3}
$$
on the Poincar\'e disk model~\eqref{eq:PoincareHX} of $\H_{S_X}$.
The $(-2)$-walls are drawn by thick lines.
\end{example}
\begin{figure}
\begin{center}
\includegraphics[bb=70 140 200 280]{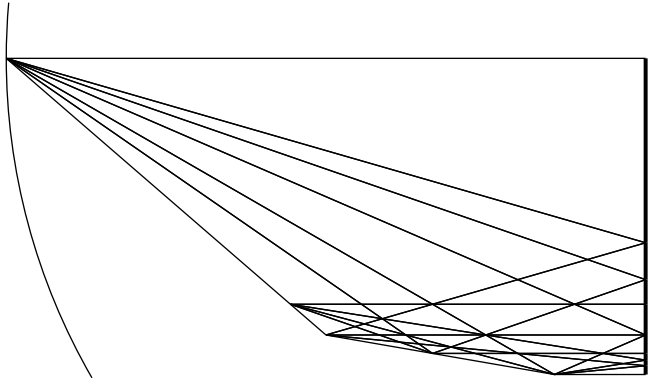}
\end{center}
\caption{Fundamental domain  for the case $-2k=-24$ in $H_X$}\label{figure:poly24}
\end{figure}
\begin{figure}
\begin{center}
\includegraphics[bb=40 0 200 320]{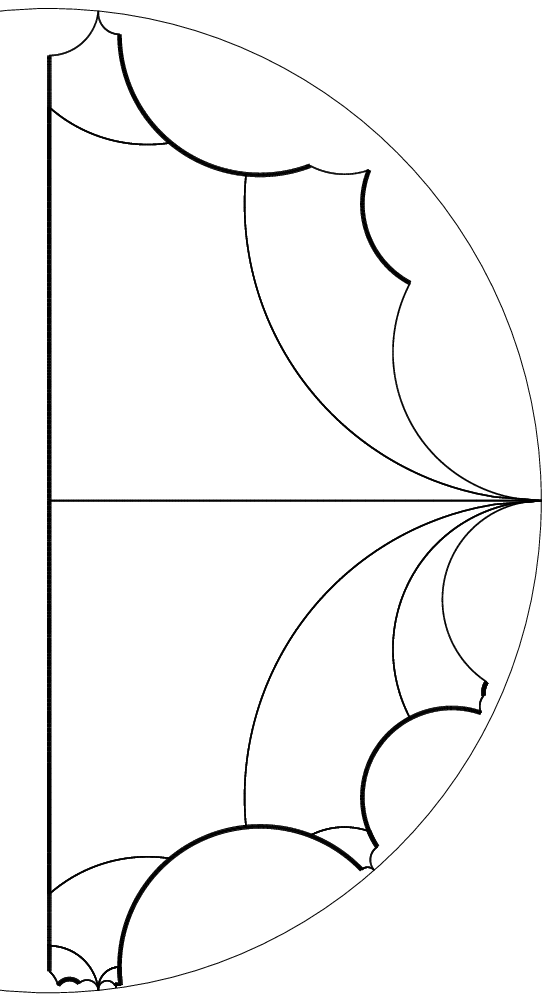}
\includegraphics[bb=-10 0 100 320]{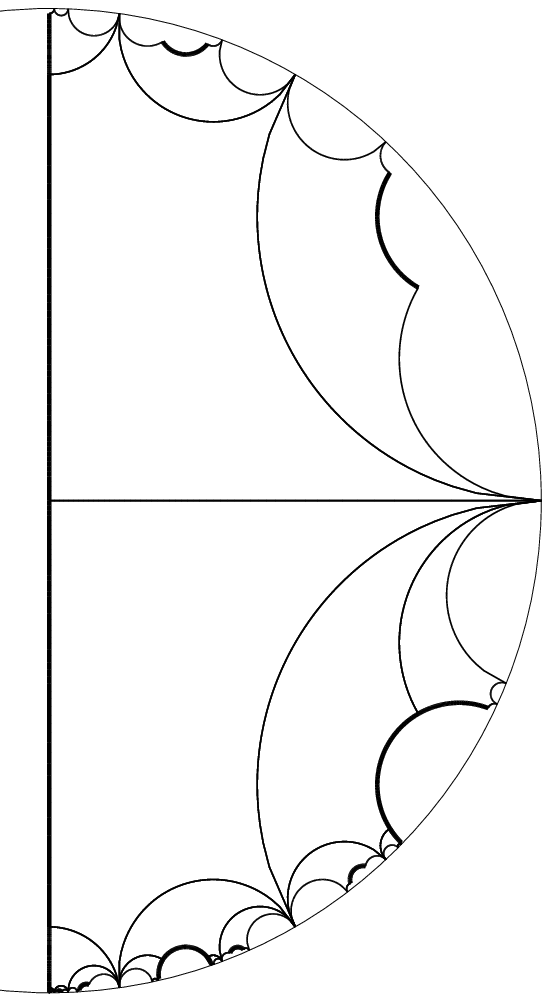}
\end{center}
\caption{Tessellation of $\Nc(X)$   for the cases $-2k=-22$ and $-2k=-24$}\label{figure:Poincare}
\end{figure}
\begin{remark}\label{rem:ExampleIntro}
Example~\ref{example:24} in Introduction is the case $-2k=-24$.
The set $\Chams$ contains $46$ chambers, and 
their union $F:=F_{\Nc(X)}$ is  given in Figure~\ref{figure:poly24}.
Each chamber $D_i$ satisfies $\aut_{G_X}(D_i)=\{1\}$.
Since $F_{\Nc(X)}$ has only one  $(-2)$-wall, $\Aut(X)$ acts on 
the set of smooth rational curves on $X$ transitively. 
Observe that $F:=F_{\Nc(X)}$ has four non-$(-2)$-walls,
which correspond to the four generators $h_1, \dots, h_4$ of $\Aut(X)$ given in Example~\ref{example:24}.
The right-hand side of Figure~\ref{figure:Poincare} shows the chambers
$$
F,\;   F^{h_1},\;   F^{h_2},\;   F^{h_3},\;  F^{h_4},\;   F^{h_i h_j},\; \; (1\le i, j\le 4, i\ne j)
$$
on the Poincar\'e disk model~\eqref{eq:PoincareHX} of $\H_{S_X}$.
\end{remark}
\begin{remark}\label{rem:moreexamples}
We have made the same calculation for $-2k=-4,-6,  \dots, -30$.
The results are presented in the author's web page~\cite{compdataweb}.
\end{remark}
\section{Singular $K3$ surfaces}\label{sec:singK3}
Recall that a $K3$ surface defined over $\C$ is singular if its Picard number attains the possible maximum $20$.
We demonstrate Algorithm~\ref{algo:main} on 
singular $K3$ surfaces $X$ 
such that their transcendental lattices $T_X$ satisfy $\disc T_X\le 16$.
There exist, up to isomorphisms,  exactly eleven singular $K3$ surfaces with $\disc T_X\le 16$.
Table~\ref{table:singK3} shows
Gram matrices 
$$
\left[
\begin{array}{cc}
a & b \\
b & c
\end{array}
\right]
$$
of the transcendental lattices
of these $K3$ surfaces.
We embed $S_X$ into an even unimodular hyperbolic lattice $\L$ of rank $26$.
By the \emph{root-$R$ condition},
we mean 
that there should exist a negative-definite root lattice $R$ of rank $6$ that satisfies
$(A_{R}, q_{R})\cong (A_{T_X}, q_{T_X})$.
If $T_X$ satisfies the root-$R$ condition, 
then there exists a primitive embedding $S_X\inj \L$ such that 
the orthogonal complement is isomorphic to a root lattice $R$,
and hence Borcherds' method terminates quickly~(see Remark~\ref{rem:Rroots}).
The column root-$R$ of Table~\ref{table:singK3} indicates the existence of such a root lattice $R$,
and its $ADE$-type if it exists.
\par
The singular $K3$ surfaces of Nos.~{1, 2, 3, 7 and 11} have been studied previously.
Therefore we treat the other six singular $K3$ surfaces Nos.~{4, 5, 6, 8, 9 and 10}.
%three of which are easy and the rest are difficult.
%
%
\begin{table}
\renewcommand{\arraystretch}{1.2}
\begin{center}
\begin{tabular}{c|ccccc|c}
No. & $\disc T_X$ & $[a,b,c]$ & root-$R$  & $|\Chams|$  & $|C_X|$&   \\
\hline
1 & $3$ & $[2,1,2]$ & $E_6$ & $1$ &$6$& \textrm{Vinberg~\cite{MR719348}}\\
2 & $4$ & $[2,0,2]$ & $D_6$ & $1$ &$4$& \textrm{Vinberg~\cite{MR719348}}\\
3 & $7$ & $[2,1,4]$ & $A_6$ & $1$ &$2$& \textrm{Ujikawa~\cite{MR3113614}}\\
4 & $8$ & $[2,0,4]$ & $D_5+A_1$ &$1$& $2$ & \textrm{Section~\ref{subsec:easyexamples}} \\
5 & $11$ & $[2,1,6]$ & none & $1098$ &$2$& \textrm{Section~\ref{subsec:216}} \\
6 & $12$ & $[2,0,6]$ & $A_5+A_1$ & $1$ &$2$& \textrm{Section~\ref{subsec:easyexamples}} \\
7 & $12$ & $[4,2,4]$ & $D_4+A_2$ & $1$ &$6$& \textrm{Keum and Kondo~\cite{MR1806732}}\\
8 & $15$ & $[2,1,8]$ & $A_4+A_2$ & $1$ &$2$& \textrm{Section~\ref{subsec:easyexamples}}\\
9 & $15$ & $[4,1,4]$ & none & $2051$ &$2$& \textrm{Section~\ref{subsec:414}} \\
10 & $16$ & $[2,0,8]$ & none & $ 4539$ &$2$& \textrm{Section~\ref{subsec:208}} \\
11 & $16$ & $[4,0,4]$ & $2A_3$ & $1$ &$4$& \textrm{Keum and Kondo~\cite{MR1806732}}\\
\end{tabular}
\end{center}
\vskip .4cm
\caption{Transcendental lattices of small discriminants}\label{table:singK3}
\end{table} 
\subsection{The period and ample classes of $X$}
The period $\omega_X$ is given in $T_X\tensor\C$ as a vector $(1, \alpha)$,
where $\alpha$ is a root of 
$$
a+2bx+cx^2=0.
$$
The choice of a root of this quadratic equation does not matter,
because,
in the eleven cases in Table~\ref{table:singK3},
the orientation reversing of $T_X$ yields an isomorphic oriented lattice;
that is,  the $\GL_2(\Z)$-equivalence class of $T_X$ 
contains only one $\SL_2(\Z)$-equivalence class~(see~\cite{MR1820211}).
Since $\OG(T_X)$ is finite,
we can calculate the subgroup $C_X$ defined by~\eqref{eq:CX}.
It turns out that,
in the six cases Nos.~{4, 5, 6, 8, 9 and 10},
we have $C_X=\{\pm 1\}$.
Therefore the conditions~$\propVG$ and $\assumpthree$
for $G_X$ and $R$ are satisfied (see Sections~\ref{subsec:req1} and~\ref{subsec:req2}).
%
%\begin{remark}
%In the case No.~1, we have $|C_X|=6$,
%and in the case No.~2, we have $|C_X|=4$.
%\end{remark}
%
%
\par
In the following, we use the general theory of elliptic surfaces,
for which we refer to~\cite{MR1081832} or~\cite{MR1813537}.
Shioda and Inose~\cite{MR0441982} showed that
every singular $K3$ surface $X$ has a Jacobian fibration $\phi\colon X\to\P^1$
with two singular fibers $\phi\inv(p)$ and $\phi\inv(p\sprime)$ of type $II^*$.
Hence $S_X$ contains an even unimodular hyperbolic lattice $\L_{18}(\phi):=U_{\phi}\oplus E_8\oplus E_8$
of rank $18$ as a sublattice, where $U_{\phi}$ is spanned by the class $f_{\phi}$ 
of a fiber of $\phi$ and the class $z_{\phi}$  of the zero section of $\phi$
(and hence the Gram matrix of $U_{\phi}$ with respect the basis $f_{\phi}, z_{\phi}$ is~\eqref{eq:nonstandardU}),
and each copy of $E_8$  is spanned by the classes of  irreducible components of 
a singular fiber of type $II^*$.
Since $\L_{18}(\phi)$ is unimodular,
there exists a negative-definite lattice $T_X^{-}$ of rank $2$ such that
$$
S_X=\L_{18}(\phi)\oplus T_X^{-}.
$$
We have $(A_{T_X^{-}}, q_{T_X^{-}})\cong(A_{T_X}, -q_{T_X})$.
Since the Jacobian fibration $\phi$ is not unique, 
the isomorphism class of $T_X^{-}$ is not unique 
in general (see~\cite{MR2346573}~and~\cite{MR2452829}). %2014May31
However, for the eleven cases in Table~\ref{table:singK3},
$T_X^{-}$ is uniquely determined  by the condition $(A_{T_X^{-}}, q_{T_X^{-}})\cong(A_{T_X}, -q_{T_X})$ 
 and is isomorphic to $(-1)T_X$.
%Therefore we obtain an explicit basis of $S_X$ and the Gram matrix with respect to this basis. 
%
\par
We search for  rational ample classes $a\in S_X\tensor\Q$.
Let $e_1, \dots, e_8$ (resp.~$e_1\sprime, \dots, e_8\sprime$)
denote the classes of the irreducible components of $\phi\inv(p)$ (resp.~of $\phi\inv(p\sprime)$)
that are disjoint from the zero section
and whose dual graph is
$$
\def\ha{40}
\def\hav{37}
\def\hd{25}
\def\hdv{22}
\def\he{10}
\def\hev{7}
\setlength{\unitlength}{1.5mm}
\centerline{
{\small
\begin{picture}(105, 9)(-20, 7)
\put(22, 16){\circle{1}}
\put(23.5, 15.5){$e\sb 1$}
\put(22, 10.5){\line(0,1){5}}
\put(9.5, \hev){$e\sb 2$}
\put(15.5, \hev){$e\sb 3$}
\put(21.5, \hev){$e\sb 4$}
\put(27.5, \hev){$e\sb 5$}
\put(33.5, \hev){$e\sb 6$}
\put(39.5, \hev){$e\sb 7$}
\put(45.5, \hev){$e\sb {8}$}
\put(51.5, \hev){$e\sb {0}$}
\put(57.5, \hev){$z_\phi$}
%\put(50, \hev){$f-\theta$}
\put(10, \he){\circle{1}}
\put(16, \he){\circle{1}}
\put(22, \he){\circle{1}}
\put(28, \he){\circle{1}}
\put(34, \he){\circle{1}}
\put(40, \he){\circle{1}}
\put(46, \he){\circle{1}}
\put(52, \he){\circle{1}}
\put(58, \he){\circle{1}}
\put(10.5, \he){\line(5, 0){5}}
\put(16.5, \he){\line(5, 0){5}}
\put(22.5, \he){\line(5, 0){5}}
\put(28.5, \he){\line(5, 0){5}}
\put(34.5, \he){\line(5, 0){5}}
\put(40.5, \he){\line(5, 0){5}}
\put(46.5, \he){\line(5, 0){5}}
\put(52.5, \he){\line(5, 0){5}}
\put(60, \hev){,}
%\put(46.5, \he){\line(5, 0){5}}
%\put(66, \hev){$f:=f_U$}
\end{picture}
}
}
$$
where we denote by $e_0$
(resp.~$e_0\sprime$)
the class of the irreducible component of 
 $\phi\inv(p)$ (resp.~of $\phi\inv(p\sprime)$)
that intersects the zero section.
We have
$$
e_0=f_{\phi}- (3e_1+ 2e_2+ 4e_3 +6e_4 +5e_5 +4e_6 +3 e_7 +2e_8),
$$
and a similar equality for $e_0\sprime$.
\par
Let $e_1\spprime$ and $e_2\spprime$ be a basis of $T_X^-$ such that
$$
\intM{e_1\spprime, e_1\spprime}{}=-a, \;\; \intM{e_1\spprime, e_2\spprime}{}=-b,\;\; \intM{e_2\spprime, e_2\spprime}{}=-c.
$$
Then the set $\Roots_{T_X\sp-}=\shortset{v\in T_X^-}{v^2=-2}$ is equal to
$$
\begin{cases}
\emptyset & \textrm{in No. 9}, \\
\{e_1\spprime, -e_1\spprime\} & \textrm{in Nos.~4, 5, 6, 8, 10}.
\end{cases}
$$
Therefore the set $\shortset{t\in \P^1}{\textrm{$\phi\inv(t)$ is reducible}}$  is equal to 
$$
\begin{cases}
\{ p, p\sprime\} & \textrm{in No. 9}, \\
\{ p, p\sprime, q\}  & \textrm{in Nos.~4, 5, 6, 8, 10, where $q\in \P^1\setminus \{p, p\sprime\}$}.
\end{cases}
$$
Suppose that we are in the case of  Nos.~4, 5, 6, 8 or 10.
Then the reducible fiber $\phi\inv (q)$ is either of type $I_2$ or $III$.
Let $C_0\spprime$ and $C_1\spprime$ be the irreducible components of $\phi\inv (q)$
such that $C_0\spprime$ intersects the zero section.
Changing $e_1\spprime$ and $e_2\spprime$ to $-e_1\spprime$ and $-e_2\spprime$ if necessary,
we can assume that $e_1\spprime$ is the class of $C_1\spprime$,
and hence the class $e_0\spprime$  of $C_0\spprime$ is equal to $f_{\phi}-e_1\spprime$.
\par
We put
$$
\BBB:=\set{[C]}{\textrm{$C$ is a smooth rational curve on $X$ such that $\intM{C, f_{\phi}}{}=0$}}.
$$
Then we have
$$
\BBB=
\begin{cases} 
 \{e_0, e_1, \dots, e_8, e_0\sprime, e_1\sprime, \dots, e_8\sprime\} & \textrm{in No. 9}, \\
 \{e_0, e_1, \dots, e_8, e_0\sprime, e_1\sprime, \dots, e_8\sprime, e_0\spprime, e_1\spprime\} & \textrm{in Nos.~4, 5, 6, 8, 10}.
\end{cases}
$$
Let $e_1\dual, \dots, e_8\dual$ (resp.~$e_1^{\prime\vee}, \dots, e_8^{\prime\vee}$)
be the  basis of $E_8$ dual to the basis
$e_1, \dots, e_8$ (resp,~$e_1\sprime, \dots, e_8\sprime$).
We have 
$$
u_p:=e_1\dual+ \cdots+ e_8\dual=-68 e_1-46 e_2- 91 e_3 -135 e_4 - 110 e_5 - 84 e_6- 57 e_7 - 29 e_8,
$$
and the similar formula for $u_{p\sprime}:=e_1^{\prime\vee}+ \cdots+ e_8^{\prime\vee}$.
Let $e_1^{\prime\prime\vee}$ and  $e_2^{\prime\prime\vee}$ be the basis of $T_X^{-\vee}$ dual to the basis $e_1\spprime, e_2\spprime$.
Then the  vector
$$
u_0:=
\begin{cases}
30 z_{\phi} +u_p +u_{p\sprime} & \textrm{in No. 9}, \\
30 z_{\phi} +u_p +u_{p\sprime} + e_1^{\prime\prime\vee} & \textrm{in Nos.~4, 5, 6, 8, 10}, 
\end{cases}
$$
satisfies
$$
\intM{u_0, v}{}>0\quad\textrm{for any}\quad v\in \BBB.
$$
(The coefficient $30$ of $z_{\phi}$ in $u_0$ is determined by  $\intM{e_0, u_p}{}=\intM{e_0\sprime, u_{p\sprime}}{}=-29$.)
Consider the projection
$\pi_{S_X}\colon  \PPP(X)\to \H_{S_X}$
to the $19$-dimensional hyperbolic space $\H_{S_X}$.
The point $b:=\pi_{S_X}(f_{\phi})$ is a rational boundary point of $\pi_{S_X}(\Nc(X))$.
By Corollary~\ref{cor:HS},
if we choose a sufficiently small closed horoball $\HB_{b}$ with the base $b$,
then $\HB_{b}\cap \pi_{S_X}(\Nc(X))$
is bounded in $\HB_b$ by the inequalities
$$
\intM{x, v}{}\ge 0
\quad\textrm{for any}\quad v\in \BBB.
$$
Therefore, if $A$ is sufficiently large,
then 
$$
a:=A f_{\phi}+u_0\;\;\in\;\; S_X\tensor\Q
$$
is contained in the interior of $\Nc(X)$.
We use this rational ample class $a$ in search for the primitive embedding $S_X\inj \L$
such that the initial $\RtsLS\sphyp$-chamber $D_0$ is contained in $N(X)$.
\par
We use $U\oplus E_8\oplus E_8\oplus E_8$ as $\L$
and the Weyl vector $w_E$ given in Remark~\ref{rem:UE83}.
Let $\DDD_0$ be the $\Roots_{\L}\sphyp$-chamber corresponding to $w_E$.
We have a natural isomorphism from $\L_{18}(\phi)$ to the sublattice $U\oplus E_8\oplus E_8$
of $\L$.
Hence, in order to find a primitive embedding $i_0\colon S_X\inj \L$,
it is enough to find a primitive embedding 
$T_X^{-}\inj E_8$;
that is, to find vectors $u_1$, $u_2$ of $E_8$ satisfying
$$
u_1^2=-a,\quad u_2^2=-c, \quad \intM{u_1, u_2}{E_8}=-b, 
$$
and generating 
a  primitive sublattice of rank $2$ in $E_8$.
From this embedding $i_0$ and using the method described in Section~\ref{subsec:req3}
with the ample class $a=Af_{\phi}+u_0$ above,
we find a primitive embedding $S_X\inj \L$ 
such that
$D_0:=\DDD_0\cap \PPP(X)$ is an $\RtsLS\sphyp$-chamber contained in $\Nc(X)$.
In the examples below,
we  confirmed that the vector
$$
h_E:=w_{E, S}\in S_X\dual
$$
is in fact in the interior of $D_0$ by Criterion~\ref{criterion:Snondeg}
and hence $\DDD_0$ is $S$-nondegenerate.
Moreover,
we confirmed 
that the interior point  $Af_{\phi}+u_0$  of $N(X)$ is contained in $D_0$
by showing
that
$$
\intM{f_{\phi}, v}{}>0\;\;\textrm{or} \;\; (\intM{f_{\phi}, v}{}=0\;\textrm{and}\; \intM{u_0, v}{}\ge 0)\quad 
\textrm{for any}\;\; v\in \Delta_{S_X\dual} (D_0), 
$$
and hence we can conclude that $D_0$ is contained in $N(X)$.
In particular, $h_E$ is also an ample class of $X$,
and hence we can consider the automorphism group
$$
\Aut(X, h_E):=\set{g\in \Aut(X)}{g^*(h_E)=h_E}
$$
of the polarized $K3$ surface $(X, h_E)$.
Since $\varphi_X$ is injective in our six cases,
we have 
$$
\Aut(X, h_E)\cong \aut_{G_X}(D_0).
$$
\par
If $T_X$ satisfies the root-$R$ condition,
the orthogonal complement $R$ of $S_X$ in $\L$ satisfies $\assumptwo$,
because $\Roots_R\ne \emptyset$.
If $T_X$ does not satisfy the root-$R$ condition,
we present below a Gram matrix of $R$ explicitly,
from which  
we immediately see that $\Roots_R\ne \emptyset$  and hence $R$ satisfies  $\assumptwo$.
%The orders of the finite group $\OG(R)$ and $\eta_R(\OG(R))$
%are given in Table~\ref{table:order}.
%In our six cases, we have $|C_X|=2$ and hence $C_X=\{\pm 1\}$.
%Therefore the liftability condition~$\assumpthree$ for $G_X$ is satisfied.
%We confirm that, 
%under an isomorphism $\delta_T\colon (A_{T_X}, q_{T_X})\isom (A_{S_X}, -q_{S_X})$
%and the isomorphism $\delta_R\colon (A_{R}, q_{R})\isom (A_{S_X}, -q_{S_X})$
%induced by $\L\subset S_X\dual \oplus R\dual$,
%the subgroup $\delta_{T*} (\eta_{T_X} (C_X))=\eta_{S_X} (G_X)$
%of $\OG(q_{S_X})$ is contained in $\delta_{R*} (\eta_R(\OG(R)))$,
%and hence the liftability condition~$\assumpthree$ is satisfied.
%
\subsection{The cases where the root-$R$ condition is satisfied}\label{subsec:easyexamples}
In these cases (Nos. 4,6 and 8), we have $\Chams=\{D_0\}$,
and hence the description of $\Aut(X)$ is simple.
We  have
$$
h_E^2=
\begin{cases}
61/2 & \textrm{in No. 4,}\\
18 & \textrm{in No. 6,}\\
12 & \textrm{in No. 8,}
\end{cases}
$$
and
$$
%\qquad 
|\Aut(X, h_E)|=|\aut_{G_X}(D_0)|=
\begin{cases}
48 & \textrm{in No. 4,}\\
144 & \textrm{in No. 6,}\\
720 & \textrm{in No. 8.}
\end{cases}
$$
Table~\ref{table:orbitdecompD0} presents the orbit decomposition of the set of walls
of $D_0$ by the action of $\Aut(X, h_E)\cong \aut_{G_X}(D_0)$.
The column $|o|$ indicates the cardinality of the orbit $o$.
Each wall of $D_0$ is uniquely written as $(v)\sperp$,
where $v$ is a primitive vector of $S_X\dual$ satisfying $\intM{v, h_E}{}>0$.
In Table~\ref{table:orbitdecompD0},
we also present the values
$$
n_v:=v^2, \quad a_v:=\intM{v, h_E}{}
$$
for each orbit.
Note that a wall $(v)\sperp$ of $D_0$ with $v$ primitive in $S_X\dual$ and $\intM{v, h_E}{}>0$
is a $(-2)$-wall if and only if there exists a positive integer $\alpha$
such that $\alpha^2\, n_v=-2$ and $\alpha v\in S_X$.
The orbits that consist of walls satisfying this condition are marked by $*$. % in Table~\ref{table:orbitdecompD0}.
\begin{table}
\begin{center}
\parbox{4cm}{
No. 4
\par\medskip
$
\begin{array}{cccc}
 |o| & n_v & a_v &\\
 \hline 
 6 & -2 & 1 & *\\ 
 12 & -2 & 1 & *\\ 
 8 & -2 & 1 & *\\ 
 3 & -3/2 & 3/2 & \\ 
 6 & -1 & 5 &\\ 
 4 & -1 & 5 &\\ 
 24 & -3/4 & 6 &\\ 
 8 & -3/4 & 6 &\\ 
 2 & -1/2 & 11/2 & *\\ 
 8 & -1/4 & 13/2 &\\
 \end{array}
$
}
\parbox{4cm}{
No. 6
\par\medskip
$
\begin{array}{cccc}
 |o| & n_v & a_v &\\
 \hline 
 12 & -2 & 1 &*\\ 
 18 & -2 & 1 &*\\ 
 4 & -3/2 & 3/2 &\\ 
 24 & -7/6 & 7/2 &\\ 
 6 & -2/3 & 4 &\\ 
 24 & -2/3 & 5 &\\ 
 36 & -2/3 & 5 &\\ 
 12 & -1/2 & 11/2 &*\\ 
 36 & -1/2 & 11/2 &*\\ 
 24 & -1/6 & 11/2 &\\
\end{array}
$
}
\parbox{4cm}{
No. 8
\par\medskip
$
\begin{array}{cccc}
 |o| & n_v & a_v &\\
 \hline 
 36 & -2 & 1 &*\\ 
 12 & -4/3 & 2 &\\ 
 40 & -6/5 & 3 &\\ 
 90 & -4/5 & 4 &\\ 
 30 & -8/15 & 4 &\\ 
 30 & -8/15 & 4 &\\ 
 120 & -2/15 & 5 &\\ 
 120 & -2/15 & 5 &\\
 \\
 \\
 \end{array}
$
}
\end{center}
\vskip .3cm
\caption{Orbit decomposition of the walls of $D_0$}\label{table:orbitdecompD0} 
\end{table}
From Table~\ref{table:orbitdecompD0},
we see the following.
\par
In No.~4,
the automorphism group $\Aut(X)$ is generated by the finite group $\Aut(X, h_E)$ of order $48$ and six extra automorphisms
corresponding to the six walls of $D_0$ that are not $(-2)$-walls.
Exactly four  orbits 
consist of $(-2)$-walls, and hence
the number of orbits of $\Aut(X)$ on the set of smooth rational curves on $X$ is at most $4$.
\par
In No.~6,
$\Aut(X)$ is generated by the finite group $\Aut(X, h_E)$ of order $144$ and six extra automorphisms.
The number of orbits of $\Aut(X)$ on the set of smooth rational curves is at most $4$.
\par
In No.~8,
$\Aut(X)$ is generated by the finite group $\Aut(X, h_E)$ of order $720$ and seven extra automorphisms,
and 
 $\Aut(X)$ acts on the set of smooth rational curves transitively.
\begin{remark}
In~\cite{ShimadaAutSingK3}, 
geometric realizations of generators of $\Aut(X)$ for No.~8 will be  given.
We also show that the finite group   $\Aut(X, h_E)$  of order $720$ for No.~8  is isomorphic to $\PGL_2(\F_9)$.
\end{remark}
\subsection{The case $[a,b,c]=[2,1,6]$ (No.~5)}\label{subsec:216}
The lattice $R$ has a Gram matrix
$$
\left[ \begin {array}{cccccc} 
-2&1&0&0&-1&0\\\noalign{\medskip}
1&-2&1&0&0&1\\\noalign{\medskip}
0&1&-2&1&0&0\\\noalign{\medskip}
0&0&1&-2&1&0\\\noalign{\medskip}
-1&0&0&1&-4&0\\\noalign{\medskip}
0&1&0&0&0&-2
\end {array} \right].
$$
The number of $G_X$-equivalence classes of $\RtsLS\sphyp$-chambers is $1098$,
and the maximum of the level (see~Section~\ref{sec:main} for the definition) is $13$.
The results are presented in Example~\ref{example:singK3disc11} in Introduction.
\subsection{The case $[a,b,c]=[4,1,4]$ (No.~9)}\label{subsec:414}
The lattice $R$ has a Gram matrix
$$
\left[ \begin {array}{cccccc} 
-2&1&0&0&0&0\\\noalign{\medskip}
1&-2&1&0&0&-1\\\noalign{\medskip}
0&1&-2&1&0&0\\\noalign{\medskip}
0&0&1&-2&1&0\\\noalign{\medskip}
0&0&0&1&-2&1\\\noalign{\medskip} 
0&-1&0&0&1&-4
\end {array} \right].
$$
The number of $G_X$-equivalence classes of $\RtsLS\sphyp$-chambers is $2051$,
and the maximum of the level is $16$.
We find that 
$\Image \varphi_X$ is generated by $1098$ elements of $\OG(S_X)$.
Moreover the number of orbits of the action of $\Aut(X)$ on the set of smooth rational curves is at most $154$.
\subsection{The case $[a,b,c]=[2,0,8]$  (No.~10)}\label{subsec:208}
The lattice $R$ has a Gram matrix
$$
\left[ \begin {array}{cccccc} -2&1&0&0&1&0\\\noalign{\medskip}1&-2&1&0
&0&1\\\noalign{\medskip}0&1&-2&1&0&0\\\noalign{\medskip}0&0&1&-2&0&-1
\\\noalign{\medskip}1&0&0&0&-2&0\\\noalign{\medskip}0&1&0&-1&0&-4
\end {array} \right].
$$
The number of $G_X$-equivalence classes of $\RtsLS\sphyp$-chambers is $4539$,
and the maximum of the level is $17$.
We find that 
$\Image \varphi_X$ is generated by $3308$ elements of $\OG(S_X)$.
Moreover the number of orbits of the action of $\Aut(X)$ on the set of smooth rational curves is at most $705$.
\begin{remark}\label{rem:compdata}
The detailed computational data of the above  three examples are 
presented in the author's web page~\cite{compdataweb}.
\end{remark}
\begin{remark}\label{rem:complexFQ}
The complex  quartic surface $X$
defined by
$w^4+x^4+y^4+z^4=0$
in $\P^3$ is a singular $K3$ surface
with 
$$
T_X=\left[
\begin{array}{cc}
8 & 0 \\ 0 & 8
\end{array}
\right].
$$
We  embed $S_X$ into $\L$ 
primitively 
%with the orthogonal complement $R$ being 
%$$
%\left[ \begin {array}{cccccc} 
%-2&0&1&0&-2&-2\\
%0&-2&0&1&-2&2\\
%1&0&-2&0&2&2\\
%0&1&0&-2&2&-2\\
%-2&-2&2&2&-8&0\\
%-2&2&2&-2&0&-8
%\end {array} \right]
%$$
in such a way that 
the Weyl vector
$w_{0, {S_X}}$ 
 of the initial $\RtsLS\sphyp$-chamber $D_0$ is in the interior of $D_0$ 
and is equal to the class of a hyperplane section of $X$ in $\P^3$.
%, so that
%$$
%\aut_{G_X}(D_0)=\Aut(X, h_F)=\set{g\in \PGL_4(\C)}{g(X)=X}
%$$
% is a finite group  of order $1536$.
We  found that 
the number of $G_X$-congruence classes of $\RtsLS\sphyp$-chambers is \emph{at least} $10000$.
\end{remark}

\par 
\medskip
{\bf Acknowledgement.}
The author is deeply grateful to Professors Daniel Allcock,
Toshiyuki Katsura, Shigeyuki Kondo and Shigeru Mukai  for their interests in this work
and many discussions.
Thanks are also due to the referees of the first version of this paper for many suggestions.

\bibliographystyle{plain}

\def\cftil#1{\ifmmode\setbox7\hbox{$\accent"5E#1$}\else
  \setbox7\hbox{\accent"5E#1}\penalty 10000\relax\fi\raise 1\ht7
  \hbox{\lower1.15ex\hbox to 1\wd7{\hss\accent"7E\hss}}\penalty 10000
  \hskip-1\wd7\penalty 10000\box7} \def\cprime{$'$} \def\cprime{$'$}
  \def\cprime{$'$} \def\cprime{$'$}

\end{document}